\documentclass[12pt]{amsart}
\usepackage{graphicx}

\usepackage{pslatex}
\usepackage{amsmath,amsfonts}
\usepackage{rotating}
\usepackage{amssymb}
\usepackage{verbatim}
\usepackage{rotating}
\usepackage{xcolor}
\usepackage{mathrsfs}
\usepackage[linkcolor=blue,citecolor=blue,colorlinks=true]{hyperref}

\newtheorem{theorem}{Theorem}[section]
\newtheorem{lemma}[theorem]{Lemma}
\newtheorem{proposition}[theorem]{Proposition}
\newtheorem{corollary}[theorem]{Corollary}

\theoremstyle{definition}
\newtheorem{definition}[theorem]{Definition}
\newtheorem{example}[theorem]{Example}

\theoremstyle{remark}
\newtheorem{remark}[theorem]{Remark}

\numberwithin{equation}{section}

\def\R{\mathbb R}
\def\C{\mathbb C}
\def\Z{\mathbb Z}
\def\N{\mathbb N}
\def\S{\mathscr S}

\def\supp{\text{supp}}
\def\({\left(}
\def\){\right)}
\def\[{\left[}
\def\]{\right]}
\def\<{\left<}
\def\>{\right>}
\def\less{\lesssim}

\allowdisplaybreaks

\begin{document}

\title{Analysis of Hyper-Singular, Fractional, and Order-Zero Singular Integral Operators}

\author{Lucas Chaffee}
\address{ L. Chaffee\\
   Mathematics Department\\
   Western Washington University\\
   Bellingham, WA, USA }
\email{ lucas.chaffee@wwu.edu }
\author{Jarod Hart}
\address{ J. Hart\\
   Higuchi Biosciences Center\\
   University of Kansas\\
   Lawrence, KS, USA
   }
\email{jvhart@ku.edu}
\author{Lucas Oliveira}
\address{ L. Oliveira\\
   Mathematics Department\\
   Universidade Federal do Rio Grande do Sul\\
   Porto Alegre, Rio Grande do Sul, Brazil }
\email{lucas.oliveira@ufrgs.br}

\thanks{  }



\date{\today}

\dedicatory{ }

\maketitle

\begin{abstract}
In this article, we conduct a study of integral operators defined in terms of non-convolution type kernels with singularities of various degrees.  The operators that fall within our scope of research include fractional integrals, fractional derivatives, pseudodifferential operators, Calder\'on-Zygmund operators, and many others.  The main results of this article are built around the notion of an operator calculus that connects operators with different kernel singularities via vanishing moment conditions and composition with fractional derivative operators.  We also provide several boundedness results on weighted and unweighted distribution spaces, including homogeneous Sobolev, Besov, and Triebel-Lizorkin spaces, that are necessary and sufficient for the operator's vanishing moment properties, as well as certain behaviors for the operator under composition with fractional derivative and integral operators.  As applications, we prove $T1$ type theorems for singular integral operators with different singularities, boundedness results for pseudodifferential operators belonging to the forbidden class $S_{1,1}^0$, fractional order and hyper-singular paraproduct boundedness, a smooth-oscillating decomposition for singular integrals, sparse domination estimates that quantify regularity and oscillation, and several operator calculus results.  It is of particular interest that many of these results do not require $L^2$-boundedness of the operator, and furthermore, we apply our results to some operators that are known not to be $L^2$-bounded.
\end{abstract}

\section{Introduction}\label{Sect1}

Our primary goal in this article is to develop a theory that connects integral operators of different singularities through composition with fractional derivatives, and to understand how these operators are related through vanishing moment conditions and distribution space boundedness properties.  The operators we consider are, formally speaking, $\nu$-order singular integral operators of the form
$$Tf(x)=\int_{\R^n}K(x,y)f(y)dy,$$
where $K$ is a kernel that satisfies $|K(x,y)|\leq|x-y|^{-(n+\nu)}$, which will be defined precisely in Section \ref{Sect2} as members of the $\nu$-order Singular Integral Operator class $SIO_\nu$.  When $\nu>0$, $T$ is a hyper-singular operator and resembles differentiation in some sense.  When $\nu<0$, $T$ is a fractional order operator and resembles a fractional integral or anti-differentiation operator.  The prototypical example for such operators are the $\nu$-order fractional derivatives $|\nabla|^\nu\in SIO_\nu$ for $\nu\neq0$, which are defined via the Fourier multiplier $|\xi|^\nu$.  Such operators are typically viewed as $\nu$-order derivatives when $\nu>0$ and $\nu$-order fractional integrals when $\nu<0$, which agrees with our rough interpretation of $SIO_\nu$.  We will also work with the zero-order, or critical index, class $SIO_0$, which include several mainstays in harmonic analysis, like the Hilbert transform, Riesz transforms, other Calder\'on-Zygmund operators, zero-order pseudodifferential operators, and others.  Zero-order operators have been studied extensively.  Some of the work most closely related to the current article include \cite{CM1,DJ,FTW,FHJW,T,HH,HartLu1,HO}.  There appears to be much less theory developed for the classes $SIO_\nu$ for $\nu\neq0$.  The most relevant sources are \cite{T} for $\nu\neq0$ and \cite{CHO} for $\nu<0$.

We work with operators in $SIO_\nu$ in several different capacities, which we will eventually show are all equivalent in some sense.  For $T\in SIO_\nu$, we consider the following questions:  Under what conditions on $T$ and $s,t\in\R$ does $|\nabla|^{-s}T|\nabla|^t$ belong to $SIO_{\nu+t-s}$?  Under what condition can $T\in SIO_\nu$ be extended to a bounded operator on various distribution spaces, including weighted and unweighted homogeneous Sobolev, Besov, and Triebel-Lizorkin spaces?  We will show that the appropriate conditions for $T$ are a $\nu$-order Weak Boundedness Property, which we refer to as $WBP_\nu$,  and vanishing moment conditions of the form $T^*(x^\alpha)=0$.  Furthermore, we show that, in many situations, both the operator calculus properties and distribution space boundedness properties of an operator $T\in SIO_\nu$ are sufficient conditions for $T^*(x^\alpha)=0$.  Thus, we develop many equivalent conditions between vanishing moment properties, distribution space boundedness, and operator composition properties.

The original motivation for this work came from the notion of an operator calculus for different classes of integral operators, some examples of which include classes of linear and bilinear pseudodifferential operators, Calder\'on-Zygmund operators, and fractional integral operator.  It appears to be a consensus that the origins of a pseudodifferential symbolic calculus lie in the theory of singular integral operators, but it is not clear exactly where or when such a calculus first appeared.  Some have credited early development of the topic to Bokobza, Calder\'on, Mihlin, H\"ormander, Kohn, Nirenberg, Seeley, Unterberger, and Zygmund;  see \cite{Seel,NS,St3} for more information on the early history of this.  An operator calculus for a forbidden class, sometimes also referred to as an exotic class, of pseudodifferential operators was formulated by Bourdaud \cite{Bourd}, and the notion of a symbolic calculus for bilinear pseudodifferential operators was introduced by B\'enyi, Maldonado, Naibo, and Torres \cite{BMNT}.  Several algebras of Calder\'on-Zygmund operators have been formulated, for example, by Coifman and Meyer in \cite{CMbook}.  These subclasses of operators can be defined in terms of almost diagonal operators, which is a notion depending on wavelet decompositions, but they ultimately bare out to be equivalent to vanishing moment conditions for the operator.  Finally, the current authors developed a restricted calculus for linear and bilinear fractional integral operators in \cite{CHO} that in some senses resembles the one we present here.  One of the main results of this article is a restricted calculus for $SIO_\nu$ (see Theorems \ref{t:1sidecalc} and \ref{t:calculus}), and as an application we also introduce some new operator algebras associated to $SIO_\nu$ in Section \ref{Sect7.6}.

Our first objective is to develop the restricted operator calculus for $SIO_\nu$ being acted on by $|\nabla|^s$ for $s\in\R$.  A little more precisely, in Theorem \ref{t:1sidecalc} we show that if $T\in SIO_\nu$ satisfies a $\nu$-order weak boundedness property and $T^*(x^\alpha)=0$ for appropriate multi-indices $\alpha\in\N_0^n$, then $|\nabla|^{-s}T|\nabla|^t$ agrees modulo polynomials with an operator in $SIO_{\nu+t-s}$ for certain ranges of $s,t\in\R$.  In the process of formulating this restricted calculus, we prove some estimates for functions of the form $\psi_k*Tf(x)$ for $T\in SIO_\nu$, which are of interest and useful on their own right; see Theorem \ref{t:truncation} and Corollary \ref{c:ao}.

Our second goal is to show that the same conditions on $T$ mentioned in the previous paragraph are also sufficient for the boundedness of $|\nabla|^{-s}T|\nabla|^t$, and hence of $T$, on certain distribution spaces.  We show that $|\nabla|^{-s}T|\nabla|^t$ can be extended, modulo polynomials, to a bounded linear operator from $\dot W^{\nu+t-s,p}$ into $L^p$, which implies $T$ is can be extended to a bounded linear operator from $\dot W^{\nu-t}$ into $\dot W^{-t,p}$ for appropriate $t>\nu$; see Theorem \ref{t:calculus} for more on this.  We extend the boundedness properties of both $|\nabla|^{-s}T|\nabla|^t$ and $T$ to other functions spaces as well, including weighted Besov and Triebel-Lizorkin spaces, in Theorem \ref{t:BesovBounds}, Theorem \ref{t:TriebelLizorkinBounds}, Corollary \ref{c:dual1}, and Corollary \ref{c:dual2}.  A significant feature of all the results mentioned to this point, including the ones in the preceding paragraph, is that no a priori boundedness of $T$ is required.  Indeed, if $T\in SIO_0$, one need not require even $L^2$-boundedness for $T$ to apply these results.  Furthermore, it is even possible to apply these results to operators that are not $L^2$-bounded, and we provide some example of such operators.  This notion will be explored in more detail in the applications provided in Section \ref{Sect7}, specifically in Sections \ref{Sect7.2}-\ref{Sect7.5}.

Our third objective is to establish several sufficient conditions for vanishing moment properties of the form $T^*(x^\alpha)=0$.  We show that under slightly stronger initial assumptions on $T$, the results pertaining to $|\nabla|^{-s}T|\nabla|^t$ and boundedness of $T$ discussed in the last two paragraphs are not only necessary for $T^*(x^\alpha)=0$ conditions, but also sufficient.  Hence combining the results obtained in the direction of our second and third goals, we prove $T1$-type theorems that provide necessary and sufficient conditions for boundedness on many classes of distribution spaces.  Furthermore, this verifies that the vanishing moment conditions $T^*(x^\alpha)=0$ for $T$ are necessary and sufficient to well-define $|\nabla|^{-s}T|\nabla|^t$ as a singular integral operator.  The results pertaining to sufficiency conditions for $T^*(x^\alpha)=0$ are in Theorem \ref{t:T1sufficient}, and $T1$-type theorems are discussed in Section \ref{Sect7.1}.

There are some general insights about sufficient conditions for vanishing moments of the form $T^*(x^\alpha)=0$ that can be gained for the results associated to our third goal.  It is well known that if $T$ is a Calder\'on-Zygmund operator that can be extended to a bounded operator on the weighted Hardy spaces $H^p$ for $p\leq1$, then $T^*(x^\alpha)=0$ for appropriate $\alpha\in\N_0^n$ depending on the size of $p$.  We refer to $p$ here as a Lebesgue parameter or Lebesgue index since $H^p$ is defined in terms of an $L^p$ norm.  Thus, if $T$ is bounded on distribution spaces with small enough Lebesgue index $p$, then $T^*$ must vanish on polynomials up to some degree.  This is one way to formulate sufficient conditions for $T^*(x^\alpha)=0$.  We show that if $T$ is bounded on spaces $\dot W^{-t,p}$ for certain ranges of $t>0$ and $1<p<\infty$, then $T^*(x^\alpha)=0$; the same holds for boundedness for Triebel-Lizorkin and Besov spaces.  This shows that boundedness on negative smoothness index spaces provide another way to formulate sufficient conditions for $T^*(x^\alpha)=0$.  Working formally, it then also follow by duality that if $T$ is bounded on positive index smoothness spaces, then $T(x^\alpha)=0$.  These two ways to formulate sufficient conditions for vanishing moments are well-understood, for example some results along these lines can be found in \cite{AM,CMbook,HartLu1}.  We provide two other types of sufficient conditions for $T^*(x^\alpha)=0$.  One is to require $T$ to be bounded on weighted distribution spaces where the weights are outside the natural weight class for the Lebesgue index of the space.  For instance, if $T$ is bounded on $H^2_w$ for all Muckenhoupt weights $w\in A_\infty$, then $T^*(x^\alpha)=0$ for all $\alpha\in\N_0^n$.  Similar results hold for other Triebel-Lizorkin spaces, for Lebesgue indices other than $2$, and for $A_q$ in place of $A_\infty$.  The final way we formulate sufficient conditions for $T^*(x^\alpha)=0$ is by requiring the operator $|\nabla|^{-s}T|\nabla|^t$ to agree modulo polynomials with operators in $SIO_{\nu+t-s}$, which of course are closely related to the boundedness of $T$ on distribution spaces, but nonetheless provide another sufficient condition for $T^*(x^\alpha)=0$.  All of these approaches to formulating sufficient conditions for $T^*(x^\alpha)=0$ are illustrated in Theorem \ref{t:T1sufficient}.

Finally, in Section \ref{Sect7} we will provide several applications.  Our first application, in Section \ref{Sect7.1}, is a $T1$ type theorem that extends the boundedness of a Calder\'on-Zygmund operator outside of the realm of Lebesgue spaces and for operators of arbitrary order $\nu\in\R$.  In Corollary \ref{c:CZvequiv}, we impose a little more on a given operator a priori, that $T$ belong to $CZO_\nu$ rather than just $SIO_\nu$, and in doing so we obtain necessary and sufficient conditions using some of the results from Sections \ref{Sect4}-\ref{Sect6}.

In Section \ref{Sect7.2}, we verify that pseudodifferential operators with symbols in the forbidden class $S_{1,1}^0$ can be treated with our results.  We show that such operators that satisfy $T(x^\alpha)=0$ type conditions can be extended to bounded operators on several smooth distribution spaces.  In particular, we apply our result to typical examples of operators in this forbidden class that are not $L^2$-bounded, as well as their transposes.

In Section \ref{Sect7.3}, we construct some paraproduct operators that belong to $SIO_\nu$, for any given $\nu\in\R$, to which we can apply our operator calculus and boundedness results.  Furthermore, we construct paraproducts in $SIO_0$ that are not $L^2$-bounded, but are bounded on homogeneous Sobolev spaces $\dot W^{-t,2}$ for all $t>0$, as well as other negative smoothness indexed spaces.  Hence we are outside of the class of Calder\'on-Zygmund operators, but still obtain several boundedness results.  The paraproducts we construct are of interest in their own right as well.  In form and function, they resemble the Bony paraproduct, however we construct them for any class $SIO_\nu$ with $\nu\in\R$, and we construct them to reproduce higher order moments as apposed to just the typical condition $\Pi_b1=b$.  See \eqref{Pib} for the definition of these paraproducts, as well as Corollary \ref{c:paraproducts} and Lemma \ref{l:paraproducts} for more information on the relevant properties they satisfy.

In Section \ref{Sect7.4}, we provide a decomposition of operators in $SIO_\nu$ into two terms, an oscillation-preserving term and a regularity-preserving term.  Roughly speaking, we show in Theorem \ref{t:Smooth+Oscillating} that under some mild moment conditions on $T\in SIO_\nu$, we can write $T=S+O$ where $S,O\in SIO_\nu$, $S$ is bounded on $\dot W^{t,p}$ for a range of $t>0$ and $1<p<\infty$, and $O$ is bounded on $\dot W^{-t,p}$ for all $t>0$ and $1<p<\infty$.  We actually show that our decomposition satisfies the conditions $S(x^\alpha)=0$ for several values of $\alpha\in\N_0^n$ and $O^*(x^\alpha)=0$ for all $\alpha\in\N_0^n$.  Then by our results in Section \ref{Sect5}, we obtain boundedness results for both $S$ and $O$ on different classes of spaces.  In some senses, Theorem \ref{t:Smooth+Oscillating} describes how any operator $T\in SIO_\nu$ can be decomposed $T=S+O$, where $S$ behaves like a convolution operator with respect to smoothness properties and $O$ behaves like a convolution operator with respect to oscillatory properties.  See Section \ref{Sect7.4} for more information on this.  It should be noted that this decomposition is valid for operators $T\in SIO_\nu$ that, once again, are not bounded from $\dot W^{\nu,2}$ into $L^2$.

In Section \ref{Sect7.5}, we prove smooth and oscillatory sparse domination principles for  operators in $SIO_\nu$.  This application is included to demonstrate the following notion.  We expend a lot of effort to provide conditions for an operator $T\in SIO_\nu$ that imply $|\nabla|^{-(\nu+t)}T|\nabla|^{t}$ is a Calder\'on-Zygmund operator (modulo polynomials).  Hence we can obtain new results for any such $T$ by applying existing results from Calder\'on-Zygmund theory to $|\nabla|^{-(\nu+t)}T|\nabla|^{t}$.  In a way, our restricted operator calculus allows us to translate $CZO_0$ theory to $SIO_\nu$ theory for $\nu\neq0$.  We demonstrate this principle through the sparse domination principle for $SIO_\nu$ in Corollary \ref{c:sparse}.

In Section \ref{Sect7.6}, we develop some new operator calculus results.  It appears that this is the first operator algebra that includes operators of non-convolution type containing hypersingular and fractional integral operators.  Furthermore, in Theorem \ref{t:operatoralgebra} we describe several operator algebras that include operators of different singularities.  Some are made up of differential operators, fractional integral operators, and/or order-zero operators in various combinations.

This article is organized as follows.  In Section \ref{Sect2}, we provide several definitions, notation, and preliminary results.  Section \ref{Sect3} contains the bulk of the work of truncating and approximating singular integral operators, and provides crucial support for the results in the sections that follow.  In Section \ref{Sect4}, we formulate our restricted operator calculus by verifying $|\nabla|^{-s}T|\nabla|^t$ has a $(\nu+t-s)$-order standard function kernel and satisfies certain boundedness properties for appropriate $T\in SIO_\nu$.  Section \ref{Sect5} is dedicated to proving several boundedness results for operators $T\in SIO_\nu$ that satisfy $T^*(x^\alpha)=0$ vanishing moment conditions, and Section \ref{Sect6} provides sufficient conditions for $T^*(x^\alpha)=0$.  Finally, we present several applications in Section \ref{Sect7}.

\section{Preliminaries, Definitions, and Notation}\label{Sect2}

Let $\S$ be the Schwartz class, $\S_P$ be the subspace of $\S$ made up of functions with vanishing moments of all order up to $P$, and $\S_\infty$ be the intersection of $\S_P$ for all $P\in\N$.  We give $\S$ the standard Schwartz semi-norm topology defined via
\begin{align*}
\rho_{\alpha,\beta}(f)=\sup_{x\in\R^n}|x^\alpha\cdot D^\beta f(x)|.
\end{align*}
It is easy to verify that $\S_P$ for $P\in\N_0$ and $\S_\infty$ are closed subspaces of $\S$.  Hence we can give $\S_P$ and $\S_\infty$ the Frech\'et topology endowed by the Schwartz semi-norms for $\S$.  Let $\S'$, $\S_P'$, and $\S_\infty'$ be the dual spaces of $\S$, $\S_P$, and $\S_\infty$ respectively, which we refer to as tempered distributions, tempered distributions modulo polynomials of degree $P$, and tempered distributions modulo polynomials, respectively.  Let $\mathcal D=C_0^\infty$ be the space of smooth compactly supported functions, and define $\mathcal D_P$ to be the subspace of $\mathcal D$ made up of all function with vanishing moments up to order $P$.  We endow $\mathcal D$ with the topology characterized by the following sequential convergence:  for $f_j,f\in\mathcal D$, we say $f_j\rightarrow f$ in $\mathcal D$ if there exists a compact set $K\subset\R^n$ such that $\supp(f),\supp(f_j)\subset K$ for all $j$ and $D^\alpha f_j\rightarrow D^\alpha f$ uniformly as $j\rightarrow\infty$ for all $\alpha\in\N_0^n$.  It follows that $\mathcal D_P$ is a closed subspace of $\mathcal D$ for any $P\in\N_0$, and hence we endow $\mathcal D_P$ with the topology inherited from $\mathcal D$.  Let $\mathcal D'$ and $\mathcal D_P'$ be the dual spaces of $\mathcal D$ and $\mathcal D_P$, respectively, which we refer to as distributions and distributions modulo polynomials of degree at most $M$.

For $1<p<\infty$, a non-negative locally integrable function $w$ belongs to the Muckenhoupt weight class $A_p$ if
\begin{align*}
[w]_{A_p}=\sup_Q\(\frac{1}{|Q|}\int_Qw(x)dx\)\(\frac{1}{|Q|}\int_Qw(x)^{-p'/p}dx\)^{p/p'}<\infty,
\end{align*}
where the supremum is taken over all cubes $Q\subset\R^n$, and $w$ belongs to $A_1$ if there exists a constant $C>0$ such that $\mathcal Mw(x)\leq Cw(x)$, where $\mathcal M$ is the Hardy-Littlewood maximal operator.  Also define $A_\infty$ to be the union of all $A_p$ for $1\leq p<\infty$.

Let $\psi\in\S_\infty$ such that $\widehat\psi$ is supported in the annulus $1/2<|\xi|<2$ and $\widehat\psi(\xi)\geq c>0$ for some $c>0$ and all $3/5<|\xi|<5/3$.  For $w\in A_\infty$, define $\dot F_{p,w}^{s,q}$ to be the collection of $f\in\S_\infty'$ such that
\begin{align*}
||f||_{\dot F_{p,w}^{s,q}}=\left|\left|\(\sum_{k\in\Z}(2^{sk}|\psi_k*f|)^q\)^{1/q}\right|\right|_{L^p_w}<\infty
\end{align*}
for $0<p,q<\infty$ and $s\in\R$, and define $\dot B_{p,w}^{s,q}$ to be the collection of all $f\in\S_\infty'$ such that
\begin{align*}
||f||_{\dot B_{p,w}^{s,q}}=\(\sum_{k\in\Z}(2^{sk}||\psi_k*f||_{L^p_w})^q\)^{1/q}<\infty
\end{align*}
for $0<p\leq\infty$, $0<q<\infty$, and $s\in\R$.  Also define $\dot F_{\infty,w}^{s,q}$ to be the collection of all $f\in\S_\infty'$ such that
\begin{align*}
\|f\|_{\dot F_{\infty,w}^{s,q}}=\sup_Q\(\frac{1}{w(Q)}\int_Q\sum_{k\in\Z:2^{-k}\leq\ell(Q)}(2^{sk}|\psi_k*f(x)|)^q\)^{1/q}<\infty,
\end{align*}
where the supremum is taking over all cubes $Q\subset\R^n$ with sides parallel to the axes and $\ell(Q)$ denotes the side length of $Q$.  Finally define $\dot B_{\infty,w}^{s,\infty}=\dot B_\infty^{s,\infty}$ for $s\in\R$ to be the collection of $f\in\S_\infty'$ such that
\begin{align*}
\|f\|_{\dot B_\infty^{s,\infty}}=\sup_{k\in\Z}2^{sk}\|\psi_k*f\|_{L^\infty}<\infty.
\end{align*}
Taking these spaces modulo polynomials makes them Banach spaces for $1\leq p,q<\infty$ and $s\in\R$ and quasi-Banach spaces when $0<p,q<\infty$ and $s\in\R$.  Note that by the work in \cite{FJ2} for the unweighted setting and \cite{Bui} for the weighted setting, it follows that $\S_\infty$ is dense in $\dot F_p^{s,q}$ for all $0<p,q<\infty$ and $s\in\R$.  It also follows that $H^p_w=L^p_w=\dot F^{0,2}_{p,w}$ for all $1<p<\infty$ and $w\in A_p$, $H^p_w=\dot F^{0,2}_{p,w}$ for all $0<p<\infty$ and $w\in A_\infty$, and $\dot W^{s,p}_w=\dot F_{p,w}^{s,2}$ for all $1<p<\infty$ and $w\in A_p$.  Here $L^p_w$ denote weighted Lebesgue spaces, $H^p_w$ denote weighted Hardy spaces, and $\dot W^{s,p}_w$ denote weighted homogeneous Sobolev spaces.  Even more, $\dot B_\infty^{s,\infty}=\dot\Lambda_s$ is the (homogeneous) space of Lipschitz functions when $s>0$ is not an integer.  When $s\in\N$, $B_\infty^{s,\infty}$ is the Zygmund class of smooth functions, which strictly contain $\dot\Lambda_s$; see \cite{Z} for more on Zygmund's smooth functions.  For $s=0$, the space $\dot B_\infty^{0,\infty}$ is sometimes referred to as the Bloch space, and it closely related to certain Bergman spaces; see for example \cite{CMbook} for more on this.

Finally, we note also that $\dot F_{\infty,w}^{s,2}=\dot F_{\infty}^{s,2}=I_s(BMO)$ for all $s>0$ and $w\in A_\infty$, which was proved in \cite{HO}.  Here $I_s(BMO)$ are Sobolev-$BMO$ spaces for $s>0$, and we take the convention $I_0(BMO)=BMO$.  See \cite{N,Str1,Str2,HO} for more information on these spaces.

Let $X$ be a closed subspace of $\S(\R^n)$.  We say that a linear operator $T$ mapping $X$ into $\S'(\R^n)$ is continuous if there exists a distribution kernel $\mathcal K\in\S'(\R^{2n})$ such that
\begin{align*}
\<Tf,g\>=\<\mathcal K,g\otimes f\>=\int_{\R^{2n}}\mathcal K(x,y)g(x)f(y)dy\,dx
\end{align*}
for all $f\in X$ and $g\in\S(\R^n)$.  Here and throughout this article, any integral that has $\mathcal K(x,y)$ in the integrand should be interpreted as a dual pairing between $\S'(\R^{2n})$ and $\S(\R^{2n})$.  We will use this notion of continuity when $X$ is $\S$ and $\S_P$ for $P\in\N_0$ at various points throughout the article.  It is obvious that continuity from $\S$ into $\S'$ implies continuity from $\S_P$ into $\S'$ for $P\in\N_0$, which implies continuity from $\S_{P+1}$ into $\S'$ and from $\S_\infty$ into $\S'$.

We consider operators that are continuous from $\S_M$ into $\S'$ since it make it easier in some situations to initially define and work with operators.  For example, consider the negative index derivative operator $|\nabla|^{-\nu}f$ defined via the Fourier multiplier $|\xi|^{-\nu}$ for $\nu>0$.  For $0<\nu<n$, it is easy to define $|\nabla|^{-\nu}f$ for $f\in\S$ since $|\xi|^{-\nu}$ is locally integrable for such $\nu$, but it is a little more tedious to define $|\nabla|^{-\nu}$ when $\nu\geq n$.  Since we allow for our operators to be defined a priori only on $\S_P$ for some $P\in\N_0$, it is much easier to work with such operators.  For any $\nu>0$, we choose $P\geq \nu$, and it follows that $\widehat{|\nabla|^{-\nu}f}(\xi)$ is uniformly bounded for $f\in\S_P$.  So $|\nabla|^{-\nu}$ trivially defines a continuous operator from $\S_P$ into $\S'$ as long as $P\geq\nu$.  This type of issue is less severe for the fractional derivative operator $|\nabla|^\nu$, defined in the same way via the Fourier multiplier $|\xi|^\nu$, but using $\S_P$ rather than $\S$ may still be of value.  This is because $|\xi|^\nu$ is not smooth at the origin (for certain $\nu>0$), and requiring $f\in\S_P$ for some $P\in\N_0$ smooths this non-regularity, at least to some degree.

Assuming that $T$ is continuous from $\S_P$ into $\S'$ makes defining the transpose of $T$ a little tricky.  In this paper we will impose on a given operator $T$ that both $T$ and $T^*$ are continuous from $\S_P$ into $\S'$ for some $P\in\N$.  By this we mean that $T$ is continuous from $\S_P$ into $\S'$, and there exists another operator $S$ that is also continuous from $\S_P$ into $\S'$ such that $\<Tf,g\>=\<Sg,f\>$ for all $f,g\in\S_P$.  We call this operator $S=T^*$ the transpose of $T$.  It should be noted that we assume that $T^*$ is continuous from $\S_P$ into $\S'$.  This does not necessarily follow from the continuity of $T$ from $\S_P$ into $\S'$.

\begin{definition}
Let $\nu\in\R$, $M\geq0$ be an integer, and $0<\gamma\leq1$.  A linear operator $T$ is in the class of $\nu$-order Singular Integral Operators, denoted $T\in SIO_\nu(M+\gamma)$, if $T$ and $T^*$ are continuous from $\S_P$ into $\S'$ from some $P\in\N_0$, there is a kernel function $K(x,y)$ such that
\begin{align*}
\<Tf,g\>=\int_{\R^{2n}}K(x,y)f(y)g(x)dy\,dx
\end{align*}
for any pair $(f,g)\in \mathcal D_P\times C_0^\infty$ or $(f,g)\in C_0^\infty\times \mathcal D_P$ with disjoint support,
\begin{align*}
|D_x^\alpha D_y^\beta K(x,y)|\less\frac{1}{|x-y|^{n+\nu+|\alpha|+|\beta|}}
\end{align*}
for all $x\neq y$ and $\alpha,\beta\in\N_0^n$ with $|\alpha|+|\beta|\leq M$, and
\begin{align*}
|D_x^\alpha D_y^\beta K(x+h,y)-D_x^\alpha D_y^\beta K(x,y)|+|D_x^\alpha D_y^\beta K(x,y)-D_x^\alpha D_y^\beta K(x,y+h)|\less\frac{|h|^\gamma}{|x-y|^{n+\nu+M+\gamma}}
\end{align*}
for all $x,y,h\in\R^n$ satisfying $|h|<|x-y|/2$ and $\alpha,\beta\in\N_0^n$ satisfying $|\alpha|+|\beta|=M$.  We will refer to $SIO_\nu$ as the union of all $SIO_\nu(M+\gamma)$ for $M\in\N_0^n$ and $0<\gamma\leq1$ and $SIO_\nu(\infty)$ as the intersection of $SIO_\nu(M+\gamma)$ for $M\in\N_0^n$ and $0<\gamma\leq1$.
\end{definition}

Note that the kernel representation imposed on $T\in SIO_\nu(M+\gamma)$ implies that $Tf$ can be realized as a function for $f\in\mathcal D_P$ and $x\notin\supp(f)$.  Indeed, for such $f$, $Tf\in\S'$ by assumption.  Then by taking appropriate kernel functions $\varphi_k\in C_0^\infty$ that generate an approximation to the identity, the kernel representation of $T$ and dominated convergence imply that $\varphi_k*Tf(x)=\<Tf,\varphi^x\>$ converges as $k\rightarrow\infty$ as long as $x\notin\supp(f)$.  Hence we can realize the distribution $Tf\in\S'$ pointwise as
$$Tf(x)=\lim_{k\rightarrow\infty}\varphi_k*Tf(x)=\int_{\R^n}K(x,y)f(y)dy$$
when $f\in\mathcal D_P$ and $x\notin\supp(f)$.

\begin{definition}
For $\nu\in\R$, $M\in\N_0$, and $0<\gamma\leq1$, an operator $T\in SIO_\nu(M+\gamma)$ is a $\nu$-order Calder\'on-Zygmund Operator, denoted $T\in CZO_\nu(M+\gamma)$, if $T$ can be continuously extended to an operator from $\dot W^{\nu,p}$ into $L^p$ for all $1<p<\infty$.  We will also refer to $CZO_\nu$ as the union of all $CZO_\nu(M+\gamma)$ for $M\in\N_0$ and $0<\gamma\leq1$ and  $CZO_\nu(\infty)$ as the intersection of all $CZO_\nu(M+\gamma)$ for $M\in\N_0$ and $0<\gamma\leq1$.
\end{definition}

\begin{definition}
An operator $T\in SIO_\nu$ satisfies the $\nu$-order Weak Boundedness Property (WBP$_\nu$) if there are integers $M,N\geq0$ and a constant $C>0$ such that
\begin{align}\label{WBPv}
\left|\<T\psi,\varphi\>\right|+\left|\<T^*\psi,\varphi\>\right|\leq C|B|^{1-\nu/n}
\end{align}
for any ball $B\subset\R^n$, $\psi\in\mathcal D_M$ and $\varphi\in C_0^\infty$ with $\supp(\psi)\cup\supp(\varphi)\subset B$ and $||D^\alpha\psi||_{L^\infty},||D^\alpha\varphi||_{L^\infty}\leq |B|^{-|\alpha|/n}$ for $|\alpha|\leq N$.
\end{definition}

\begin{remark}
Note that the definition of $SIO_\nu$ and $WBP_\nu$ are both symmetric under $T$ and $T^*$, but this is not always the case for $CZO_\nu$.  When $\nu\neq0$, $T\in CZO_\nu$ and $T^*\in CZO_\nu$ are not equivalent conditions.  On the other hand when $\nu=0$, $CZO_0$ is closed under transposes and actually collapses to the traditional definition of a Calder\'on-Zygmund operator.
\end{remark}

For a function $F$ on $\R^n$, $x_0\in\R^n$, and an integer $L\geq0$, we define the Taylor polynomial (sometimes also called the jet) centered at $x_0$ by
\begin{align*}
J_{x_0}^L[F](x)=\sum_{|\alpha|\leq L}\frac{D^\alpha F(x_0)}{\alpha!}(x-x_0)^\alpha.
\end{align*}

\begin{definition}\label{d:Tx}
Let $\nu\in\R$, $M\geq0$ be an integer, $0<\gamma\leq1$ and $T\in SIO_\nu(M+\gamma)$.  Let $P\in\N_0$ be the integer specified for the kernel representation in the definition of $T\in SIO_\nu(M+\gamma)$, and without loss of generality assume $P\geq M+|\nu|$.  Let $\eta\in\mathcal D_{2P}$ such that $\eta=1$ on $B(0,1)$, and define $\eta_R(x)=\eta(x/R)$ for $R>0$ and $x\in\R^n$.  For $\alpha\in\N_0^n$ with $|\alpha|< M+\nu+\gamma$, define  $T^*(x^\alpha)\in\mathcal D_{2P}'$ by
\begin{align*}
\<T(x^\alpha),\psi\>=\lim_{R\rightarrow\infty}\<T(x^\alpha\cdot\eta_R),\psi\>=\lim_{R\rightarrow\infty}\int_{\R^{2n}}\mathcal K(x,y)y^\alpha \eta_R(y)\psi(x)dy\,dx
\end{align*}
for $\psi\in\mathcal D_P$.  Here $\mathcal K\in\S'(\R^{2n})$ denotes the distributional kernel of $T$, and the integrals above should be interpreted as the dual pairing between $\S(\R^{2n})$ and $\S'(\R^{2n})$.  Note that if $M+\nu+\gamma<0$, then we do not define $T(x^\alpha)$ for any $\alpha\in\N_0$.  Furthermore, we will simply state that $T(x^\alpha)=0$ to mean that there exists an integer $P\geq0$ so that $T(x^\alpha)=0$ in $\mathcal D_P'$.
\end{definition}

It is a somewhat standard argument by now to show that $T(x^\alpha)$ is well-defined.  We provide a brief sketch of this since our definition is slightly different than others that have appeared; compare, for example, with the corresponding definitions in \cite{FTW,FHJW,T,Hart2,HartLu1,HO}.

\begin{proof}
Let $\eta_R\in\mathcal D_{2P}$ be as in Definition \ref{d:Tx}.  Let $\psi\in\mathcal D_{P}$ with $\supp(\psi)\subset B(0,R_0/8)$.  For $|\alpha|< M+\nu+\gamma$ and $R>0$,
\begin{align*}
\<T(x^\alpha\cdot\eta_R),\psi\>&=\<T(x^\alpha\cdot\eta_{R_0}),\psi\>\\
&\hspace{.5cm}+\lim_{R\rightarrow\infty}\int_{\R^{2n}}\(K(x,y)-J_0^M[K(\cdot,y)](x)\)y^\alpha (\eta_R(y)-\eta_{R_0}(y))\psi(x)dy\,dx\\
&=\<T(x^\alpha\cdot\eta_{R_0}),\psi\>+\int_{\R^{2n}}\(K(x,y)-J_0^M[K(\cdot,y)](x)\)y^\alpha (1-\eta_{R_0}(y))\psi(x)dy\,dx.
\end{align*}
The first term is well-defined since $T$ is maps $\S_P$ into $\S'$ and $\eta\in\mathcal D_{2P}$ implies $x^\alpha\cdot\eta_R\in\mathcal D_P$ for $|\alpha|<M+\nu+\gamma$.  The second term is also well-defined by the support properties of $\eta_R(x)-\eta_{R_0}(x)$, the kernel representation of $T$, and the vanishing moment properties of $\psi$.  In fact, dominated convergence can be applied to the second term since
\begin{align*}
\left|\(K(x,y)-J_0^M[K(\cdot,y)](x)\)y^\alpha (\eta_R(y)-\eta_{R_0}(y))\right|\less\frac{R_0^{M+\gamma}}{(R_0+|y|)^{n+\nu+M+\gamma-|\alpha|}}
\end{align*}
for $x\in\supp(\psi)$.  Therefore $T(x^\alpha)$ is also well defined.  Furthermore, it is not hard to see that the definition of $T(x^\alpha)$ does not depend on the particular function $\eta_R\in\mathcal D_{2P}$ chosen in Definition \ref{d:Tx}.
\end{proof}

Throughout, we will use the notation $\Phi_k^N(x)=2^{kn}(1+2^k|x|)^{-N}$ for $N>0$, $k\in\Z$, and $x\in\R^n$.  It is well known that $\Phi_k^N*|f|(x)\less \mathcal Mf(x)$ for any locally integrable function $f$ and $N>n$, where the constant may depend on $N$ and $\mathcal M$ is the Hardy-Littlewood maximal operator.  It is also well known that $\Phi_k^N*\Phi_j^M(x)\less\Phi_{\min(j,k)}^{\min(M,N)}(x)$ for any $j,k\in\Z$, $M,N>0$ such that $M,N>n$, and the constant depends on $M,N$, but not on $j,k,x$.

The following lemmas are also well-known.  More information can be found for example in \cite{FJ2}.

\begin{lemma}\label{l:calderon}
Let $P\geq0$ be an integer.  There exist functions $\psi\in\mathcal D_P$ and $\widetilde\psi\in\S_\infty$ such that
\begin{align*}
f(x)=\sum_{k\in\Z}\psi_k*\widetilde\psi_k*f(x)
\end{align*}
in $\S_\infty$ for any $f\in\S_\infty$.  Furthermore, $\psi$ and $\widetilde\psi$ can be chosen to be radial.
\end{lemma}

\begin{lemma}\label{l:wavelet}
Let $P\geq0$ be an integer.  There exist functions $\phi\in\mathcal D_P$ and $\widetilde\phi\in\S_\infty$ and an integer $N_0\in\Z$ such that
\begin{align*}
f(x)=\sum_{k\in\Z}\sum_{Q:\ell(Q)=2^{-(k+N_0)}}\widetilde\phi_k*f(c_Q)\phi_k(x-c_Q)
\end{align*}
in $\S_\infty$ for any $f\in\S_\infty$.  The sum in $Q$ here is over all dyadic cubes of side length $2^{-(k+N_0)}$ and $c_Q$ denotes the lower-left hand corner of $Q$.  Furthermore, $\widetilde\phi$ and $\phi$ can be chosen to be radial.
\end{lemma}

Throughout this article, we will choose convolution kernel functions $\psi$, $\widetilde\psi$, $\phi$, and $\widetilde\phi$ to be radial so that they are self-transpose; that is, so that $\<\psi_j*f,g\>=\<\psi_j*g,f\>$.  We will make the same convention when working with approximation to the identity operators and functions $\psi$ used to define Besov and Triebel-Lizorkin spaces.  Of course, this simplification is not necessary, but it eases the need for complicated notation.

Define
\begin{align*}
\mathcal M_j^rf(x)=\left\{\mathcal M\[\(\sum_{Q:\ell(Q)=2^{-(j+N_0)}}f(c_Q)\cdot \chi_Q\)^r\;\](x)\right\}^{1/r}.
\end{align*}

\begin{lemma}\label{l:Mj1}
Let $f:\R^n\rightarrow\C$ be a non-negative continuous function, $\mu>0$, and $\frac{n}{n+\mu}<r\leq1$.  Then
\begin{align*}
\sum_{Q:\ell(Q)=2^{-(j+N_0)}}|Q|\Phi_{\min(j,k)}^{n+\mu}(x-c_Q)f(c_Q)\less2^{\max(0,j-k)\mu}\mathcal M_j^rf(x)
\end{align*}
for all $x\in\R^n$.
\end{lemma}

The properties of $\mathcal M_j^r$ in Lemma \ref{l:Mj1} can, at least in part, be attributed to Frazier and Jawerth, but can be found in several other places as well.  A proof of it, as stated here, can also be found in \cite[Proposition 2.2]{HartLu1}.

\begin{lemma}\label{l:Mj2}
For any $0<p,q<\infty$, $0<r<\min(p,q)$, $t\in\R$, $w\in A_{p/r}$ and $f\in\dot F_{p,w}^{t,q}$, we have
\begin{align*}
\left|\left|\(\sum_{j\in\Z}\[2^{tj}\mathcal M_j^r(\widetilde\phi_j*f)\]^q\)^{1/q}\right|\right|_{L^p_w}\less||f||_{\dot F_{p,w}^{t,q}},
\end{align*}
where $\widetilde\phi$ is chosen as in Lemma \ref{l:wavelet}.
\end{lemma}

Lemma \ref{l:Mj2} is implicit in the work of Bui in \cite{Bui}.  Indeed Lemma \ref{l:Mj2} can be proved with a standard argument using the weighted version of the Fefferman-Stein vector-valued maximal inequality proved in \cite{AJ} and the discrete Littlewood-Paley characterization of $\dot F_{p,w}^{t,q}$ proved in \cite{Bui}.  See also \cite{LZ}, where they prove this result in the setting of weighted Hardy spaces (i.e. for $t=0$ and $q=2$).

\section{Smoothly Truncated Kernel Estimates}\label{Sect3}

In this section, we start to work with operators $T\in SIO_\nu(M+\gamma)$.  A priori, we are very limited in what we can do with such operators.  We will impose a little more structure on $T$ through the $WBP_\nu$ and $T^*(x^\alpha)=0$ conditions to help gain some traction, which we will see later are necessary conditions in many situations.  The main purpose of these results is to find representation formulas for $T$ that make it easier to work with.  Our first result provides an integral representation and estimates for $\psi_k*Tf$, which will be useful for working with boundedness properties of $T$ on Besov and Triebel-Lizorkin spaces.  These estimates are similar to ones proved in \cite{HartLu1}.  Though it should be noted that the corresponding result in \cite{HartLu1} is only for order-zero operators and it assumes that the operator is $L^2$-bounded, whereas we address $\nu$ order operators and only assume Weak Boundedness Properties here.

\begin{theorem}\label{t:truncation}
Let $\nu\in\R$, $M,L\geq0$ be integers satisfying $\nu\leq L\leq M+\nu$, $0<\gamma\leq1$, and $T\in SIO_\nu(M+\gamma)$.  Assume that $T$ satisfies the $WBP_\nu$ and that $T^*(x^\alpha)=0$ for $|\alpha|\leq L$.  Then for any $\psi\in\mathcal D_P$, with $P$ sufficiently large, there is a kernel function $\theta_k(x,y)$ such that
\begin{align*}
\psi_k*Tf(x)=\int_{\R^n}\theta_k(x,y)f(y)dy
\end{align*}
for $f\in\S_P$, where $\theta_k(x,y)$ satisfies
\begin{align}
&|D_y^\alpha \theta_k(x,y)|\less2^{(\nu+|\alpha|)}\Phi_k^{n+\nu+M+\gamma}(x-y)&			&\text{for }|\alpha|\leq \lfloor L-\nu\rfloor\label{sqker1}\\
&|D_y^\alpha \theta_k(x,y+h)-D_y^\alpha \theta_k(x,y)|\less2^{(\nu+|\alpha|)k}(2^k|h|)^{\gamma\,'}\Phi_k^{n+\nu+M+\gamma}(x-y)&		&\text{for }|\alpha|= \lfloor L-\nu\rfloor\label{sqker2}
\end{align}
for all $x,y,h\in\R^n$ with $|h|<(2^{-k}+|x-y|)/2$ and any $0<\gamma\,'<\gamma$.

\end{theorem}

\begin{proof}
Fix $P\in\N_0$ large enough so that $P\geq M$, $T$ satisfies $WBP_\nu$ with parameters $M=N=P$, the kernel representation for $T$ is valid for $f,g\in\mathcal D_P$ with disjoint support, and $T^*(x^\alpha)=0$ in $\mathcal D_{P}'$ for all $|\alpha|\leq L$.  Let $\psi\in\mathcal D_P$, and assume without loss of generality that $\supp(\psi)\subset B(0,1)$.  Since $T$ is continuous from $\S_P$ into $\S'$, we have $T^*\psi_k^x\in\S'$ for all $x\in\R^n$ and $k\in\Z$, where $\psi_k^x(y)=\psi_k(x-y)$.  This is the distribution kernel of $\psi_k*Tf$ in the sense that
\begin{align*}
\psi_k*Tf(x)&=\<Tf,\psi_k^x\>=\<T^*\psi_k^x,f\>
\end{align*}
for $f\in \S_P$.  So we would like to define $\theta_k(x,y)=T^*\psi_k^x(y)$, but this expression may not be well-defined pointwise since $T^*\psi_k^x$ is a priori only a distribution (not necessarily a function in $y$).  However, we will show that this is merely a technicality, and that $T^*\psi_k^x$, as a distribution, agrees with integration against a function $T^*\psi_k^x(y)$.

Let $\varphi\in C_0^\infty$ with integral $1$ such that $\widetilde\psi(x)=2^{n}\varphi(2x)-\varphi(x)$ and $\widetilde\psi\in\mathcal D_{P}$.  Define $P_kf=\varphi_k*f$ and $\widetilde Q_\ell f=\widetilde\psi_\ell*f$.  Note that $P_Nf\rightarrow f$ as $N\rightarrow\infty$ in $\S$ for $f\in\S$, and so $P_NT^*\psi_k^x\rightarrow T^*\psi_k^x $ as $N\rightarrow\infty$ in $\S'$.  Furthermore, $P_NT^*\psi_k^x(y)=\<T^*\psi_k^x,\varphi_N^y\>$ is a function in $y$ for all $x\in\R^n$ and $k,N\in\Z$ by the definition of distributional convolution (in fact, it is a $C^\infty$ function of tempered growth since $T^*\psi_k^x\in\S'$).  Define
\begin{align}\label{thetadefn}
\theta_k(x,y)=\lim_{N\rightarrow\infty}P_NT^*\psi_k^x(y).
\end{align}
It should be noted that from what we have shown so far, it is not clear yet that the limit in \eqref{thetadefn} exits for all $x,y\in\R^n$ and $k\in\Z$.  We will show that this is indeed the case and that the kernel function defined in \eqref{thetadefn} satisfies estimates \eqref{sqker1} and \eqref{sqker2}.  Assuming for the moment that \eqref{thetadefn} holds pointwise and that $\theta_k$ satisfies estimates \eqref{sqker1} and \eqref{sqker2} (which we will prove), by the $\S'$ convergence $P_NT^*\psi_k^x\rightarrow T^*\psi_k^x$ and by dominated convergence, we have
\begin{align*}
\psi_k*Tf(x)=\<T^*\psi_k^x,f\>=\lim_{N\rightarrow\infty}\int_{\R^n}P_NT^*\psi_k^x(y)f(y)dy=\int_{\R^n}\theta_k(x,y)f(y)dy
\end{align*}
for all $f\in\S_P$.  So to complete the proof, we must show that the limit in \eqref{thetadefn} exists for each $x,y,k$ and that $\theta_k$ satisfies \eqref{sqker1} and \eqref{sqker2}.

For $|x-y|>2^{3-k}$, it follows from the kernel representation of $T$ that $T^*\psi_k^x$ is a continuous function on a sufficiently small neighborhood of $y$ for any fixed $x\in\R^n$ and $k\in\Z$, and that $P_NT^*\psi_k^x(y)\rightarrow T^*\psi_k^x(y)$ pointwise as $N\rightarrow\infty$.  In particular, the limit in \eqref{thetadefn} exists when $|x-y|>2^{3-k}$.  Furthermore, still assuming that $|x-y|>2^{3-k}$, by the kernel representation for $T$ and the support properties of $\psi_k$, it follows that $T^*\psi_k^x$ is $M$-times differentiable on a neighborhood of $y$.  For $|\alpha|\leq \lfloor L-\nu\rfloor$, we have
\begin{align*}
|D_y^\alpha\theta_k(x,y)|=|D_y^\alpha T^*\psi_k^x(y)|&=\left|\int_{\R^n}\(D_1^\alpha K(u,y)-J_x^{M-|\alpha|}[D_1^\alpha K(\cdot,y)](u)\)\psi_k(u-x)du\right|\\
&\less\int_{\R^n}\frac{|x-u|^{M-|\alpha|+\gamma}}{|x-y|^{n+\nu+M+\gamma}}|\psi_k(u-x)|du
\less2^{(\nu+|\alpha|)k}\Phi_k^{n+\nu+M+\gamma}(x-y).
\end{align*}
If $|h|\geq2^{-k}$, then it trivially follows that
\begin{align*}
|D_y^\alpha \theta_k(x,y+h)-D_y^\alpha \theta_k(x,y)|&\less2^{(\nu+|\alpha|)k}\[\Phi_k^{n+\nu+M+\gamma}(x-y)+\Phi_k^{n+\nu+M+\gamma}(x-y-h)\]\\
&\less2^{(\nu+|\alpha|)k}(2^k|h|)^\gamma\[\Phi_k^{n+\nu+M+\gamma}(x-y)+\Phi_k^{n+\nu+M+\gamma}(x-y-h)\].
\end{align*}
Otherwise we assume that $|h|<2^{-k}$, and it follows that $|x-y-h|\geq|x-y|-|h|>|x-y|/2\geq2^{1-k}$.  In this situation, we consider two cases: if $|\alpha|=M$ and if $|\alpha|<M$.  When $|\alpha|=M$, we have
\begin{align*}
|D_y^\alpha \theta_k(x,y+h)-D_y^\alpha \theta_k(x,y)|&=\left|\int_{\R^n}\(D_y^\alpha K(u,y+h)- D_y^\alpha K(u,y)\)\psi_k(u-x)du\right|\\
&\hspace{0cm}\less\int_{\R^n}\frac{|h|^{\gamma}}{|u-y|^{n+\nu+|\alpha|+\gamma}}|\psi_k(u-x)|du\\
&\less2^{(\nu+|\alpha|) k}(2^{k}|h|)^{\gamma}\Phi_k^{n+\nu+M+\gamma}(x-y).
\end{align*}
Now assume that $|\alpha|<M$.  Then
\begin{align*}
|D_y^\alpha \theta_k(x,y+h)-D_y^\alpha \theta_k(x,y)|&=\left|\int_{\R^n}\(D_y^\alpha K(u,y+h)- D_y^\alpha K(u,y)\)\psi_k(u-x)du\right|\\
&\hspace{-5cm}=\left|\int_{\R^n}\( \(D_y^\alpha K(u,y+h)- D_y^\alpha K(u,y)\)-J_x^{M-1-|\alpha|}[D_y^\alpha K(\cdot,y+h)- D_y^\alpha K(\cdot,y)](u)\)\psi_k(u-x)du\right|\\
&\hspace{-5cm}\leq\int_{\R^n}\sum_{|\beta|=M-|\alpha|}|D_0^\beta D_1^\alpha K(\xi,y+h)-D_0^\beta D_1^\alpha K(\xi,y)|\,|u-x|^{M-|\alpha|}\,|\psi_k(u-x)|du\\
&\hspace{2cm}\text{ for some }\xi=cx+(1-c)u\text{ with }0<c<1\\
&\hspace{-5cm}\less\int_{\R^n}\frac{|h|^{\gamma}|x-u|^{M-|\alpha|}}{|\xi-y|^{n+\nu+M+\gamma}}|\psi_k(u-x)|du
\less2^{(\nu+|\alpha|) k}(2^{k}|h|)^{\gamma}\Phi_k^{n+\nu+M+\gamma}(x-y).
\end{align*}
So $\theta_k(x,y)$ is well-defined and satisfies \eqref{sqker1} and \eqref{sqker2} when $x$ and $y$ are far apart.

When $|x-y|\leq 2^{3-k}$, we decompose $P_NT^*\psi_k^x$ further,
\begin{align}\label{furtherdecomp}
P_NT^*\psi_k^x(y)=\sum_{\ell=k}^{N-1} \widetilde Q_\ell T^*\psi_k^x(y)+P_kT^*\psi_k^x(y).
\end{align}
Recall that $\varphi\in C_0^\infty$ was chosen so that $\widetilde\psi_\ell=\varphi_{\ell+1}-\varphi_\ell\in\mathcal D_P$.  For the remainder of the proof, we drop the tilde on top of $\widetilde\psi$ for the sake of simplifying notation.  This causes an overlap in notation between $\widetilde\psi_\ell$ and $\psi_k$, which is ultimately harmless, but the distinction can still be recovered at any point in the remainder of the proof by identifying the subscripts, $\ell$ versus $k$.

Let $\alpha\in\N_0^n$ with $|\alpha|\leq \lfloor L-\nu\rfloor$.  Using the hypothesis $T^*(x^\mu)=0$ for $|\mu|\leq L$ we write
\begin{align*}
&\left|D_y^\alpha\<T\psi_{\ell}^{y},\psi_k^x\>\right|\leq |A_{\ell,k}(x,y)|+|B_{\ell,k}(x,y)|,\text{ where}\\
&\hspace{.25cm}A_{\ell,k}(x,y)=2^{\ell|\alpha|}\int_{\R^n}T((D^\alpha\psi)_{\ell}^{y})(u)\(\psi_k^x(u)-J_y^L[\psi_k^x](u)\)\eta_{2^{2-\ell}}(u-y)du,\\
&\hspace{.25cm}B_{\ell,k}(x,y)=\lim_{R\rightarrow\infty}2^{\ell |\alpha|}\int_{\R^n}T((D^\alpha\psi)_{\ell}^{y})(u)\(\psi_k^x(u)-J_y^L[\psi_k^x](u)\)(\eta_R(u)-\eta_{2^{2-\ell}}(u-y))du,
\end{align*}
where $\eta_R\in\mathcal D_P$ with $\eta_R(x)=\eta(x/R)$, $\supp(\eta)\subset B(0,2)$, and $\eta=1$ on $B(0,1)$.  We apply the $WBP_\nu$ property to estimate $A_{\ell,k}$ as follows,
\begin{align*}
|A_{\ell,k}(x,y)|
&=2^{\ell|\alpha|}2^{(\ell+k)n}2^{(L+\gamma)(k-\ell)}\left|\<T\(\frac{(D^\alpha\psi)_{\ell}^{y}}{2^{\ell n}}\),\frac{\psi_k^x-J_y^L[\psi_k^x]}{2^{kn}2^{(L+\gamma)(k-\ell)}}\eta_{2^{2-\ell}}(\cdot-y)\>\right|\\
&\hspace{0cm}\less2^{\ell |\alpha|}2^{\ell \nu}2^{(L+\gamma)(k-\ell)}2^{kn}=2^{k(\nu+|\alpha|)}2^{(L+\gamma-\nu-|\alpha|)(k-\ell)}2^{kn}.
\end{align*}
Note that $\supp\(2^{-\ell n}(D^\alpha\psi)_{\ell}^{y}\)\subset B(y,2^{3-\ell})$ and $||D^\mu\(2^{-\ell n}(D^\alpha\psi)_{\ell}^{y}\)||_{L^\infty}\less2^{|\mu|\ell}$ for all $\mu\in\N_0^n$, where the associated constants are independent of $x,y,\ell,k$.  Similarly, we have $\supp(\eta_{2^{2-\ell}}(\cdot-y))\subset B(y,2^{3-\ell})$ and
\begin{align*}
\left|D^\mu\(\frac{\psi_k^x(u)-J_y^L[\psi_k^x](u)}{2^{kn}2^{(L+\gamma)(k-\ell)}}\eta_{2^{2-\ell}}(u-y)\)\right|\less 2^{\ell|\mu|}
\end{align*}
for all $\mu\in\N_0^n$ as long as $k\leq \ell$, where the associated constant does not depend on $u,x,y,k,\ell$.  Hence it is an appropriate function to apply $WBP_\nu$ here.

Let $\gamma\,',\gamma\,''>0$ such that $\gamma\,'<\gamma\,''<\gamma$.  The $B_{\ell,k}$ term is bounded using the kernel representation of $T$ in the following way
\begin{align*}
|B_{\ell,k}(x,y)|&\leq2^{\ell|\alpha|}\limsup_{R\rightarrow\infty}\int_{|u-y|>2^{1-\ell}}\int_{\R^n}\left|K(u,v)-J_y^M[K(u,\cdot)](v)\right||(D^\alpha\psi)_\ell^y(v) |dv\\
&\hspace{4.5cm}\times\left|\psi_k^x(u)-J_y^L[\psi_k^x](u)\right||\eta_R(u)|du\\
&\hspace{-1.2cm}\less2^{\ell |\alpha|}\sum_{m=1}^\infty\int_{2^{m-\ell}<|u-y|\leq 2^{m+1-\ell}}\(\int_{\R^{n}}\frac{2^{-(M+\gamma)\ell}}{2^{(n+\nu+M+\gamma)(m-\ell)}}|(D^\alpha\psi)_{\ell}^{y}(v)|dv\) 2^{kn}(2^k2^{m-\ell})^{L+\gamma\,''}du\\
&\hspace{-1.2cm}\less2^{k(n+\nu+|\alpha|)}2^{(L+\gamma\,''-\nu-|\alpha|)(k-\ell)}\sum_{m=1}^\infty 2^{-(\nu+M+\gamma-L-\gamma\,'')m}\\
&\hspace{-1.2cm}\less2^{k(n+\nu+|\alpha|)}2^{(L+\gamma\,''-\nu-|\alpha|)(k-\ell)}.
\end{align*}
Here we used that $\nu+M+\gamma>L+\gamma\,''$ since $L\leq M+\nu$ and $\gamma\,''<\gamma$.  It is not crucial here that we took $\gamma\,'<\gamma\,''<\gamma$, but this estimate will be used again later where our choice of $\gamma\,'<\gamma\,''$ will be important.  It follows that
\begin{align*}
\sum_{\ell\geq k}2^{\ell|\alpha|}\left|\<T((D^\alpha\psi)_{\ell}^{y}),\psi_k^x\>\right|\less2^{k(n+\nu+|\alpha|)}\less2^{k(\nu+|\alpha|)}\Phi_k^{n+\nu+M+\gamma}(x-y).
\end{align*}
Here we use that $\alpha\in\N_0^n$ must satisfy $|\alpha|\leq L-\nu$, which implies that $\nu+|\alpha|<L+\gamma\,''$.  This verifies that the limit of \eqref{furtherdecomp} as $N\rightarrow\infty$ exists for $|x-y|\leq2^{3-k}$ as well, and that the first term on the right hand side of \eqref{furtherdecomp} satisfies \eqref{sqker1}.  Hence $\theta_k$ is well-defined by \eqref{thetadefn}.  Since $T$ satisfies the $WBP_\nu$, it also follows that
\begin{align*}
|D_y^\alpha P_kT^*\psi_k^x(y)|
&=2^{|\alpha|k}2^{2kn}\left|\<T^*\(\frac{\psi_k^x(u)}{2^{kn}}\),\frac{(D^\alpha\varphi)_k(\cdot-y)}{2^{kn}}\>\right|
\less2^{(n+\nu+|\alpha|)k}.
\end{align*}
Then second term on the right hand side of \eqref{furtherdecomp} also satisfies \eqref{sqker1}.  Therefore $\theta_k$ satisfies \eqref{sqker1}.

We also verify the $\gamma\,'$-H\"older regularity estimate \eqref{sqker2}:  let $\alpha\in\N_0^n$ with $|\alpha|=L$.  It trivially follows from the estimates already proved that
\begin{align*}
&\sum_{\ell\geq k:\;2^{-\ell}<|h|}\left|\<D_y^\alpha T(\psi_{\ell}^{y+h}-\psi_{\ell}^{y}),\psi_k^x\>\right|\\
&\less2^{k(\nu+|\alpha|)}(2^k|h|)^{\gamma\,'}\(\Phi_k^{n+\nu+M+\gamma}(x-y)+\Phi_k^{n+\nu+M+\gamma}(x-y-h)\)\sum_{\ell\geq k:\;2^{-\ell}<|h|}2^{(L+\gamma\,''-\gamma\,'-\nu-|\alpha|)(k-\ell)}\\
&\less2^{k(\nu+|\alpha|)}(2^k|h|)^{\gamma\,'}\(\Phi_k^{n+\nu+M+\gamma}(x-y)+\Phi_k^{n+\nu+M+\gamma}(x-y-h)\).
\end{align*}
Note that $\nu+|\alpha|<L+\gamma\,''-\gamma\,'$ since we chose $\gamma\,'<\gamma\,''$.  On the other hand, for the situation where $|h|\leq2^{-\ell}$, we consider
\begin{align*}
&\sum_{\ell\geq k:\;2^{-\ell}\geq|h|}\left|\<D_1^\alpha T(\psi_{\ell}^{y+h}-\psi_{\ell}^{y}),\psi_k^x\>\right|
\leq |A_{\ell,k}(x,y,h)|+|B_{\ell,k}(x,y,h)|,
\end{align*}
where
\begin{align*}
&A_{\ell,k}(x,y,h)=2^{\ell|\alpha|}\int_{\R^n}T((D^\alpha\psi)_{\ell}^{y+h}-(D^\alpha\psi)_{\ell}^{y})(u)\(\psi_k^x(u)-J_y^L[\psi_k^x](u)\)\eta_{2^{2-\ell}}(u-y)du\;\;\text{ and }\\
&B_{\ell,k}(x,y,h)=\lim_{R\rightarrow\infty}2^{\ell|\alpha|}\int_{\R^n}T((D^\alpha\psi)_{\ell}^{y+h}-(D^\alpha\psi)_{\ell}^{y})(u)\(\psi_k^x(u)-J_y^L[\psi_k^x](u)\)\\
&\hspace{10cm}\times(\eta_R(u)-\eta_{2^{2-\ell}}(u-y))du.
\end{align*}
Note that when $|h|\leq2^{-\ell}$, the function $ 2^{-\ell n}(2^\ell|h|)^{-\gamma\,'}\[ (D^\alpha\psi)_{\ell}^{y+h}-(D^\alpha\psi)_{\ell}^{y}\]\in\mathcal D_P$ is supported in
\begin{align*}B(y,2^{-\ell})\cup B(y+h,2^{-\ell})\subset B(y,|h|+2^{-\ell})\subset B(y,2^{1-\ell})
\end{align*}
with
\begin{align*}
\left|\left|2^{-\ell n}(2^\ell|h|)^{-\gamma\,'} D^\mu\[(D^\alpha\psi)_{\ell}^{y+h}-(D^\alpha\psi)_{\ell}^{y}\]\right|\right|_{L^\infty}\less 2^{\ell |\mu|}
\end{align*}
for $|\mu|\leq N$, where the constant does not depend on $y$, $h$, or $\ell$; here $N$ is the integer specified in the $WBP_\nu$ condition for $T$.  Then using the $WBP_\nu$ for $T$, we have have
\begin{align*}
|A_{\ell,k}(x,y,h)|
&\leq2^{\ell|\alpha|}(2^\ell|h|)^{\gamma\,'}2^{kn}2^{(k-\ell)(L+\gamma)}2^{\ell n}\\
&\hspace{2cm}\times\left|\<T\(\frac{(D^\alpha\psi)_{\ell}^{y+h}-(D^\alpha\psi)_{\ell}^{y}}{2^{\ell n}(2^\ell|h|)^{\gamma\,'}}\),\frac{\psi_k^x-J_y^L[\psi_k^x]}{2^{kn}2^{(k-\ell)(L+\gamma)}}\eta_{2^{2-\ell}}(\cdot-y)\>\right|\\
&\hspace{0cm}\less2^{\ell(\nu+|\alpha|)}(2^\ell|h|)^{\gamma\,'}2^{kn}2^{(L+\gamma)(k-\ell)}=2^{k(n+\nu+|\alpha|)}2^{(\gamma-\gamma\,'-\nu)(k-\ell)}(2^k|h|)^{\gamma\,'}.
\end{align*}
Recall the selection of $\gamma\,''$ such that $0<\gamma\,'<\gamma\,''<\gamma$.  The $B_{\ell,k}$ term is bounded using the kernel representation of $T$
\begin{align*}
|B_{\ell,k}(x,y,h)|&\leq2^{\ell|\alpha|}\int_{|u-y|>2^{1-\ell}}\int_{\R^{n}}|K(u,v)-J_y^M[K(u,\cdot)](v)|\\
&\hspace{1cm}\times|(D^\alpha\psi)_{\ell}^{y+h}(v)-(D^\alpha\psi)_\ell^{y}(v)|\,|\psi_k^x(u)-J_y^L[\psi_k^x](u)|du\,dv\\
&\hspace{0cm}\less2^{\ell|\alpha|}\sum_{m=1}^\infty\int_{2^{m-\ell}<|u-y|\leq2^{m+1-\ell}}\int_{\R^{n}}\frac{2^{-(M+\gamma)\ell}}{2^{(n+\nu+M+\gamma)(m-\ell)}}(2^{\ell}|h|)^{\gamma\,'}\\
&\hspace{1cm}\times\(\Phi_{\ell}^{n+1}(y-v)+\Phi_{\ell}^{n+1}(y+h-v)\)dv\,2^{kn}(2^k|u-y|)^{L+\gamma\,''}du\\
&\hspace{0cm}\less2^{k(n+\nu+|\alpha|)}(2^{k}|h|)^{\gamma\,'}2^{(L-\nu-|\alpha|+\gamma\,''-\gamma\,')(k-\ell)}\sum_{m=1}^\infty 2^{-(\nu+M+\gamma-L-\gamma\,'')m}    \\
&\hspace{0cm}\less2^{k(n+\nu+|\alpha|)}(2^{k}|h|)^{\gamma\,'}2^{(L-\nu-|\alpha|+\gamma\,''-\gamma\,')(k-\ell)}
\end{align*}
Again we use that $\nu+M+\gamma>L+\gamma\,''$.  Then it follows that
\begin{align*}
\sum_{\ell=k}^\infty|A_{\ell,k}(x,y,y')|+|B_{\ell,k}(x,y,y')|&\less2^{k(n+\nu+|\alpha|)} (2^{k}|h|)^{\gamma\,'}
\end{align*}
since $|\alpha|+\nu+\gamma\,'<L+\gamma\,''$.  This completes the estimate in \eqref{sqker2} for the first term on the right hand side of \eqref{furtherdecomp}.  Now we check that $P_kT^*\psi_k^x(y)$, the second term from the right hand side of \eqref{furtherdecomp}, also satisfies the $\gamma\,'$-H\"older estimate.  The estimate is trivial when $|h|\geq2^{-k}$.  Since $(2^k|h|)^{-\gamma\,'}(\psi_k^{x+h}-\psi_k^x)\in\mathcal D_P$ with the appropriate derivative estimates.  When $|h|\leq2^{-k}$, it follows from the WBP$_\nu$ for $T$ that
\begin{align*}
|D_y^\alpha P_kT^*(\psi_k^{x+h}-\psi_k^x)(y)|
&=(2^k|h|)^{\gamma\,'}2^{|\alpha|k}2^{2kn}\left|\<T^*\(\frac{(\psi_k^{x+h}-\psi_k^x)}{2^{kn}(2^k|h|)^{\gamma\,'}}  \),\frac{(D^\alpha\varphi_k^y)}{2^{kn}}\>\right|\\
&\less(2^k|h|)^{\gamma\,'}2^{(n+\nu+|\alpha|)k}.
\end{align*}
Hence both terms on the right hand side of \eqref{furtherdecomp} satisfy the appropriate estimates, and hence so does $\theta_k(x,y)$.  This completes the proof.
\end{proof}

\begin{corollary}\label{c:ao}
Let $\nu\in\R$, $M,L\geq0$ be integers satisfying $\nu\leq L\leq M+\nu$, $0<\gamma\leq1$, and $T\in SIO_\nu(M+\gamma)$.  Assume that $T^*(x^\alpha)=0$ for $|\alpha|\leq L$ and that $T\in WBP_\nu$.  Fix $\psi,\tilde\psi\in\mathcal D_P$ for $P$ sufficiently large, and define $\lambda_{j,k}(x,y)=\<T^*\psi_j^x,\tilde\psi_k^y\>$ for $j,k\in\Z$ and $x,y\in\R^n$.  Then
\begin{align*}
&|D_x^\alpha D_y^\beta\lambda_{j,k}(x,y)|\less2^{(\nu+|\alpha|)j+|\beta|k}2^{(\widetilde L+\gamma')\min(0,j-k)}\Phi_{\min(j,k)}^{n+\nu+M+\gamma}(x-y)\\
&|D_x^\alpha D_y^\beta\lambda_{j,k}(x+h,y)-D_x^\alpha D_y^\beta\lambda_{j,k}(x,y)|\less 2^{j(\nu+|\alpha|)+k|\beta|}(2^j|h|)^{\gamma\,'-\delta}  2^{\min(0,j-k)(\widetilde L+\delta)}\Phi_{\min(j,k)}^{n+\nu+M+\gamma}(x-y)
\end{align*}
for all $\alpha,\beta\in\N_0^n$, $0\leq\delta\leq\gamma\,'<\gamma$ and $x,y,h\in\R^n$ satisfying $|h|<(2^{-\min(j,k)}+|x-y|)/2$, where $\widetilde L=\lfloor L-\nu\rfloor$.
\end{corollary}

\begin{proof}
This follows immediately by applying Theorem \ref{t:truncation} with $(D^\alpha\psi)_j*Tf$ in place of $\psi_k*Tf$.  When $k\geq j$, we have
\begin{align*}
|D_x^\alpha D_y^\beta\lambda_{j,k}(x,y)|&\leq2^{j|\alpha|+k|\beta|}\left|\int_{\R^n}\(\theta_j(x,u)-J_y^{\widetilde L}[\theta_j(x,\cdot)](u)\)(D^\beta\tilde\psi)_k^y(u) du\right|\\
&\hspace{-2cm}\less2^{j|\alpha|+k|\beta|}\int_{\R^n}2^{\nu j}(2^j|u-y|)^{\widetilde L+\gamma\,'}\(\Phi_j^{n+\nu+M+\gamma}(x-u)+\Phi_j^{n+\nu+M+\gamma}(x-y)\)|(D^\beta\tilde\psi)_k^y(u)| du\\
&\hspace{-2cm}\less2^{\nu j}2^{j|\alpha|+k|\beta|}2^{(\widetilde L+\gamma\,')(j-k)}\Phi_{\min(j,k)}^{n+\nu+M+\gamma}(x-y).
\end{align*}
When $k<j$, we have
\begin{align*}
|D_x^\alpha D_y^\beta\lambda_{j,k}(x,y)|&\leq2^{j|\alpha|+k|\beta|}\int_{\R^n}|T^*(D^\alpha\psi)_j^x(u)(D^\beta\tilde\psi)_k^y(u)| du\\
&\less2^{\nu j}2^{j|\alpha|+k|\beta|}\Phi_{\min(j,k)}^{n+\nu+M+\gamma}(x-y).
\end{align*}
If $|h|\geq2^{-j}$, then the second estimate trivially follows from the first.  When $|h|\leq2^{-j}$, we have
\begin{align*}
|D_x^\alpha D_y^\beta\lambda_{j,k}(x+h,y)-D_x^\alpha D_y^\beta\lambda_{j,k}(x,y)|
&\hspace{0cm}\less 2^{j|\alpha|+k|\beta|}  \int_{\R^n}2^{\nu j}  (2^j|h|)^{\gamma\,'}(2^j|u-y|)^{\widetilde L+\gamma}|(D^\beta\tilde\psi)_k^y(u)|\\
&\hspace{1.75cm}\times\(\Phi_j^{n+\nu+M+\gamma}(x-u)+\Phi_j^{n+\nu+M+\gamma}(x-y)\)du\\
&\hspace{0cm}\less 2^{j(\nu+|\alpha|)+k|\beta|}(2^j|h|)^{\gamma\,'}  2^{(\widetilde L+\gamma)\min(0,j-k)} \Phi_{\min(j,k)}^{n+\nu+M+\gamma}(x-y).
\end{align*}
Combining this with the first estimate yields the second one.
\end{proof}

\section{A Restricted Operator Calculus}\label{Sect4}

In this section, we prove two operator calculus type results, in Theorems \ref{t:1sidecalc} and \ref{t:calculus}.  The first provides conditions on $T\in SIO_\nu$ of the form $T^*(x^\alpha)=0$ so that $|\nabla|^{-s}T|\nabla|^t\in SIO_{\nu+t-s}$ (technically, this hold modulo polynomials), and the second provides conditions so that $|\nabla|^{-s}T|\nabla|^t\in CZO_{\nu+t-s}$.  It is of particular interest to note that in Theorem \ref{t:calculus}, we can actually conclude that $|\nabla|^{-s}T|\nabla|^t\in CZO_{\nu+t-s}$ while $T$ only belongs to $SIO_\nu$.  We provide some example later that show there are operators $T$ in $SIO_\nu$ but not $CZO_\nu$ such that $|\nabla|^{-s}T|\nabla|^t\in CZO_{\nu+t-s}$.  In fact, we construct two classes of such example, pseudodifferential operators with symbols in the forbidden class $S_{1,1}^0$ and a variant of the Bony paraproduct.

\begin{lemma}\label{l:kernelsum}
For any $M>-n$ such that $N>n+M$ and $x\neq 0$
\begin{align*}
\sum_{j\in\Z}2^{jM}\Phi_j^N(x)\less\frac{1}{|x|^{n+M}}.
\end{align*}
\end{lemma}

\begin{proof}
This estimate is straightforward to prove.  For $M$ and $N$ as above and $x\neq0$,
\begin{align*}
\sum_{j\in\Z}2^{jM}\Phi_j^N(x)&\leq\sum_{j\in\Z:2^j\leq|x|^{-1}}2^{j(n+M)}+\frac{1}{|x|^N}\sum_{j\in\Z:2^j>|x|^{-1}}2^{-j(N-n-M)}\less\frac{1}{|x|^{n+M}}.
\end{align*}
\end{proof}

\begin{theorem}\label{t:1sidecalc}
Let $\nu\in\R$, $M,L\geq0$ be integers satisfying $\nu\leq L\leq M+\nu$, $0<\gamma\leq1$, and $T\in SIO_\nu(M+\gamma)$.  Assume that $T^*(x^\alpha)=0$ for $|\alpha|\leq L$ and that $T\in WBP_\nu$.  Also fix $s,t\in\R$ satisfying $t<\lfloor L-\nu\rfloor+\gamma$, $s>\nu$, and $t-s<n+M+\gamma$.  Then there exists $T_{s,t}\in SIO_{\nu+t-s}(\gamma\,')$ for $0<\gamma\,'<\gamma$ such that $\<|\nabla|^{-s}T|\nabla|^tf,g\>=\<T_{s,t}f,g\>$ for all $f,g\in\S_\infty$.  The kernel $ K_{s,t}(x,y)$ of $ T_{s,t}$ satisfies
\begin{align*}
|D^\alpha D^\beta K_{s,t}(x,y)|&\less\frac{1}{|x-y|^{n+\nu+t-s+|\alpha|+|\beta|}}
\end{align*}
for all $x\neq y$ and $\alpha,\beta\in\N_0^n$ satisfying $|\alpha|<s-\nu$, $|\beta|<\lfloor L-\nu\rfloor+\gamma-t$, and $|\alpha|+|\beta|<n+M+\gamma+s-t$, and
\begin{align*}
|D^\alpha D^\beta K_{s,t}(x+h,y)-D^\alpha D^\beta K_{s,t}(x,y)|&\less\frac{|h|^{\gamma\,'}}{|x-y|^{n+\nu+t-s+|\alpha|+|\beta|+\gamma\,'}}\\
|D^\alpha D^\beta K_{s,t}(x,y+h)-D^\alpha D^\beta K_{s,t}(x,y)|&\less\frac{|h|^{\gamma\,'}}{|x-y|^{n+\nu+t-s+|\alpha|+|\beta|+\gamma\,'}}
\end{align*}
for all $x,y,h\in\R^n$ with $|h|<|x-y|/2$, $0<\gamma\,'<\gamma$, and $\alpha,\beta\in\N_0^n$ satisfying $|\alpha|<s-(\nu+\gamma\,')$, $|\beta|\leq\lfloor L-\nu\rfloor-t+\gamma-\gamma\,'$, and $|\alpha|+|\beta|<n+M+\gamma+s-(t+\gamma\,')$.  Furthermore, $T_{s,t}$ and $T_{s,t}^*$ are continuous from $\S_P$ into $\S'$ for $P$ sufficiently larger, and can be defined
\begin{align*}
\<T_{s,t}f,g\>=\sum_{j,k\in\Z}2^{tk-sj}\int_{\R^{2n}}\<T^*\psi_j^x,\psi_k^y\>(|\nabla|^t\widetilde\psi)_k*f(y)(|\nabla|^{-s}\widetilde \psi)_j*g(x)dx\,dy,
\end{align*}
where $\psi$ and $\widetilde\psi$ as chosen as in Lemma \ref{l:calderon}.
\end{theorem}

\begin{proof}
Let $\psi\in\mathcal D_{P}$ and $\widetilde\psi\in\S_\infty$ be as in Lemma \ref{l:calderon}.  Define $\Lambda_{j,k}=Q_kTQ_j$, whose kernel is given by $\lambda_{j,k}(x,y)=\<T^*\psi_j^x,\psi_k^y\>$ for $j,k\in\Z$ and $x,y\in\R^n$.  For any $f,g\in\S_\infty$, it follows that
\begin{align*}
\<|\nabla|^{-s}T|\nabla|^tf,g\>&=\sum_{j,k\in\Z}\<|\nabla|^{-s}\widetilde Q_jQ_jTQ_k\widetilde Q_k|\nabla|^tf,g\>\\
&=\sum_{j,k\in\Z}2^{tk-sj}\int_{\R^n}Q_jTQ_k\widetilde Q_k^tf(u)\widetilde Q_j^{-s}g(u)du\\
&=\sum_{j,k\in\Z}2^{tk-sj}\int_{\R^{2n}}\(\int_{\R^{2n}}\lambda_{j,k}(u,v)\widetilde\psi_k^t(v-y)\widetilde \psi_j^{-s}(u-x)dv\,du\)g(x)f(y)dy\,dx.
\end{align*}
Here we denote $\widetilde Q_k^tf=(|\nabla|^t\psi)_k*f$ and likewise for $\widetilde Q_j^{-s}$.  For $x\neq y$, define
\begin{align*}
K_{s,t}(x,y)=\sum_{j,k\in\Z}2^{tk-sj}\int_{\R^{2n}}\lambda_{j,k}(u,v)\widetilde\psi_k^t(v-y)\widetilde \psi_j^{-s}(u-x)dv\,du.
\end{align*}
Let $0<\gamma\,'<\gamma$ and $\alpha,\beta\in\N_0^n$ satisfying $|\alpha|<s-\nu$, $|\beta|<\widetilde L+\gamma-t$, and $|\alpha|+|\beta|<n+M+\gamma+s-t$, where $\widetilde L=\lfloor L-\nu\rfloor$.  Then for all $x\neq y$, using Corollary \ref{c:ao}, we have
\begin{align*}
&\sum_{j,k\in\Z}2^{tk-sj}\left|D_x^\alpha D_y^\beta\int_{\R^{2n}}\lambda_{j,k}(u,v)\widetilde\psi_k^t(v-y)\widetilde \psi_j^{-s}(u-x)dv\,du\right|\\
&\hspace{1cm}\leq\sum_{j,k\in\Z}2^{(|\alpha|-s)j+(t+|\beta|)k}\int_{\R^{2n}}|\lambda_{j,k}(u,v)(D^\beta\widetilde\psi)_k^t(v-y)(D^\alpha\widetilde \psi)_j^{-s}(u-x)|dv\,du\\
&\hspace{1cm}\leq\sum_{j,k\in\Z}2^{(\nu+|\alpha|-s)j+(t+|\beta|)k}2^{(\widetilde L+\gamma\,')\min(0,j-k)}\\
&\hspace{3cm}\times\int_{\R^{2n}} \Phi_{\min(j,k)}^{n+\nu+M+\gamma}(u-v)\Phi_k^{n+\nu+M+\gamma}(v-y)\Phi_j^{n+\nu+M+\gamma}(x-u)dv\,du\\
&\hspace{1cm}\leq\sum_{j,k\in\Z}2^{(\nu+|\alpha|-s)j+(t+|\beta|)k}2^{(\widetilde L+\gamma\,')\min(0,j-k)} \Phi_{\min(j,k)}^{n+\nu+M+\gamma}(x-y)\\
&\hspace{1cm}\less\sum_{j,k\in\Z:j\leq k}2^{(\widetilde L+\gamma\,'+\nu+|\alpha|-s)j}2^{(t+|\beta|-\widetilde L-\gamma\,')k} \Phi_j^{n+\nu+M+\gamma}(x-y)\\
&\hspace{3cm}+\sum_{j,k\in\Z:j>k}2^{(\nu+|\alpha|-s)j}2^{(t+|\beta|)k} \Phi_k^{n+\nu+M+\gamma}(x-y)\\
&\hspace{1cm}\less\sum_{j\in\Z}2^{(\nu+|\alpha|+|\beta|+t-s)j} \Phi_j^{n+\nu+M+\gamma}(x-y)+\sum_{k\in\Z}2^{(\nu+|\alpha|+|\beta|+t-s)k} \Phi_k^{n+\nu+M+\gamma}(x-y)\\
&\hspace{1cm}\less\frac{1}{|x-y|^{n+\nu+|\alpha|+|\beta|+t-s}}.
\end{align*}
Here we use that $t+|\beta|<\widetilde L+\gamma$ and $|\alpha|<s-\nu$ to assure that the summations above converge and Lemma \ref{l:kernelsum} to justify the last inequality.  It follows that for $x\neq y$, we have
\begin{align*}
|D_x^\alpha D_y^\beta K_{s,t}(x,y)|&\leq\sum_{j,k\in\Z}2^{tk-sj}\left|D_x^\alpha D_y^\beta\int_{\R^{2n}}\lambda_{j,k}(u,v)\widetilde\psi_k^t(v-y)\widetilde \psi_j^{-s}(u-x)dv\,du\right|\\
&\less\frac{1}{|x-y|^{n+\nu+|\alpha|+|\beta|+t-s}}.
\end{align*}
Now suppose $\alpha,\beta\in\N_0^n$ with $|\alpha|<s-(\nu+\gamma\,')$, $|\beta|<\widetilde L+\gamma-\gamma\,'-t$, and $|\alpha|+|\beta|<n+M+\gamma+s-(t+\gamma\,')$.  For any $x,y,h\in\R^n$ satisfying $|h|<|x-y|/2$, we have
\begin{align*}
&\sum_{j,k\in\Z}2^{tk-sj}\left|D_x^\alpha D_y^\beta\int_{\R^{2n}}\lambda_{j,k}(u,v)\widetilde\psi_k^t(v-y)(\widetilde \psi_j^{-s}(u-x)-\widetilde \psi_j^{-s}(u-x-h))dv\,du\right|\\
&\hspace{1cm}\leq|h|^{\gamma\,'}\sum_{j,k\in\Z}2^{(\nu+|\alpha|-s+\gamma\,')j+(t+|\beta|)k}2^{(\widetilde L+\gamma\,')\min(0,j-k)}\\
&\hspace{3cm}\times\int_{\R^{2n}} \Phi_{\min(j,k)}^{n+\nu+M+\gamma}(u-v)\Phi_k^{n+\nu+M+\gamma}(v-y)\Phi_j^{n+\nu+M+\gamma}(x-u)dv\,du\\
&\hspace{1cm}\less|h|^{\gamma\,'}\sum_{j\in\Z}2^{(\nu+|\alpha|+|\beta|+\gamma\,'+t-s)j} \Phi_j^{n+\nu+M+\gamma}(x-y)\\
&\hspace{3cm}+|h|^{\gamma\,'}\sum_{k\in\Z}2^{(\nu+|\alpha|+|\beta|+\gamma\,'+t-s)k}\Phi_{k}^{n+\nu+M+\gamma}(x-y)\\
&\hspace{1cm}\less\frac{|h|^{\gamma\,'}}{|x-y|^{n+\nu+|\alpha|+|\beta|+\gamma\,'+t-s}}.
\end{align*}
Here we use that $|\beta|<\widetilde L+\gamma-\gamma\,'-t$ and $|\alpha|<s-(\nu+\gamma\,')$ so that the summations above converge, and Lemma \ref{l:kernelsum} for the last line.  Fix $\gamma\,'<\gamma\,''<\gamma$ such that $|\beta|<\widetilde L+\gamma\,''-\gamma\,'-t$.  By a similar argument, but applying Corollary \ref{c:ao} with $\gamma\,''$ in place of $\gamma\,'$, it follows that
\begin{align*}
&\sum_{j,k\in\Z}2^{tk-sj}\left|D_x^\alpha D_y^\beta\int_{\R^{2n}}\lambda_{j,k}(u,v)(\widetilde\psi_k^t(v-y)-\widetilde\psi_k^t(v-y-h))\widetilde \psi_j^{-s}(u-x)dv\,du\right|\\
&\hspace{2cm}\less|h|^{\gamma\,'}\sum_{j,k\in\Z:j\leq k}2^{(\widetilde L+\gamma\,''+\nu+|\alpha|-s)j}2^{(t+|\beta|+\gamma\,'-\widetilde L-\gamma\,'')k} \Phi_j^{n+\nu+M+\gamma}(x-y)\\
&\hspace{4cm}+|h|^{\gamma\,'}\sum_{j,k\in\Z:j>k}2^{(\nu+|\alpha|-s)j+(t+|\beta|+\gamma\,')k}\Phi_{k}^{n+\nu+M+\gamma}(x-y)\\
&\hspace{2cm}\less\frac{|h|^{\gamma\,'}}{|x-y|^{n+\nu+|\alpha|+|\beta|+\gamma\,'+t-s}},
\end{align*}
for which we use that $|\beta|<\widetilde L+\gamma\,''-\gamma\,'-t$ and $|\alpha|<s-\nu$.
\end{proof}

\begin{theorem}\label{t:calculus}
Let $\nu\in\R$, $L$ be an integer with $L\geq|\nu|$, $(L-\nu)_*<\gamma\leq1$, and $M\geq\max(L,L-\nu)$.  If $T\in SIO_\nu(M+\gamma)$ satisfies $WBP_\nu$ and $T^*(x^\alpha)=0$ for all $|\alpha|\leq L$, then for each $\nu<s<\nu+\lfloor L-\nu\rfloor+\gamma$ and $0<t<\lfloor L-\nu\rfloor+\gamma$, there exists $T_{s,t}\in CZO_{\nu+t-s}$ such that $|\nabla|^{-s}T|\nabla|^tf-T_{s,t}f$ is a polynomial for all $f\in\S_\infty$.
\end{theorem}

\begin{proof}
Assume that $T\in SIO_\nu(M+\gamma)$ satisfies $WBP_\nu$ and that there is an integer $P\geq M$ such that $T^*(x^\alpha)=0$ in $\mathcal D_P'$ for all $|\alpha|\leq L$.  Fix $\psi\in\mathcal D_P$ and $\widetilde\psi\in\S_\infty$ as in Lemma \ref{l:calderon}.  Define $\Lambda_{j,k}f=Q_jTQ_k$, whose kernel is given by $\lambda_{j,k}(x)=\<T^*\psi_j^x,\psi_k^y\>$.  Fix $s,t\in\R$ so that $0<s-\nu,t<\widetilde L+\gamma$, and define the operator $T_{s,t}$, which is continuous from $\S$ into $\S'$, by
\begin{align*}
T_{s,t}f(x)=\sum_{j,k\in\Z}2^{tk-js}\widetilde Q_j^{-s}\Lambda_{j,k}\widetilde Q_k^tf(x).
\end{align*}
By Theorem \ref{t:1sidecalc}, it follows that $T_{s,t}\in SIO_{\nu+t-s}(\gamma\,')$ for all $0<\gamma\,'<\gamma$ and that $|\nabla|^{-s}T|\nabla|^tf-T_{s,t}f$ is a polynomial for all $f\in\S_\infty$.  By Corollary \ref{c:ao}, it follows that
\begin{align*}
|Q_\ell\widetilde Q_j^{-s}\Lambda_{j,k}f(x)|\less 2^{\nu j}2^{-K|\ell-j|}2^{(\widetilde L+\gamma\,')\min(0,j-k)}\mathcal Mf(x),
\end{align*}
where we choose $\gamma\,'$ so that $\max(0,s-\nu-\widetilde L,t-\widetilde L)<\gamma\,'<\gamma$ and $K>0$.  Then for all $f\in\S_P$ and $1<p,q<\infty$, it follows that
\begin{align*}
||T_{s,t}f||_{\dot F_p^{0,q}}&\leq\left|\left|\(\sum_{\ell\in\Z}\[\sum_{j,k\in\Z}2^{tk-js}|Q_\ell\widetilde Q_j^{-s}\Lambda_{j,k}\widetilde Q_k^tf|\]^q\)^{1/q}\right|\right|_{L^p}\\
&\less\left|\left|\(\sum_{\ell\in\Z}\[\sum_{j,k\in\Z}2^{tk-js}2^{\nu j}2^{-K|\ell-j|}2^{(\widetilde L+\gamma\,')\min(0,j-k)}\mathcal M(\widetilde Q_k^tf)\]^q\)^{1/q}\right|\right|_{L^p}\\
&\less\left|\left|\(\sum_{\ell,j,k\in\Z}2^{q(\nu+t-s)k}2^{(s-\nu) (k-j)}2^{-K|\ell-j|}2^{(\widetilde L+\gamma\,')\min(0,j-k)}\[\mathcal M(\widetilde Q_k^tf)\]^q\)^{1/q}\right|\right|_{L^p}\\
&\less\left|\left|\(\sum_{k\in\Z}2^{q(\nu+t-s)k}\[\mathcal M(\widetilde Q_k^tf)\]^q\)^{1/q}\right|\right|_{L^p}\less||f||_{\dot F_p^{\nu+t-s,q}}.
\end{align*}
Therefore $T_{s,t}$ can be extended to a bounded linear operator from $\dot F_p^{\nu+t-s,q}$ into $\dot F_{p}^{0,q}$ by density.  In particular taking $q=2$, $T_{s,t}$ can be extended to a bounded linear operator from $\dot W^{\nu+t-s,p}$ into $L^p$ for all $1<p<\infty$.  Then we have verified that $T_{s,t}\in CZO_{\nu+t-s}(\gamma\,')$ for all $0<\gamma\,'<\gamma$.
\end{proof}

\section{Boundedness of Singular Integral Operators}\label{Sect5}

In this section, we show that $T^*(x^\alpha)=0$ conditions for $T\in SIO_\nu$ are sufficient for several boundedness properties.  The first two results, in Theorems \ref{t:BesovBounds} and \ref{t:TriebelLizorkinBounds}, are on negative smoothness Besov and Triebel-Lizorkin spaces, and they are proved directly with the help of Corollary \ref{c:ao} and Theorem \ref{t:1sidecalc}.  The last two boundedness results, in Corollaries \ref{c:dual1} and \ref{c:dual2}, are consequences of Theorems \ref{t:BesovBounds} and \ref{t:TriebelLizorkinBounds} by duality.  To aid in the discussion of the boundedness of these spaces, we describe some geometric properties of the parameter space related to the Triebel-Lizorkin space estimates for $T$.  Another interesting feature of the results in this section is that cancellation properties $T^*(x^\alpha)=0$ for $T$ allow for us to conclude boundedness on certain weighted Triebel-Lizorkin spaces $\dot F_{p,w}^{s,q}$, where the weight $w$ does not belong to $A_p$.  This type of behavior for operators was observed in \cite{LZ,HO} in relation to Hardy spaces.  We should also note that the boundedness results proved in this section are also sufficient for $T^*(x^\alpha)=0$ conditions.  This is addressed in the next section.

\begin{theorem}\label{t:BesovBounds}
Let $\nu\in\R$, $L$ be an integer with $L\geq|\nu|$, $M\geq\max(L,L-\nu)$, $(L-\nu)_*<\gamma\leq1$, and $T\in SIO_\nu(M+\gamma)$ satisfy $WBP_\nu$.  If $T^*(x^\alpha)=0$ for all $|\alpha|\leq L$, then $T$ can be extended to a bounded operator from $\dot B_{p,w}^{\nu-t,q}$ into $\dot B_{p,w}^{-t,q}$ for all $1<p<\infty$, $0<q<\infty$, $w\in A_p$, and $\nu<t<\nu+\lfloor L-\nu\rfloor+\gamma$.
\end{theorem}

\begin{proof}
We denote $\widetilde L=\lfloor L-\nu\rfloor$.  Fix $\nu<t<\nu+\widetilde L+\gamma$, $1<p<\infty$, and $0<\gamma\,'<\gamma$ such that $t<\nu+\widetilde L+\gamma\,'$.  Let $T_t$ be as in Theorem \ref{t:1sidecalc} such that $|\nabla|^{-t}T|\nabla|^tf-T_tf$ is a polynomial for all $f\in\S_\infty$.  It follows that
\begin{align*}
\widetilde Q_\ell T_{t}f(x)&=\sum_{j,k\in\Z}2^{t(k-j)}\widetilde Q_\ell\widetilde Q_j^{-t}\Lambda_{j,k}\widetilde Q_k^tf(x).
\end{align*}
For any $K>0$ and $\epsilon>0$ sufficiently small, it follows that
\begin{align*}
|\widetilde Q_\ell\widetilde Q_j^{-s}\Lambda_{j,k}f(x)|&\less 2^{\nu j}2^{-K|\ell-j|}\mathcal Mf(x)\\
|\widetilde Q_\ell\widetilde Q_j^{-s}\Lambda_{j,k}f(x)|&\less 2^{\nu j}2^{(\widetilde L+\gamma\,'+\epsilon)\min(0,j-k)}\mathcal Mf(x).
\end{align*}
The first estimate here is a standard almost orthogonality estimate, and the second is obtained by applying Corollary \ref{c:ao}.  Then it follows that
\begin{align*}
|\widetilde Q_\ell\widetilde Q_j^{-s}\Lambda_{j,k}f(x)|
&\less 2^{\nu (j-k)}2^{-\tilde K|\ell-j|}2^{(\widetilde L+\gamma\,')\min(0,j-k)}2^{\nu k}\mathcal Mf(x)
\end{align*}
for some $\tilde K>0$.  Using that $\nu<t<\nu+\widetilde L+\gamma\,'$, for $w\in A_p$ it follows that
\begin{align*}
||T_tf||_{\dot B_{p,w}^{0,q}}
&\less \(\sum_{\ell\in\Z}\[\sum_{j,k\in\Z:j\leq k}2^{-\tilde K|\ell-j|}2^{(\widetilde L+\gamma\,'+\nu-t)(j-k)}2^{\nu k}||\mathcal M(\widetilde Q_k^tf)||_{L^p_w}\]^q\)^{1/q}\\
&\hspace{1cm}+ \(\sum_{\ell\in\Z}\[\sum_{j,k\in\Z:j>k}2^{-\tilde K|\ell-j|}2^{(t-\nu)(k-j)}2^{\nu k}||\mathcal M(\widetilde Q_k^tf)||_{L^p_w}\]^q\)^{1/q}\\
&\less \(\sum_{\ell,j,k\in\Z:j\leq k}2^{-\tilde K|\ell-j|}2^{(\widetilde L+\gamma\,'+\nu-t)(j-k)}\[2^{\nu k}||\mathcal M(\widetilde Q_k^tf)||_{L^p_w}\]^q\)^{1/q}\\
&\hspace{1cm}+ \(\sum_{\ell,j,k\in\Z:j>k}2^{-\tilde K|\ell-j|}2^{(t-\nu)(k-j)}\[2^{\nu k}||\mathcal M(\widetilde Q_k^tf)||_{L^p_w}\]^q\)^{1/q}\\
&\less \(\sum_{k\in\Z}\[2^{\nu k}||\mathcal M(\widetilde Q_k^tf)||_{L^p_w}\]^q\)^{1/q}\less||f||_{\dot B_{p,w}^{\nu,q}}.
\end{align*}
Therefore $T_t$ is bounded from $\dot B_{p,w}^{\nu,q}$ into $\dot B_{p,w}^{0,q}$ for all $1<p,q<\infty$, $w\in A_p$, and $\nu<t<\nu+\widetilde L+\gamma$.  Then for all $f\in\S_\infty$, it follows that
\begin{align*}
||Tf||_{\dot B_{p,w}^{-t,q}}&=||\,|\nabla|^{-t}T|\nabla|^t(|\nabla|^{-t}f)||_{\dot B_{p,w}^{0,q}}=||T_t(|\nabla|^{-t}f)||_{\dot B_{p,w}^{0,q}}\less||\,|\nabla|^{-t}f||_{\dot B_{p,w}^{0,q}}=||f||_{\dot B_{p,w}^{-t,q}}.
\end{align*}
Therefore $T$ can be extended to a bounded linear operator on $\dot B_{p,w}^{-t,q}$ for all $1<p,q<\infty$, $w\in A_p$, and $\nu<t<\nu+\widetilde L+\gamma$.
\end{proof}

A similar argument to the one above can be made to show that $T$ is bounded on from $\dot F_{p,w}^{\nu-t,q}$ into $\dot F_{p,w}^{-t,q}$ under the same assumptions on $T$ and for the same ranges of parameters (in particular imposing that $w\in A_p$).  However, we do not pursue this argument since we can prove something stronger for Triebel-Lizorkin spaces, where the weight is allowed to range outside of the $A_p$ class corresponding to the Lebesgue space parameter $p$.  This stronger result is the following.

\begin{theorem}\label{t:TriebelLizorkinBounds}
Let $\nu\in\R$, $L$ be an integer with $L\geq|\nu|$, $M\geq\max(L,L-\nu)$, $(L-\nu)_*<\gamma\leq1$, and $T\in SIO_\nu(M+\gamma)$.  Further assume that $T$ satisfies $WBP_\nu$.  If $T^*(x^\alpha)=0$ for all $|\alpha|\leq L$, then $T$ can be extended to a bounded operator from $\dot F_{p,w}^{\nu-t,q}$ into $\dot F_{p,w}^{-t,q}$ for all $\nu<t<\nu+\lfloor L-\nu\rfloor +\gamma$, $1/\lambda<p<\infty$, $\min(1,p)\leq q<\infty$, and $w\in A_{\lambda p}$ where $\lambda =\frac{n+\nu+\lfloor L-\nu\rfloor+\gamma-t}{n}$.  Furthermore, there is an increasing function $N:[1,\infty)\rightarrow(0,\infty)$ (possibly depending on $\nu$, $p$, $t$, and $q$) such that, for the same range of indices, we have
\begin{align*}
||Tf||_{\dot F_{p,w}^{-t,q}}\leq N([w]_{A_p})||f||_{\dot F_{p,w}^{\nu-t,q}}.
\end{align*}
\end{theorem}

\begin{proof}
Fix $\nu<t<\nu+\widetilde L+\gamma$, $1/\lambda<p<\infty$, and $w\in A_{\lambda p}$.  Here we denote $\widetilde L=\lfloor L-\nu\rfloor$ and $\lambda=\frac{n+\nu+\widetilde L+\gamma-t}{n}$.  Then there exists $1/\lambda<r<\min(1,p)$ such that $w\in A_{\lambda r}$.  Also let $0<\gamma\,'<\gamma$ and $0<\mu<\nu+\widetilde L+\gamma\,'-t$ so that $t<\nu+\widetilde L+\gamma\,'$ and $1/\lambda<\frac{n}{n+\mu}<r<\min(1,p)$.  Let $T_t$ be defined as in Theorem \ref{t:1sidecalc}.  Now we fix $\phi,\widetilde\phi\in\S$ as in Lemma \ref{l:wavelet}.  Then it follows that
\begin{align*}
\widetilde Q_\ell T_{t}f(x)&=\sum_{j,k\in\Z}2^{t(k-j)}Q_\ell\widetilde Q_j^{-t}\Lambda_{j,k}\widetilde Q_k^tf(x)\\
&=\sum_{j,k,m\in\Z}\sum_{Q:\ell(Q)=2^{-m}}2^{t(k-j)}\widetilde Q_\ell \widetilde Q_j^{-t}\Lambda_{j,k}\widetilde Q_k^t\phi_m^{c_Q}(x)\tilde\phi_m*f(c_Q).
\end{align*}
For any any $K>0$ and $\epsilon>0$ sufficiently small, it follows that
\begin{align*}
|\widetilde Q_\ell \widetilde Q_j^{-t}\Lambda_{j,k}\widetilde Q_k^t\phi_m^{c_Q}(x)|&\less2^{\nu j}2^{-K|\ell-j|}\Phi_{\min(\ell,j,k,m)}^{n+M+\gamma}(x-c_Q),\\
|\widetilde Q_\ell \widetilde Q_j^{-t}\Lambda_{j,k}\widetilde Q_k^t\phi_m^{c_Q}(x)|&\less2^{\nu j}2^{-K|k-m|}\Phi_{\min(\ell,j,k,m)}^{n+M+\gamma}(x-c_Q),\\
|\widetilde Q_\ell \widetilde Q_j^{-t}\Lambda_{j,k}\widetilde Q_k^t\phi_m^{c_Q}(x)|&\less2^{\nu j}2^{(\widetilde L+\gamma\,'+\epsilon)\min(0,j-k)}\Phi_{\min(\ell,j,k,m)}^{n+M+\gamma}(x-c_Q).
\end{align*}
As before, the first two lines here follow from standard almost orthogonality estimates, and the third follows from Corollary \ref{c:ao}.  Combining these estimates, it also follows that
\begin{align*}
|Q_\ell \widetilde Q_j^{-t}\Lambda_{j,k}\widetilde Q_k^t\phi_m^{c_Q}(x)|&\less2^{\nu j}2^{-\tilde K|\ell-j|}2^{-\tilde K|k-m|}2^{(L+\gamma\,')\min(0,j-k)}\Phi_{\min(\ell,j,k,m)}^{n+\nu+M+\gamma}(x-c_Q)
\end{align*}
for some $\tilde K>2\mu+|t|$, as long as $K$ is selected sufficiently large, depending on $L$, $\gamma$, and $\nu$.  Then using Lemma \ref{l:Mj1}, it follows that
\begin{align*}
&\sum_{Q:\ell(Q)=2^{-m}}|Q_\ell\widetilde Q_j^{-t}\Lambda_{j,k}\widetilde Q_k^t\phi_m^{c_Q}(x)\tilde\phi_m*f(c_Q)|\\
&\hspace{1cm}\leq 2^{\nu j}2^{-\tilde K|\ell-j|}2^{-\tilde K|k-m|}2^{(\widetilde L+\gamma\,')\min(0,j-k)} 2^{\mu\max(0,m-j,m-k,m-\ell)}\mathcal M_m^r(\tilde\phi_m*f)\\
&\hspace{1cm}\leq2^{\nu j}2^{-(\tilde K-\mu)|\ell-j|}2^{-(\tilde K-2\mu)|k-m|}2^{(\widetilde L+\gamma\,'-\mu)\min(0,j-m)}\mathcal M_m^r(\tilde\phi_m*f).
\end{align*}
Using that $\tilde K>2\mu+|t|$ and that $\nu<t<\nu+\widetilde L+\gamma\,'-\mu$, we also have
\begin{align*}
\sum_{\ell\in\Z}|Q_\ell T_tf|^q&\less\sum_{\ell\in\Z}\[\sum_{j,k,m\in\Z:j\leq m}2^{-(\tilde K-\mu)|\ell-j|}2^{-(\tilde K-2\mu-|t|)|k-m|}2^{(\nu+\widetilde L+\gamma\,'-t-\mu)(j-m)}2^{\nu m}\mathcal M_m^r(\tilde\phi_m*f)\]^q\\
&\hspace{1cm}+\sum_{\ell\in\Z}\[\sum_{j,k,m\in\Z:j>m}2^{(t-\nu)(m-j)}2^{-(\tilde K-\mu)|\ell-j|}2^{-(\tilde K-2\mu-|t|)|k-m|}2^{\nu m}\mathcal M_m^r(\tilde\phi_m*f)\]^q\\
&\less\sum_{\ell,j,k,m\in\Z:j\leq m}2^{-(\tilde K-\mu)|\ell-j|}2^{-(\tilde K-2\mu-|t|)|k-m|}2^{(\nu+\widetilde L+\gamma\,'-t-\mu)(j-m)}\[2^{\nu m}\mathcal M_m^r(\tilde\phi_m*f)\]^q\\
&\hspace{1cm}+\sum_{\ell,j,k,m\in\Z:j>m}2^{(t-\nu)(m-j)}2^{-(\tilde K-\mu)|\ell-j|}2^{-(\tilde K-2\mu-|t|)|k-m|}\[2^{\nu m}\mathcal M_m^r(\tilde\phi_m*f)\]^q\\
&\less\sum_{m\in\Z}\[2^{\nu m}\mathcal M_m^r(\tilde\phi_m*f)\]^q,
\end{align*}
and hence that
\begin{align*}
||T_tf||_{\dot F_{p,w}^{0,q}}&=\left|\left|\(\sum_{\ell\in\Z}|Q_\ell T_tf|^q\)^{1/q}\right|\right|_{L^p_w}\less\left|\left|\(\sum_{m\in\Z}\[2^{\nu m}\mathcal M_m^r(\tilde\phi_m*f)\]^q\)^{1/q}\right|\right|_{L^p_w}\less||f||_{\dot F_{p,w}^{\nu,q}},
\end{align*}
where we use Lemma \ref{l:Mj2} in the last inequality.  Therefore $T_t$ is bounded from $\dot F_{p,w}^{\nu,q}$ into $\dot F_{p,w}^{0,q}$ for all $\nu<t<\nu+\widetilde L+\gamma$, $1/\lambda<p<\infty$, $\min(1,p)\leq q<\infty$, and $w\in A_{\lambda p}$.  Then for the same range of parameters and for all $f\in\S_\infty$, it follows that
\begin{align*}
||Tf||_{\dot F_{p,w}^{-t,q}}&=||\,|\nabla|^{-t}T|\nabla|^t(|\nabla|^{-t}f)||_{\dot F_{p,w}^{0,q}}=||T_t(|\nabla|^{-t}f)||_{\dot F_{p,w}^{0,q}}\less||\,|\nabla|^{-t}f||_{\dot F_{p,w}^{\nu,q}}=||f||_{\dot F_{p,w}^{\nu-t,q}}.
\end{align*}
Therefore $T$ can be extended to a bounded linear operator form $\dot F_{p,w}^{\nu-t,q}$ into $\dot F_{p,w}^{-t,q}$ for all $\nu<t<\nu+\widetilde L+\gamma$, $1/\lambda<p<\infty$, $\min(1,p)\leq q<\infty$, and $w\in A_{\lambda p}$.  It is not hard to note that the dependence on $w\in A_p$ from the argument about yields an estimate depending on a positive power of $[w]_{A_p}$, and hence $N$ can be taken as such (in particular, such a function exists).
\end{proof}

Below we represent $\dot F_p^{t,2}$ in parameter space $(t,\frac{1}{p})$ for $t\in\R$ and $0<p\leq\infty$.  Most of the discussion here will apply for $\dot F_p^{t,q}$ when $q\neq2$, but for simplicity we only discuss restrict to $q=2$.  However, we caution that some of the discussion here and below may not apply when $0<q\leq1$ since the duality of $\dot F_p^{t,q}$ does not behave the same as the $q=2$ case.

With our parameter identification defined, we note that the vertical axis $(t,\frac{1}{p})=(0,\frac{1}{p})$ represents the Hardy spaces $\dot F_p^{0,2}= H^p$ for $0<p<\infty$ and Lebesgue spaces $\dot F_p^{0,2}= L^p$ for $1<p<\infty$.  The horizontal lines given by $(t,\frac{1}{p})$ when $1<p<\infty$ describe the homogeneous Sobolev spaces $\dot F_p^{t,2}= \dot W^{t,p}$.  We will always identify the origin $(t,\frac{1}{p})=(0,0)$ as $\dot F_\infty^{0,\infty}= BMO$, but for $(t,\frac{1}{p})=(t,0)$ our notation is somewhat inconsistent.  Sometimes these locations will represent Sobolev-$BMO$ via $\dot F_\infty^{t,2}= I_t(BMO)$ for $t>0$, and other times they should be interpreted as the Besov-Lipschitz spaces $\dot B_\infty^{t,\infty}$ when $t>0$.  Recall that $\dot B_\infty^{t,\infty}$ is a Lipschitz space only when $t>0$ is not an integer; otherwise it is the Zygmund class of smooth functions; see \cite{Z}.

Let us first consider boundedness results for $T\in SIO_\nu(M+\gamma)$ when $\nu=0$.  The graph on the left in Figure \ref{f:duality} is a depiction of the boundedness properties of $T$ provided in Theorem \ref{t:TriebelLizorkinBounds} (restricted to the unweighted version).  That is, if $T^*(x^\alpha)=0$ for $|\alpha|\leq L$, then $T$ is bounded on $\dot F_p^{t,2}$ when $(t,\frac{1}{p})$ lies in the blue shaded region (excluding the boundary) in the left picture.  If we were to assume in addition that $T$ is $L^2$-bounded, then it also follows that $T$ is bounded on $\dot F_p^{0,2}= H^p$ when $(t,\frac{1}{p})=(0,\frac{1}{p})$ satisfies $\frac{n}{N+L+\gamma}<p<\infty$.  This does not follow from what we've proved here, but these boundedness properties are classical in the Lebesgue space setting when $1<p<\infty$ and were proved in \cite{HartLu1} for $0<p\leq1$.\footnote{Plots appearing in this article were generated using the Desmos.com online graphing tool and Matlab.}

\begin{figure}[h]
\centering
  \includegraphics[width=.95\linewidth]{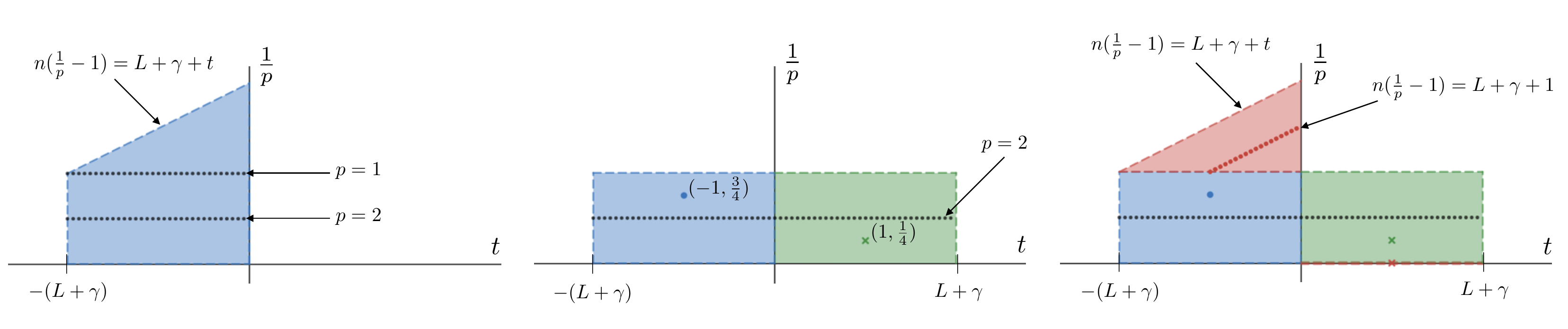}
  \caption{Parameter space $(t,\frac{1}{p})$ for boundedness properties of $T$ and $T^*$  on $\dot F_p^{t,2}$, pictured with $n=2$, $L=1$, and $\gamma=1$.}\label{f:duality}
\end{figure}

The middle picture of Figure \ref{f:duality} restricts the boundedness region for $T$ to where $1<p<\infty$, where the spaces $\dot F_p^{t,2}$ are reflexive.  In this situation, we have $(\dot F_p^{t,2})^*=\dot F_{p'}^{-t,2}$ for $t\in\R$ and $1<p<\infty$.  Geometrically, the dual of $\dot F_p^{t,2}$ can be found by reflecting over the vertical line $t=0$ and the horizontal line $p=2$, as pictured.  Then by duality, $T^*(x^\alpha)=0$ for $|\alpha|\leq L$ implies that $T^*$ is bounded on $\dot F_p^{t,2}$ for all $1<p<\infty$ and $0<t<L+\gamma$.  This boundedness result for $T^*$ is shown in the green shaded region in Figure \ref{f:duality}.

In the picture on the right in Figure \ref{f:duality}, we describe the dual boundedness implications for $T^*$ when $0<p\leq1$, which are more delicate that the ones already discussed.  We consider two situations: where $p=1$ and where $0<p<1$.  It was proved by Frazier and Jawerth \cite{FJ2} that $(\dot F_1^{t,2})^*=\dot F_\infty^{-t,2}$.  Then the boundedness of $T$ on $\dot F_1^{t,2}$ for indices on the horizontal line segment given by $(t,\frac{1}{p})=(t,1)$ with $-(L+\gamma)<t<0$ implies that $T^*$ is bounded on the Sobolev-BMO spaces $\dot F_\infty^{t,2}=I_t(BMO)$ for $0<t<L+\gamma$.  Geometrically, this summarizes boundedness for $T^*$ on the horizontal line $(t,\frac{1}{p})=(t,0)$ with $0\leq t<L+\gamma$, where we make the Triebel-Lizorkin space identification $\dot F_\infty^{t,2}$.  This duality still obeys the geometric rule of reflecting over the lines $t=0$ and $p=2$ to obtain the appropriate indices for dual spaces.

Now we turn our attention to the picture on the right in Figure \ref{f:duality} when $0<p<1$.  The remaining portion of the red region lying to the left of the axis is where $-(L+\gamma)<t<0$ and $\frac{n}{n+L+\gamma-t}<p<1$.  In this situation, we invoke a duality result of Jawerth \cite{Jaw} that says $(\dot F_p^{t,2})^*=\dot B_\infty^{-t+n(1/p-1),\infty}$ for $0<p<1$ and $t\in\R$.  Note that the duals of $\dot F_p^{t,2}$ coincide for several values of $t$ and $p$ here.  In particular, for any $t\in\R$ and $0<p<1$ with $-t+n(1/p-1)=s$ satisfies $(\dot F_p^{t,2})^*=\dot B_\infty^{s,\infty}$.  This is depicted above by the highlighted red line, and the associated red x on the horizontal axis located at $(t,\frac{1}{p})=(1,0)$.  Then by duality $T^*(x^\alpha)=0$ for $|\alpha|\leq L$ implies that $T^*$ is bounded on $\dot B_\infty^{t,\infty}$ for $0<t<L+\gamma$.  This conclusion can be made by duality from the boundedness of $T$ at any point along the appropriate line.  Geometrically, this deviates slightly from the previous cases.  In particular, when $0<p<1$ and $t\in\R$, in order to obtain the appropriate indices for the dual of $\dot F_p^{t,2}$, first project $(t,\frac{1}{p})$ along a line with slope $1/n$ onto the line $p=1$, then reflect over $t=0$ and $p=2$.  Making these geometric manipulations lands the dual indices on the horizontal axis, overlapping with the previous case where $p=1$.  We emphasize that when $0<p<1$, the appropriate interpretation of the pictures above is that $(\dot F_p^{t,2})^*=\dot B_\infty^{-t+n(1/p-1),\infty}$, with this Besov space in place of $\dot F_\infty^{-t,2}$.  Hence the distinction $0<p<1$ versus $p=1$ determines when we interpret the horizontal axis as $\dot F_\infty^{t,2}$ versus $\dot B_\infty^{t,\infty}$.

The estimates for $T$ on $\dot F_p^{t,2}$ indicated in the blue shaded region in the left picture of Figure \ref{f:duality} describes only the the unweighted estimates proved in Theorem \ref{t:TriebelLizorkinBounds}, but we can extend this representation to weighted estimates as well.  In order to do so, consider the parameter space made up of ordered triples of the form $(t,\frac{1}{p},\frac{1}{q})$ for which $T$ is bounded on $\dot F_{p,w}^{t,2}$ when $w\in A_q$.  Theorem \ref{t:TriebelLizorkinBounds} says that the triple $(t,\frac{1}{p},\frac{1}{q})$ represents where $T$ is bounded on the weighted spaces $\dot F_{p,w}^{t,2}$ for all $w\in A_q$ when $\frac{n}{n+L+\gamma-t}\,\frac{1}{p}\leq\frac{1}{q}\leq1$.  This defines a solid in $\R^3$ lying under the blue shaded region on the left picture in Figure \ref{f:duality}, which is shown in Figure \ref{f:duality3d}.

\begin{figure}[h]
\centering
  \includegraphics[width=.85\linewidth]{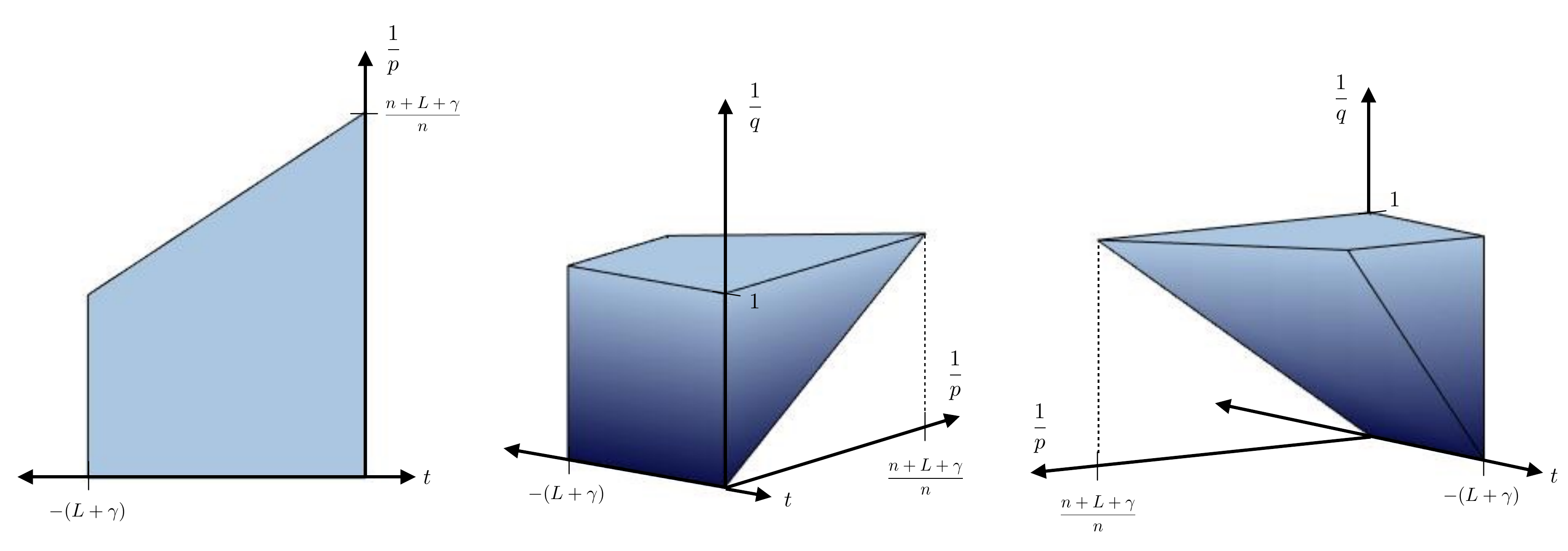}
  \caption{Parameter space $(t,\frac{1}{p},\frac{1}{q})$ for boundedness properties of $T$ on $\dot F_{p,w}^{t,2}$ with $w\in A_q$, pictured with $n=2$, $L=1$, and $\gamma=1$.  The plot on the left depicts the same region as the plot on the left of Figure \ref{f:duality}, which coincides with the $q=1$ cross-section of the middle and right plots.}\label{f:duality3d}
\end{figure}

All of these duality results are made precise in the following corollaries, which we state for a general $\nu\in\R$.  Similar geometric depictions of the boundedness results above for $\nu\neq0$ can be made by tracking the original and/or terminal indices $(t,\frac{1}{p})$ of the boundedness properties of $T$ from $\dot F_p^{t,2}$ into $\dot F_p^{t-\nu,2}$ for $\nu<-t<\nu+\widetilde L+\gamma$ and $\frac{n}{n+\nu+\widetilde L+\gamma+t}<p<\infty$.  Figure \ref{f:duality2} briefly demonstrates how the boundedness of $T$ and $T^*$ is expressed in parameter space when $\nu\neq0$.

\begin{figure}[h]
\centering
  \includegraphics[width=.95\linewidth]{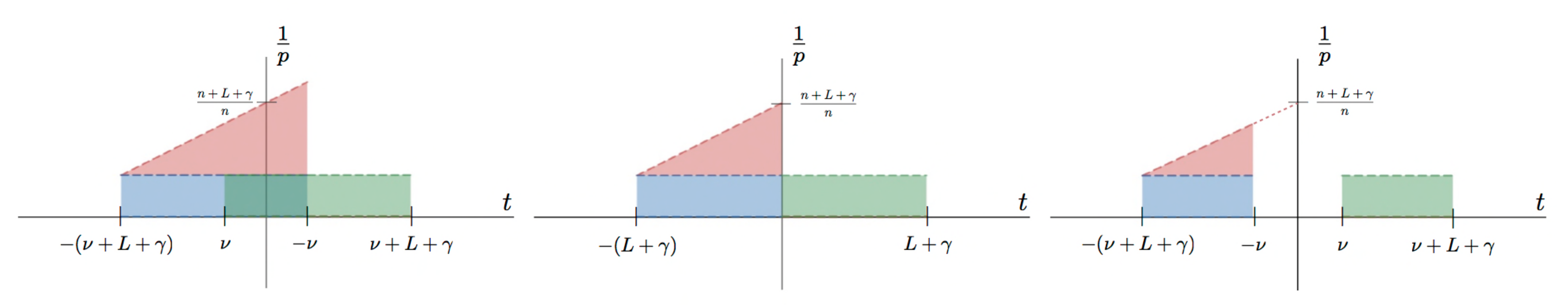}
  \caption{Parameter space for boundedness properties of $T$ from $\dot F^{\nu+t,2}_p$ into $\dot F^{t,2}_p$ and for $T^*$ from $\dot F^{t,2}_p$ into $\dot F^{t-\nu,2}_p$, pictured with $n=2$, $L=3$, and $\gamma=\frac{1}{2}$.  The left, middle, and right plots correspond to $T$ belonging to $SIO_\nu$ for $\nu=-\frac{3}{2},0,\frac{3}{2}$ respectively, which yields $\widetilde L=5,3,\frac{5}{2}$ respectively.}\label{f:duality2}
\end{figure}

\begin{corollary}\label{c:dual1}
Let $\nu\in\R$, $L$ be an integer with $L\geq|\nu|$, $M\geq\max(L,L-\nu)$, $(L-\nu)_*<\gamma\leq1$, and $T\in SIO_\nu(M+\gamma)$.  Further assume that $T$ satisfies $WBP_\nu$.  If $T^*(x^\alpha)=0$ for all $|\alpha|\leq L$, then $T^*$ can be extended to a bounded operator from $\dot F_p^{t,q}$ into $\dot F_p^{t-\nu,q}$ and from $\dot B_p^{t,q}$ into $\dot B_p^{t-\nu,q}$ for all $1<p,q<\infty$ and $\nu<t<\nu+\lfloor L-\nu\rfloor+\gamma$.
\end{corollary}

This corollary is immediate given Theorems \ref{t:BesovBounds} and \ref{t:TriebelLizorkinBounds} applied only in the unweighted and $1<p,q<\infty$ situation.

\begin{corollary}\label{c:dual2}
Let $\nu\in\R$, $L$ be an integer with $L\geq|\nu|$, $M\geq\max(L,L-\nu)$, $(L-\nu)_*<\gamma\leq1$, and $T\in SIO_\nu(M+\gamma)$.  Further assume that $T$ satisfies $WBP_\nu$.  If $T^*(x^\alpha)=0$ for all $|\alpha|\leq L$, then $T^*$ is bounded from $\dot F_\infty^{t,q}$ into $\dot F_\infty^{t-\nu,q}$ and from $\dot B_\infty^{t,\infty}$ into $\dot B_\infty^{t-\nu,\infty}$ for all $\nu< t<\nu+\widetilde L+\gamma$ and $1<q<\infty$.
\end{corollary}

\begin{proof}
For $0<t<\nu+\widetilde L+\gamma$ and $1<q<\infty$, Theorem \ref{t:TriebelLizorkinBounds} implies that $T$ is bounded from $\dot F_1^{\nu-t,q}$ into $\dot F_1^{-t,q}$.  Then by duality (see for example the Frazier and Jawerth article \cite[Theorem 5.13]{FJ2}), it follows that $T^*$ is bounded from $\dot F_\infty^{t,q}$ into $\dot F_\infty^{t-\nu,q}$ for $\nu<t<\nu+\widetilde L+\gamma$ and $1<q<\infty$.  For $\nu\leq t<\nu+\widetilde L+\gamma$, choose $\nu<s<L+\gamma$ and $\frac{n}{n+\nu+\widetilde L+\gamma-s}<p<1$ such that $s+n(1/p-1)=t$.  Then it follows that $T$ is bounded from $\dot F_p^{\nu-s,2}$ into $\dot F_p^{-s,2}$.  So by duality, it follows that $T^*$ is bounded from $\dot B_\infty^{s+n(1/p-1),\infty}$ into $\dot B_\infty^{s-\nu+n(1/p-1),\infty}$.  That is, $T^*$ is bounded from $\dot B_\infty^{t,\infty}$ into $\dot B_\infty^{t-\nu,\infty}$.  Here we use the duality result of Jawerth \cite[Theorem 4.2]{Jaw}.
\end{proof}

\section{Necessity of Vanishing Moment Conditions}\label{Sect6}

In this section, we establish the necessity of the $T^*(x^\alpha)=0$ condition for many boundedness results.  In fact, this provides a type of $T1$ theorem that characterizes necessary and sufficient conditions for $T$ (or $T^*$) to be bounded based on Weak Boundedness Properties and cancellation conditions.  It is interesting to note that $T^*(x^\alpha)=0$ is necessary and sufficient cancellation for many of these results, while we need not require anything on $T(x^\alpha)$.  Some of these implications come from Proposition \ref{p:extrapolation} and Lemma \ref{l:moments}, both of which are interesting in their own right.

\begin{proposition}\label{p:extrapolation}
Let $T$ be an operator, $s,t\in\R$, $0<p_0,q<\infty$, and $\lambda\geq1/p_0$.  Assume that $T$ is bounded from $\dot F_{p,w}^{s,q}$ into $\dot F_{p,w}^{t,q}$, and there is an increasing function $N:\R\rightarrow[1,\infty)$ such that
\begin{align*}
||Tf||_{\dot F_{p_0,w}^{t,q}}\leq N([w]_{A_{\lambda p_0}})||f||_{\dot F_{p_0,w}^{s,q}}
\end{align*}
for all $w\in A_{\lambda p_0}$.  Then $T$ is bounded from $\dot F_{p,w}^{s,q}$ into $\dot F_{p,w}^{t,q}$ for all $1/\lambda<p<\infty$ and $w\in A_{\lambda p}$.
\end{proposition}

\begin{proof}
Define
\begin{align*}
&(F,G)=\(\(\sum_{k\in\Z}(2^{sk}|Q_kf|)^q\)^\frac{1}{q\lambda},\(\sum_{k\in\Z}(2^{tk}|Q_kTf|)^q\)^\frac{1}{q\lambda}\)
\end{align*}
For all $w\in A_{\lambda p_0}$
\begin{align*}
||F||_{L^{\lambda p_0}_w}=||Tf||_{\dot F_{p_0,w}^{t,q}}^{1/\lambda}\leq N([w]_{A_{\lambda p_0}})^{1/\lambda}||f||_{\dot F_{p_0,w}^{s,q}}^{1/\lambda}=N([w]_{A_{\lambda p_0}})^{1/\lambda}||G||_{L^{\lambda p_0}_w}
\end{align*}
We apply extrapolation to the pairs of functions $(F,G)$ indexed by $f\in\S_\infty$.  Note that $1\leq\lambda p_0<\infty$.  Then, by extrapolation, it follows that
\begin{align*}
||G||_{L^r_w}\leq K(w)^{1/\lambda}||F||_{L^r_w}
\end{align*}
for all $1<r<\infty$ and $w\in A_r$, where $K(w)$ is specified in \cite{Duo}.  Therefore
\begin{align*}
||Tf||_{\dot F_{r/\lambda,w}^{t,q}}=||G||_{L^r_w}^{\lambda}\leq K(w)||F||_{L^r_w}^\lambda=K(w) ||f||_{\dot F_{r/\lambda,w}^{s,q}}
\end{align*}
for all $1<r<\infty$ and $w\in A_r$.  Now we simply shift notation to $p=r/\lambda$, and it follows that
\begin{align*}
||Tf||_{\dot F_{p,w}^{t,q}}\leq K(w)||f||_{\dot F_{p,w}^{t,q}}
\end{align*}
for all $1/\lambda<p<\infty$, $w\in A_{\lambda p}$, and $f\in\S_\infty$.  By density, $T$ is bounded from $\dot F_{p,w}^{t,q}$ into $\dot F_{p,w}^{s,q}$ for the same range of indices.
\end{proof}

\begin{remark}
Though Proposition \ref{p:extrapolation} is a relatively trivial application of Rubio de Francia's extrapolation, there are some interesting subtleties that can be observed in this result.  It demonstrates a way to ``trade'' the seemingly unnatural weighted estimates on $L^p_w$ for $w\in A_r$ when $r>p$ for the ability to move the index $p$ below $1$.  In particular, this provides a way to avoid a typical difficulty that arrises in Hardy space theory for indices smaller than $1$.  Suppose you'd like to prove that a given operator $T$ is bounded on $H^{p_0}$ for some $1/2<p_0<1$.  For such $p_0$, the duality theory of $H^{p_0}$ can be cumbersome.  However, by Proposition \ref{p:extrapolation} it is sufficient to prove that $T$ is bounded on $H^2_w= \dot F_{2,w}^{0,2}$ for all $w\in A_{2p_0}$ (note that $H_w^2\neq L^2_w$ for all such $w$ since $p_0>1/2$).  However, $H^2_w$ may be easier to work with since, for example, it is a Banach space rather than only a quasi-Banach space like $H^{p_0}$.  Being able to ``bump up'' the index from $p_0$ to $2$ may also make it possible to use duality arguments that may not be viable for quasi-Banach spaces.
\end{remark}

\begin{lemma}\label{l:moments}
Let $L\geq0$ be an integer and $\gamma>0$.  If $f\in \dot B_p^{-t,\infty}\cap L^1(1+|x|^{L+\gamma})$ for all $0<t<L+\gamma$ and $1< p<\infty$, then
\begin{align*}
\int_{\R^n}f(x)x^\alpha dx=0
\end{align*}
for all $|\alpha|\leq L$.

Consequently, for any $0<q<\infty$, the same conclusion holds if $f\in\dot F_p^{-t,q}\cap L^1(1+|x|^{L+\gamma})$ or $f\in\dot B_p^{-t,q}\cap L^1(1+|x|^{L+\gamma})$ for all $0<t<L+\gamma$ and $1< p<\infty$.
\end{lemma}

\begin{proof}
We proceed by induction.  First assume that $L=0$.  Let $0<t<\gamma$ and $1<p<\infty$ be small enough so that $n/p'<t<n/p'+\gamma$.  Assume $f\in\dot B_p^{-t,\infty}\cap L^1(1+|x|^\gamma)$.  Then for any $\psi\in\S_\infty$, we have
\begin{align*}
||f||_{\dot B_p^{-t,\infty}}
&\geq \sup_{k<0}2^{(n/p'-t)k}||\psi||_{L^p}\left|\int_{\R^n}f(y)dy\right|\\
&\hspace{2cm}- \sup_{k<0}2^{-tk}\[\int_{\R^n}\left|\int_{\R^n}(\psi_k(x-y)-\psi_k(x))f(y)dy\right|^pdx\]^{1/p}.
\end{align*}
The second term above is bounded since we have
\begin{align*}
&\sup_{k<0}2^{(\gamma-t)k}\[\int_{\R^n} \(\int_{\R^n} |y|^\gamma\(\Phi_k^{N}(x-y)+\Phi_k^N(x)\)|f(y)|dy\)^pdx\]^{1/p}\\
&\hspace{1cm}\less \sup_{k<0}2^{(\gamma-t)k}\[\int_{\R^n}  \Phi_k^N*(|y|^\gamma|f|)(x)^pdx\]^{1/p}+||f||_{L^1(|y|^\gamma)}\sup_{k<0}2^{(\gamma-t)k}\[\int_{\R^n} \Phi_k^N(x)^pdx\]^{1/p}\\
&\hspace{1cm}\less||f||_{L^1(|x|^\gamma)}.
\end{align*}
Note that we chose $t$ so that $n/p'<t<n/p'+\gamma$, which implies $2^{(\gamma-t)k}||\Phi_k^N||_{L^p}$ is bounded uniformly in $k$ (as long as $N>n/p$) and that $2^{(n/p'-t)k}$ is unbounded for $k<0$.  Then it follows that $f$ must have integral zero.  Now assume that Lemma \ref{l:moments} holds for all $M\leq L-1$.  Let $0<t<L+\gamma$ and $1<p<\infty$ be small enough so that $n/p'+L<t<n/p'+L+\gamma$.  Assume $f\in\dot B_p^{-t,\infty}\cap L^1(1+|x|^{L+\gamma})$.  Then for any $\psi\in\S_\infty$
\begin{align*}
||f||_{\dot B_p^{-t,\infty}}&\geq \sup_{k<0}2^{-tk}\[\int_{\R^n}\left|\int_{\R^n}J_x^L[\psi_k](y)f(y)dy\right|^pdx\]^{1/p}\\
&\hspace{1cm}- \sup_{k<0}2^{-tk}\[\int_{\R^n}\left|\int_{\R^n}(\psi_k(x-y)-J_x^L[\psi_k](y))f(y)dy\right|^pdx\]^{1/p}.
\end{align*}
Again the second term is bounded above since
\begin{align*}
&\sup_{k<0}2^{-tk}\[\int_{\R^n}\left|\int_{\R^n}(\psi_k(x-y)-J_x^L[\psi_k](y))f(y)dy\right|^pdx\]^{1/p}\\
&\hspace{2cm}\less\sup_{k<0}2^{-tk}\[\int_{\R^n}\(\int_{\R^n} (2^{k}|y|)^{L+\gamma}\(\Phi_k^{N}(x-y)+\Phi_k^N(x)\)|f(y)|dy\)^pdx\]^{1/p}\\
&\hspace{2cm}\less\sup_{k<0}2^{(L+\gamma+n/p'-t)k}||f||_{L^1(|x|^{L+\gamma})}\leq ||f||_{L^1(|x|^{L+\gamma})}.
\end{align*}
We also have, by the inductive hypothesis, that
\begin{align*}
&2^{-tk}\[\int_{\R^n}\left|\int_{\R^n}J_x^L[\psi_k](y)f(y)dy\right|^pdx\]^\frac{1}{p}=2^{(L+n/p'-t)k}\[\int_{\R^n}\left|\sum_{|\alpha|=L} \frac{D^\alpha\psi(x)}{\alpha!}\int_{\R^n}f(y)y^\alpha dy\right|^pdx\]^\frac{1}{p}.
\end{align*}
Since $2^{(L+n/p'-t)k}$ is unbounded for $k<0$, it follows that
\begin{align*}
\int_{\R^n}\left|\sum_{|\alpha|=L} \frac{D^\alpha\psi(x)}{\alpha!}\int_{\R^n}f(y)y^\alpha dy\right|^pdx=0
\end{align*}
for all $\psi\in\S_\infty$, and hence that
\begin{align*}
\int_{\R^n}f(y)y^\alpha dy=0
\end{align*}
for all $|\alpha|=L$.  By induction, this completes the proof when $f\in \dot B_p^{-t,\infty}$.  Note that the remaining properties trivially follow since $\dot F_p^{-t,q}\subset\dot B_p^{-t,\infty}$ and $\dot B_p^{-t,q}\subset\dot B_p^{-t,\infty}$ for all $0<q<\infty$, $1<p<\infty$, and $t\in\R$.
\end{proof}

\begin{remark}
It is known that $H^p$ quantifies vanishing moment properties for its members (see e.g. \cite{GH} by Grafakos and He on weak Hardy spaces), and by Lemma \ref{l:moments} we have vanishing moment properties for negative smoothness index Triebel-Lizorkin and Besov spaces.  In particular, Lemma \ref{l:moments} should be interpreted as follows.  The spaces $\dot F_p^{-t,q}$ for $L\leq t+n(1/p-1)<L+1$ quantify vanishing moment properties for order $|\alpha|=L$ for its members, as described in Lemma \ref{l:moments}.
\end{remark}

\begin{theorem}\label{t:T1sufficient}
Let $\nu\in\R$, $L$ be an integer with $L\geq|\nu|$, $M\geq\max(L,L-\nu)$, $(L-\nu)_*<\gamma\leq1$, and $T\in CZO_\nu(M+\gamma)$.  If any one of the conditions hold, then $T^*(x^\alpha)=0$ for all $|\alpha|\leq L$.
\begin{itemize}
\item[(1)] For every $1<p<\infty$ and $\nu<t<\nu+\widetilde L+\gamma$, there exists $0<q\leq\infty$ such that $T$ is bounded from $\dot F_p^{\nu-t,q}$ into $\dot F_p^{-t,q}$
\item[(2)] For every $1<p<\infty$ and $\nu<t<\nu+\widetilde L+\gamma$, there exists $1<q<\infty$ such that $T^*$ is bounded from $\dot F_p^{t,q}$ into $\dot F_p^{t-\nu,q}$
\item[(3)] For every $1<p<\infty$ and $\nu<t<\nu+\widetilde L+\gamma$, there exists $0<q\leq\infty$ such that $T$ is bounded from $\dot B_p^{\nu-t,q}$ into $\dot B_p^{-t,q}$
\item[(4)] For every $1<p<\infty$ and $\nu<t<\nu+\widetilde L+\gamma$, there exists $1<q<\infty$ such that $T^*$ is bounded from $\dot B_p^{t,q}$ into $\dot B_p^{t-\nu,q}$
\item[(5)] For each $\nu<t<\nu+\widetilde L+\gamma$, there exist $0<q\leq\infty$ and $1/\lambda<p<\infty$ such that $T$ is bounded from $\dot F_{p,w}^{\nu-t,q}$ into $\dot F_{p,w}^{-t,q}$ and there is an increasing function $N:[1,\infty)\rightarrow(0,\infty)$ that does not depend on $w$ such that
\begin{align*}
||Tf||_{\dot F_{p,w}^{-t,q}}\leq N([w]_{A_{\lambda p}})||f||_{\dot F_{p,w}^{\nu-t,q}}
\end{align*}
for all $w\in A_{\lambda p}$, where $\lambda=\frac{n+\nu+\widetilde L+\gamma-t}{n}$
\item[(6)] For every $\nu<t<\nu+\widetilde L+\gamma$, $T^*$ is bounded from $\dot B_\infty^{t,\infty}$ into $\dot B_\infty^{t-\nu,\infty}$.
\item[(7)] For every $\nu<t<\nu+\widetilde L+\gamma$, there exists a $1<q<\infty$ such that $T^*$ is bounded from $\dot F_\infty^{t,q}$ into $\dot F_\infty^{t-\nu,q}$
\item[(8)] For every $\nu<t<\nu+\widetilde L+\gamma$, there exists a $1<q<\infty$ such that $T$ is bounded from $\dot F_1^{\nu-t,q}$ into $\dot F_1^{-t,q}$
\item[(9)]  For each $\nu<s<\nu+\lfloor L-\nu\rfloor+\gamma$ and $0<t<\lfloor L-\nu\rfloor+\gamma$ there exists $T_{s,t}\in CZO_{\nu+t-s}$ such that $|\nabla|^{-s}T|\nabla|^tf-T_{s,t}f$ is a polynomial for all $f\in\S_\infty$.
\end{itemize}
\end{theorem}

\begin{proof}
Assume that (1) holds, and let $\psi\in\mathcal D_{P}$, for $P\in\N_0$ sufficiently large, with $\supp(\psi)\subset B(0,R_0/4)$ for some $R_0>1$.  Note that $T\psi$ is locally integrable by the $T\in CZO_\nu$ assumption.  Also if $x\notin B(0,R_0)$, then it follows that
\begin{align*}
|T\psi(x)|&=\left|\int_{\R^{2n}}\( K(x,y)-J_0^M\[ K(x,\cdot)\](y)\)\psi(y)dy\right|\less\int_{\R^{2n}}\frac{|y|^{M+\gamma}}{|x|^{n+\nu+M+\gamma}}|\psi(y)|dy\less\frac{R_0^{L+\gamma-\nu}||\psi||_{L^1}}{|x|^{n+L+\gamma}}.
\end{align*}
Then it follows that $T\psi\in L^1(1+|x|^{L+\gamma'})$ for any $0<\gamma'<\gamma$ since
\begin{align*}
\int_{\R^n}|T\psi(x)|(1+|x|^{L+\gamma\,'})dx
&\less(1+R_0^{L+\gamma})\int_{|x|\leq R_0}|T\psi(x)|dx+R_0^{\gamma\,'-\nu}.
\end{align*}
This, in addition to (1), says that $T\psi\in \dot B_\infty^{-t,\infty}\cap L^1(1+|x|^{L+\gamma\,'})$ for all $0<t<L+\gamma\,'$ and $1<p<\infty$.  So by Lemma \ref{l:moments}, it follows that $T\psi$ has vanishing moments up to order $L$.  But this means exactly that $T^*(x^\alpha)=0$ for the same $\alpha$'s since
\begin{align*}
\<T^*x^\alpha,\psi\>&=\lim_{R\rightarrow\infty}\int_{\R^n}T\psi(x) \eta_R(x)x^\alpha dx=\int_{\R^n}T\psi(x)x^\alpha dx=0,
\end{align*}
where we use dominated convergence and that $T\psi\in L^1(1+|x|^{L+\gamma\,'})$ to handle the limit in $R$.  Therefore $T^*(x^\alpha)=0$ for all $|\alpha|\leq L$.

By exactly the same argument, it follows that condition (3) also implies $T^*(x^\alpha)=0$ for $|\alpha|\leq L$.  Furthermore, by duality (2) implies (1) and (4) implies (3).  Hence we have shown that any one of the conditions (1)--(4) implies $T^*(x^\alpha)=0$ for $|\alpha|\leq L$.

Assume that (5) holds.  Then by Proposition \ref{p:extrapolation} it follows that $T$ is bounded from $\dot F_{p,w}^{\nu-t,q}$ into $\dot F_{p,w}^{-t,q}$ for all $0<t<\nu+\widetilde L+\gamma$, $1/\lambda<p<\infty$, and $w\in A_{\lambda p}$, where $\lambda=\frac{n+\nu+\widetilde L+\gamma-t}{n}$.  In particular, (5) implies (1) which in turn implies that $T^*(x^\alpha)=0$ for $|\alpha|\leq L$.

Assume that (6) holds.  Let $\alpha\in\N_0^n$ with $|\alpha|\leq L$.  Note that $(L-\nu)_*<\gamma$ implies that $L<\nu+\widetilde L+\gamma$.  Then there is a $t\notin\Z$ such that $\max(\nu,|\alpha|)<t<\nu+\widetilde L+\gamma$.  Also let $0<p<1$ such that $n(1/p-1)=t-\nu$.  Then for $\psi\in\mathcal D_{2P}$, with $P$ sufficiently large, we have
\begin{align*}
|\<T^*(x^\alpha),\psi\>|
&\less\limsup_{R\rightarrow\infty}||\eta_Rx^\alpha||_{\dot B_\infty^{t,\infty}}||\psi||_{H^p}\less\limsup_{R\rightarrow\infty}R^{|\alpha|-t}||\psi||_{H^p}=0.
\end{align*}
Here, we simply note that when $t>0$ is a non-integer, $\dot B_\infty^{t,\infty}$ is the class of $t$-Lipschitz functions, and it easily follows that $||\phi(\cdot/R)||_{\dot B_\infty^{t,\infty}}\less R^{-t_*}$ for any $\phi\in C_0^\infty$, where again $t_*$ is the decimal part of $t$.  Therefore $T^*(x^\alpha)=0$ for $|\alpha|\leq L$.

Assume that (7) holds.  Note that $\dot F_\infty^{t-\nu,q}\subset \dot B_\infty^{t-\nu,\infty}$ and $||f||_{\dot B_\infty^{t-\nu,\infty}}\leq||f||_{\dot F_\infty^{t-\nu,q}}$ for any $\nu<t<\nu+\widetilde L+\gamma$ and $1<q<\infty$.  Then we argue as we did in the previous case.  For $|\alpha|\leq L$, let $\max(\nu,|\alpha|)<t<\nu+\widetilde L+\gamma$ and $\psi\in\mathcal D_{2P}$.  Then
\begin{align*}
|\<T^*(x^\alpha),\psi\>|
&\less\limsup_{R\rightarrow\infty}||\eta_Rx^\alpha||_{\dot F_\infty^{t-\nu,q}}||\psi||_{\dot F_1^{\nu-t,q'}}\leq\limsup_{R\rightarrow\infty}||\eta_Rx^\alpha||_{\dot B_\infty^{t-\nu,\infty}}||\psi||_{\dot F_1^{\nu-t,q'}}=0
\end{align*}
Therefore $T^*(x^\alpha)=0$ for $|\alpha|\leq L$.

By duality (8) implies (7) and by density (9) implies (1).  Hence both (8) and (9) also imply $T^*(x^\alpha)=0$ for $|\alpha|\leq L$.
\end{proof}

\section{Applications}\label{Sect7}

\subsection{Necessary and Sufficient Conditions for Classes of Singular Integral Operators}\label{Sect7.1}

In this section we collect and summarize the results in the preceding three sections to form several equivalent conditions for cancellation and boundedness of operators $T\in CZO_\nu$.  This result is essentially a combination of a few of the operator calculus results from Section \ref{Sect4}, the boundedness results from Section \ref{Sect5}, and the sufficiency for vanishing moments from Section \ref{Sect6}.  As a result we obtain the following $T1$ type necessary and sufficient condition for several boundedness results for $SIO_\nu$.

\begin{corollary}\label{c:CZvequiv}
Let $\nu\in\R$, $L\geq|\nu|$ be an integer, $(L-\nu)_*<\gamma\leq1$, $\widetilde L=\lfloor L-\nu\rfloor$, and $T\in CZO_\nu(L+\gamma)$.  Then the following are equivalent.
\begin{enumerate}
\item $T^*(x^\alpha)=0$ for all $|\alpha|\leq L$
\item For every $\nu<s<\nu+\widetilde L+\gamma$ and $0<t<\widetilde L+\gamma$, there exists $T_{s,t}\in CZO_{\nu+t-s}(\gamma\,')$ for $0<\gamma\,<\gamma$ such that $|\nabla|^{-s}T|\nabla|^tf-T_{s,t}f$ is a polynomial for all $f\in\S_\infty$
\item For all $\nu<t<\nu+\widetilde L+\gamma$, $1/\lambda<p<\infty$, and $\min(1,p)\leq q<\infty$, $T$ is bounded from $\dot F_{p,w}^{\nu-t,q}$ into $\dot F_{p,w}^{-t,q}$ and there is an increasing function $N:[1,\infty)\rightarrow(0,\infty)$ that does not depend on $w$ such that
\begin{align*}
||Tf||_{\dot F_{p,w}^{-t,q}}\leq N([w]_{A_{\lambda p}})||f||_{\dot F_{p,w}^{\nu-t,q}}
\end{align*}
for all $w\in A_{\lambda p}$, where $\lambda=\frac{n+\nu+\widetilde L+\gamma-t}{n}$.
\end{enumerate}

\end{corollary}

This corollary follows immediately from Theorems \ref{t:calculus}, \ref{t:TriebelLizorkinBounds}, and \ref{t:T1sufficient}.  We should also note that one could obtain many other equivalent conditions to put on this list by turning to Theorems \ref{t:TriebelLizorkinBounds} and \ref{t:T1sufficient}, as well as Corollaries \ref{c:dual1} and \ref{c:dual2}.

There is a long history of results along the lines Corollary \ref{c:CZvequiv}.  Several boundedness results for $\nu=0$ similar to the ones proved in Theorems \ref{t:BesovBounds} and \ref{t:TriebelLizorkinBounds} (as well as Corollaries \ref{c:dual1} and \ref{c:dual2}) can be found for example in \cite{AM,FTW,T,FHJW,HH,CMbook,LW,HartLu1,HO} as well as several of the references therein.  However, we note that there do not seem to be many results like Theorem \ref{t:TriebelLizorkinBounds} in the sense that we obtain boundedness on for Triebel-Lizorkin spaces with weights in a Muckenhoupt $A_{\lambda p}$ for $\lambda>1$ class that exceed the Lebesgue space parameter $p$.  The only articles we are aware of where such estimates are proved are \cite{LZ,HO}, where the results are limited to $\nu=0$ order operator acting on Hardy spaces.  Along these lines, when $\nu=0$ one can add more equivalent conditions than the ones already mentioned by involving weighted and unweighted Hardy space boundedness properties; see \cite{AM,HartLu1,HO} for more information on this.

The class of operators $CZO_\nu$ for $\nu\neq0$ have been studied to much lesser extant than the $\nu=0$ order operators that fall within scope of traditional zero-order Calder\'on-Zygmund theory.  The most relevant resource in the literature for non-zero-order operator results of this type is \cite{T}, where several sufficient conditions for an operator in $SIO_\nu$ to be bounded on homogeneous Besov and Triebel-Lizorkin spaces.  However, conditions of the form $T^*(x^\alpha)=0$ and $T1=0$ were assumed in order to prove such boundedness results.  Here we remove the assumption on $T1=0$, show that such boundedness properties are also sufficient for $T^*(x^\alpha)=0$ conditions, and include several other equivalent conditions involving weighted estimates, endpoint Besov and Triebel-Lizorkin spaces, and our restricted operator calculus.  One can also compare the following corollary for $\nu<0$ to the results in \cite{CHO}, but the results here do not imply the ones in \cite{CHO}, nor vice versa.

\subsection{Pseudodifferential Operators}\label{Sect7.2}

In this application we consider the forbidden class of pseudodifferential operators $OpS_{1,1}^0$.  They are defined as follows.  We say $\sigma\in S_{1,1}^0$ if
\begin{align*}
|D_\xi^\alpha D_x^\beta\sigma(x,\xi)|\less(1+|\xi|)^{|\beta|-|\alpha|}.
\end{align*}
for all $\alpha,\beta\in\N_0^n$, and $T_\sigma\in OpS_{1,1}^0$ is the associated operator defined
\begin{align*}
T_\sigma f(x)=\int_{\R^n}\sigma(x,\xi)\widehat f(\xi)e^{ix\xi}d\xi
\end{align*}
for $f\in\S$.  The reason $OpS_{1,1}^0$ is referred to as a forbidden class, or sometimes an exotic class, of operators is because it is not closed under transpose, and they are not necessarily $L^2$-bounded.  However, any $T_\sigma\in OpS_{1,1}^0$ has a standard kernel and is bounded on several smooth function spaces.  For instance, such $T_\sigma$ is bounded on several classes of inhomogeneous Lipschitz, Sobolev, Besov, and Triebel-Lizorkin spaces; see for example \cite{M1,R,Bourd,Horm1,Horm2,T2,St3,CMbook}.  All of these inhomogeneous space estimates are obtained in the absence of vanishing moment assumptions.  On the other hand, Meyer showed that under vanishing moment conditions $T_\sigma(x^\alpha)=0$ for $\sigma\in S_{1,1}^0$, $T_\sigma$ is also bounded on homogeneous Lipschitz and Sobolev spaces; see \cite{M2}.  Our next result provides more estimates along the lines of Meyer's that require vanishing moments for the operator.  We also note that Bourdaud proved a noteworthy result in \cite{Bourd} about the largest sub-algebra of $OpS_{1,1}^0$.  In particular, he showed that the subclass of $OpS_{1,1}^0$ made up of operators $T_\sigma\in OpS_{1,1}^0$ so that $T_\sigma^*\in OpS_{1,1}^0$ is an algebra and that all such operators are $L^2$-bounded.

\begin{corollary}\label{c:S11}
Let $T_\sigma\in OpS_{1,1}^0$ and $L\in\N_0$.  If $T_\sigma^*(x^\alpha)=0$ for all $|\alpha|\leq L$, then Theorem \ref{t:calculus}, Theorem \ref{t:BesovBounds}, Theorem \ref{t:TriebelLizorkinBounds}, Corollary \ref{c:dual1}, and Corollary \ref{c:dual2} can all be applied to $T_\sigma$.  If $T_\sigma(x^\alpha)=0$ for all $|\alpha|\leq L$, then the same results can be applied to $T_\sigma^*$.
\end{corollary}

Note that Corollary \ref{c:S11} does not require, nor imply, that $T_\sigma$ is bounded on $L^2$.  In fact, there are standard constructions of operators to which we can apply Corollary \ref{c:S11} that are not $L^2$-bounded, as is shown in the next example.

It should also be noted here that even though $OpS_{1,1}^0$ is not closed under transposes, Corollary \ref{c:S11} applies to both $T_\sigma$ and its transpose for any $\sigma\in S_{1,1}^0$.  This is because $S_{1,1}^0\subset SIO_0(\infty)$ and $SIO_0(\infty)$ is closed under transposes.  Hence for any $T_\sigma\in OpS_{1,1}^0$, both $T_\sigma,T_\sigma^*\in SIO_0(\infty)$, and so Corollary \ref{c:S11} is even capable of concluding operator estimates for operators that do not belong to $OpS_{1,1}^0$.

\begin{proof}[of Corollary \ref{c:S11}]
It is well known that $\sigma\in S_{1,1}^0$ implies $T_\sigma\in  SIO_0(\infty)$.  That is, it is known that such $T_\sigma$ are continuous from $\S$ into $\S'$, and have a standard functional kernel $K(x,y)$.  It is also easy to show that $|\<T_\sigma f,g\>|\less\|\widehat f\|_{L^1}\|g\|_{L^1}$ for all $f,g\in\S$, and so $T_\sigma$ trivially satisfies $WBP_0$.  Recall that $SIO_0(\infty)$ and $WBP_0$ are closed under transposition, and the corollary easily follows.
\end{proof}

\begin{example}
Let $\psi\in\S_\infty$ be such that $\widehat\psi$ is supported in an annulus, and define
\begin{align*}
\sigma(x,\xi)=\sum_{k\in\Z}e^{-i2^kx}\widehat\psi(2^{-k}\xi),
\end{align*}
as well as the associated pseudodifferential operator $T_\sigma$.  It is known that $\sigma\in S_{1,1}^0$ and hence $T_\sigma\in OpS_{1,1}^0$; see for example \cite{St3} for more details.  It is easy to verify that $(T_\sigma^*)^*(x^\alpha)=T_\sigma(x^\alpha)=0$ for all $\alpha\in\N_0^n$.  Then Corollary \ref{c:S11} can be applied to $T_\sigma^*$.  Furthermore, since $T_\sigma(x^\alpha)=0$ for all $\alpha\in\N_0^n$, the restrictions involving $L$ can be removed entirely and one can allow $t>0$ without bound.  So $T_\sigma$ is bounded, for example, on $\dot F_{p,w}^{t,q}$ for all $1<p,q<\infty$, $t>0$, and $w\in A_\infty$.  Also, for every $s,t>0$, there is an operator $T_{s,t}\in CZO_{t-s}$ such that $|\nabla|^{-s}T_\sigma^*|\nabla|^tf-T_{s,t}f$ is a polynomial for all $f\in\S_\infty$.  In particular, $|\nabla|^{-t}T_\sigma^*|\nabla|^t-P_f$ and $|\nabla|^tT_\sigma|\nabla|^{-t}-\tilde P_f$ are Calder\'on-Zygmund operators in $CZO_0$ for all $t>0$, where $P_f$ and $\tilde P_f$ are polynomials depending on $f$ and $t$.
\end{example}

\subsection{Paraproducts}\label{Sect7.3}

In this section, we consider a generalization of the Bony paraproduct, constructed originally in \cite{B}.  The crucial properties of this operator are, for a given $b\in BMO$, the Bony paaproduct $\Pi_b$ is a Calder\'on-Zygmund operator (in particualr $L^2$-bounded), $\Pi_b1=b$, and $\Pi_b^*1=0$.  This operator played a crucial role in the proof of the $T1$ theorem of David and Journ\'e \cite{DJ}, and it has appear in many other places in various forms.

In this section, we construct paraproducts $\Pi_b^\alpha\in SIO_\nu$ for $b\in \dot B_\infty^{|\alpha|-\nu}$, $\alpha\in\N_0^n$, and $\nu\in\R$.  They satisfy prescribed polynomial moment conditions, including $(\Pi_b^\alpha)^*(x^\alpha)=0$ conditions, and hence are bounded on several negative smoothness distribution spaces.  However, they need not (and some in fact do not) belong to $CZO_\nu$ or satisfy any continuous mapping properties into Lebesgue spaces.  See the Corollary \ref{c:paraproducts}, Lemma \ref{l:paraproducts}, and the remarks at the end of Sections \ref{Sect7.3} and \ref{Sect7.4} for more on this.

Let $\nu\in\R$, $\alpha\in\N_0^n$, and $b\in \dot B_\infty^{|\alpha|-\nu,\infty}$.  Also let $\widetilde \psi$ and $\psi$ be as in Lemma \ref{l:calderon}, and $\varphi\in\S$ with integral $1$.  Define the paraproduct operator
\begin{align}\label{Pib}
\Pi_b^\alpha f(x)=\frac{(-1)^{|\alpha|}}{\alpha!}\sum_{k\in\Z}\widetilde Q_k(Q_kb\cdot P_kD^\alpha f)(x).
\end{align}
We can apply our results to these paraproducts as well, as is shown in the next corollary.

\begin{corollary}\label{c:paraproducts}
Let $\nu\in\R$, $\alpha\in\N_0^n$, and $b\in \dot B_\infty^{|\alpha|-\nu,\infty}$.  Then $\Pi_b^\alpha\in SIO_\nu(\infty)$ satisfies $WBP_\nu$ and $(\Pi_b^\alpha)^*(x^\beta)=0$ for all $\beta\in\N_0^n$, where $\Pi_b^\alpha$ is as in \eqref{Pib}.  Hence Theorem \ref{t:calculus}, Theorem \ref{t:BesovBounds}, Theorem \ref{t:TriebelLizorkinBounds}, Corollary \ref{c:dual1}, and Corollary \ref{c:dual2} can all be applied to $\Pi_b^\alpha$ with any number of vanishing moments.
\end{corollary}

\begin{proof}
It is trivial to see that $\Pi_b$ is continuous from $\S_P$ into $\S'$ for an appropriately chosen $P\in\N$.  Indeed, taking $M\in\N$ to be an even integer larger than $|\nu|$, and $g\in \S$, we have
\begin{align*}
|\<\Pi_bf,g\>|
&\leq\|b\|_{\dot B_\infty^{|\alpha|-\nu,\infty}}\sum_{k\in\Z}\int_{\R^n}2^{(M+\nu) k}|(|\nabla|^MD^\alpha\varphi)_k* (|\nabla|^{-M}f)(x)\widetilde Q_kg(x)|dx\\
&\less\|b\|_{\dot B_\infty^{|\alpha|-\nu,\infty}}\|\,|\nabla|^{-M}f\|_{L^2}\sum_{k\in\Z}2^{(M+\nu) k}\|\widetilde Q_kg\|_{L^2}\less\|b\|_{\dot B_\infty^{|\alpha|-\nu,\infty}}\|f\|_{\dot W^{-M,2}}\|g\|_{\dot B_2^{M+\nu,1}}
\end{align*}
as long as $M>|\nu|$ (which assures that $\S\subset\dot B_2^{M+\nu,1}$ since $M+\nu>0$).  For $f\in\S$ and $g\in \S_P$
\begin{align*}
|\<\Pi_bf,g\>|
&\leq\|b\|_{\dot B_\infty^{|\alpha|-\nu,\infty}}\sum_{k\in\Z}\int_{\R^n}2^{\nu k}|(D^\alpha\varphi)_k* f(x)\widetilde Q_kg(x)|dx\\
&\less\|b\|_{\dot B_\infty^{|\alpha|-\nu,\infty}}\|f\|_{L^2}\sum_{k\in\Z}2^{\nu k}\|\widetilde Q_kg\|_{L^2}\less\|b\|_{\dot B_\infty^{|\alpha|-\nu,\infty}}\|f\|_{L^2}\|g\|_{\dot B_2^{\nu,1}}.
\end{align*}
Here we choose $P\in\N$ large enough so that $\mathcal D_P\subset \dot W^{-M,2}\cap \dot B_2^{\nu,1}$.  Note that this is also sufficient to show that $\Pi_b^\alpha$ and $(\Pi_b^\alpha)^*$ both satisfy $WBP_\nu$.  The kernel of $\Pi_b$ is
\begin{align*}
\pi_b^\alpha (x,y)&=\frac{(-1)^{|\alpha|}}{\alpha!}\sum_{k\in\Z}2^{k|\alpha|}\int_{\R^n}\widetilde\psi_k(x-u)Q_kb(u) (D^\alpha\varphi)_k(u-y)du.
\end{align*}
For $\beta,\mu\in\N_0^n$ and $x\neq y$, it follows that
\begin{align*}
|D_x^\beta D_y^\mu\pi_b^\alpha (x,y)|&\less\|Q_kb\|_{\dot B_\infty^{|\alpha|-\nu}}\sum_{k\in\Z}2^{(\nu+|\beta|+|\mu| )k}\Phi_k^{n+\nu+|\beta|+|\mu|+1}(x-y)\\
&\less\|Q_kb\|_{\dot B_\infty^{|\alpha|-\nu}}\frac{1}{|x-y|^{n+\nu+|\beta|+|\mu|}}.
\end{align*}
Therefore $\Pi_b\in SIO_\nu(\infty)$.  It is easy to see that $(\Pi_b^\alpha)^*(x^\alpha)=0$ since $\widetilde\psi_k\in\S_\infty$.
\end{proof}

\begin{remark}
It is worth noting that $\Pi_b^\alpha$ may not belong to $CZO_\nu$, but we still conclude many boundedness results for it.  In particular, in when $\nu=0$ and $b\in \dot B_\infty^{0,\infty}\backslash BMO$, the $T1$ theorem implies that $\Pi_b^0$ is not $L^2$-bounded.  However, we still conclude boundedness results for $\Pi_b^0$ on negative smoothness spaces, just not on any Lebesgue spaces.  It is likely, by analogy, that $b\in \dot B_\infty^{|\alpha|-\nu,\infty}\backslash I_{|\alpha|-\nu}(BMO)$ implies that $\Pi_b^\alpha$ is not bounded from $\dot W^{\nu,2}$ into $L^2$ (and hence $\Pi_b^\alpha$ would not belong to $CZO_\nu$), but we don't pursue this property here.
\end{remark}

\subsection{Smooth and Oscillating Operator Decompositions}\label{Sect7.4}

In this application, we decompose a singular integral operator $T$ into a sum of two terms $S+O$, one that preserves smoothness and one that preserves oscillatory properties of the input function.  We achieve this by constructing several paraproducts which satisfy $\Pi^*(x^\alpha)=0$ for all $\alpha$.  A sum of such operators define $O$, and hence $O$ enjoys all of the oscillatory preserving properties associated to the cancellation conditions of the form $O^*(x^\alpha)=0$.  Furthermore, $S$ will be constructed so that $S(x^\alpha)=0$ for appropriate multi-indices $\alpha$, which is sufficient for $S$ to be bounded on many smooth function spaces.

To motivate this, let's consider for a moment an operator $T\in SIO_\nu(\infty)$ of convolution type.  That is, assume there is a distribution kernel $k\in\S'(\R^n)$ such that $Tf(x)=\<k,f(x-\cdot)\>$ for $f\in \S_P$ for some $P\in\N_0^n$.  Such an operator preserves both regularity and oscillation since convolution operators commute.  For instance, suppose that $T$ is bounded from $X$ into $Y$, where $X,Y\subset\S'/\mathcal P$ are Banach spaces.  It follows that $T$ is bounded from $I_s(X)$ into $I_s(Y)$ for all $s\in\R$, where $I_s(X)=\{|\nabla|^{s}f:f\in X\}$ with the natural norm $\|f\|_{I_s(X)}=\|\,|\nabla|^{s}f\|_X$.  This is because
$$\|Tf\|_{I_s(X)}=\|\,|\nabla|^{s}(Tf)\|_{X}=\|T(|\nabla|^{s}f)\|_{X}\less\|\,|\nabla|^{s}f\|_Y=\|f\|_{I_s(Y)}.$$
When $s>0$, this says that if $f$ has $s$-order derivatives in $X$, then $Tf$ has $s$-order derivatives in $Y$.  For $s<0$, it says that $f$ has $s$-order anti-derivatives in $X$, then $Tf$ has $s$-order anti-derivatives in $Y$, which in many situations quantify oscillatory properties of $f$ and $Tf$.  Hence convolution operators simultaneously preserve both regularity and oscillatory properties of its input functions.  This cannot be expected for non-convolution operators, but the main result of this section formulates a decomposition that extends this principle to non-convolution operators in some sense.  We show, roughly, that for any operator $T\in SIO_\nu$, we can decompose $T=S+O$, where $S$ preserves smoothness and $O$ preserves oscillation.  This is achieved by constructing $S$ and $O$ so that $S|\nabla|^{s}\approx |\nabla|^{s}S$ and $O|\nabla|^{-s}\approx |\nabla|^{-s}O$ for $s>0$, in the appropriate sense, so that $S$ and $O$ each behave like a convolution operator on one side.  Based on our operator calculus from Section \ref{Sect4}, to construct $S$ and $O$ in this way, it is sufficient to make sure that $S(x^\alpha)=0$ and $O^*(x^\alpha)=0$ for appropriate $\alpha$.  This is made precise below.

We will have to use the non-convolutional moment for singular integral operators, which were defined for $SIO_0$ in \cite{HartLu1,HO}.  For $T\in SIO_\nu(M+\gamma)$, $\alpha\in\N_0^n$ with $|\alpha|\leq \nu+M$, define $[[T]]_\alpha\in\mathcal D_{2P}$ by
\begin{align*}
\<[[T]]_\alpha,\psi\>&=\lim_{R\rightarrow\infty}\int_{\R^{2n}}\mathcal K(x,y)(x-y)^\alpha \eta_R(y)\psi(x)dy\,dx.
\end{align*}
Following the same argument used to justify the definition of $T(x^\alpha)$, we can define $[[T]]_\alpha\in\mathcal D_{2P}'$ for the same ranges of indices.  See also \cite{HartLu1,HO} for more information on this construction.

\begin{lemma}\label{l:TxEquiv}
Let $T\in SIO_\nu(M+\gamma)$ and $L\leq \nu+M$.  Then $T(x^\alpha)=0$ for all $|\alpha|\leq L$ if and only if $[[T]]_\alpha=0$ for all $|\alpha|\leq L$.
\end{lemma}

\begin{proof}
This proof follows immediately from the following formula, which is just expanding the polynomial $(x-y)^\alpha$.  Let $|\alpha|\leq L$, $\eta_R\in\mathcal D_{2P}$ and $\psi\in\mathcal D_P$, for $P\in\N_0$ sufficiently larger, be as in the definition of $[[T]]_\alpha$, and we have
\begin{align*}
\<[[T]]_\alpha,\psi\>
&=\lim_{R\rightarrow\infty}\sum_{\beta+\mu=\alpha}c_{\beta,\mu}\int_{\R^{2n}}\mathcal K(x,y) y^\mu\eta_R(y)x^\beta\psi(x)dy\,dx=\sum_{\beta+\mu=\alpha}c_{\beta,\mu}\<T(x^\mu),x^\beta\psi\>.
\end{align*}
It immediately follows that $T(x^\alpha)$ vanishes for all $|\alpha|\leq L$ if and only if $[[T]]_\alpha$ does.
\end{proof}

\begin{lemma}\label{l:paraproducts}
Let $\nu\in\R$, $\alpha\in\N_0^n$, $b\in \dot B_\infty^{|\alpha|-\nu,\infty}$, and $\Pi_b^\alpha$ be defined as in \eqref{Pib}.  Then $[[\Pi_b^\alpha]]_\beta\in \dot B_\infty^{|\beta|-\nu,\infty}$ for all $\beta\in\N_0^n$, $\Pi_b^\alpha(x^\beta)=0$ for all $|\beta|\leq|\alpha|$ with $\beta\neq\alpha$, and $[[\Pi_b^\alpha]]_\alpha=b$.
\end{lemma}

\begin{proof}
For any $\beta\in\N_0^n$ and $\psi\in\mathcal D_P$ with $P$ sufficiently large, we have
\begin{align*}
|\psi_j*[[\Pi_b^\alpha]]_\beta(x)|&=\lim_{R\rightarrow\infty}\left|\int_{\R^{3n}}\pi_b^\alpha(u,y)(u-y)^\beta\eta_R(y)\psi_j^x(u)dy\,du\right|\\
&\leq\limsup_{R\rightarrow\infty}\sum_{\rho+\mu=\beta}c_{\rho,\mu}\sum_{k\in\Z}2^{k|\alpha|}\left|\int_{\R^{3n}}(D^\alpha\varphi)_k(v-y)(v-y)^\mu\eta_R(y)\right.\\
&\hspace{5.5cm}\left.\phantom{\int}\times Q_kb(v)\widetilde\psi_k(u-v)(u-v)^\rho  \psi_j^x(u)dv\,du\,dy\right|\\
&\leq\limsup_{R\rightarrow\infty}\sum_{\rho+\mu=\beta}c_{\rho,\mu}\sum_{k\in\Z}2^{k(|\alpha|-|\beta|)}\int_{\R^n}\(\int_{\R^n}|(D^\alpha\varphi)_k(y)(2^ky)^\mu\eta_R(y)|dy\)\\
&\hspace{5.5cm}\phantom{\int}\times |Q_kb(v)|\,|\widetilde\psi_k^\rho*\psi_j^x(v)|dv\\
&\less\|b\|_{\dot B_\infty^{|\alpha|-\nu,\infty}}\sum_{k\in\Z}2^{-K|j-k|}2^{k(\nu-|\beta|)}\less2^{(\nu-|\beta|)j}\|b\|_{\dot B_\infty^{|\alpha|-\nu,\infty}},
\end{align*}
where $\widetilde\psi_k^\rho(x)=\widetilde\psi_k(x)x^\rho$.  Then $[[\Pi_b^\alpha]]_\beta\in\dot B_\infty^{|\beta|-\nu,\infty}$ for all $\beta\in\N_0^n$.  For $\psi\in\mathcal D_P$, with $P\in\N_0^n$ fixed sufficiently large, we have
\begin{align*}
\<[[\Pi_b^\alpha]]_\alpha,\psi\>&=\lim_{R\rightarrow\infty}\int_{\R^{2n}}\pi_b^\alpha(x,y)(x-y)^\alpha\eta_R(y)\psi(x)dy\,dx\\
&=\lim_{R\rightarrow\infty}\frac{(-1)^{|\alpha|}}{\alpha!}\sum_{k\in\Z}2^{k|\alpha|}\int_{\R^{2n}}\(\int_{\R^n}(D^\alpha\varphi)_k(u-y)(x-y)^\alpha\eta_R(y)dy\)\\
&\hspace{5.5cm}\times\widetilde\psi_k(x-u)Q_kb(u) \psi(x)du\,dx\\
&=\sum_{k\in\Z}\int_{\R^{2n}}\widetilde\psi_k(x-u)Q_kb(u) \psi(x)du\,dx=\<b,\psi\>.
\end{align*}
Similarly, for $\psi\in\mathcal D_{P}$ with $P$ sufficiently large and $|\beta|\leq|\alpha|$ such that $\alpha\neq \beta$
\begin{align*}
\<\Pi_b^\alpha(x^\beta),\psi\>&=\lim_{R\rightarrow\infty}\frac{1}{\alpha!}\sum_{k\in\Z}\int_{\R^n}Q_kb(x) \(\int_{\R^n}\varphi_k(x-y) D^\alpha\(\eta_R(y) y^\beta\) dy\)\widetilde Q_k\psi(x)dx\\
&\hspace{0cm}=\frac{1}{\alpha!}\sum_{k\in\Z}\int_{\R^n}Q_kb(x)\(\int_{\R^n}\varphi_k(x-y)D^\alpha(y^\beta) dy\)\widetilde Q_k\psi(x)dx=0.
\end{align*}
Note that $|\beta|\leq|\alpha|$ and $\beta\neq\alpha$ implies that $D^\alpha(y^\beta)=0$.
\end{proof}

\begin{theorem}\label{t:Smooth+Oscillating}
Let $T\in SIO_\nu(M+\gamma)$ for some $M\in\N_0$ and $0<\gamma\leq1$.  If $[[T]]_\alpha\in \dot B_\infty^{|\alpha|-\nu,\infty}$ for all $|\alpha|\leq M$.  Then there exist operators $S\in SIO_\nu(M+\gamma)$ and $O\in SIO_\nu(\infty)$ such that $T=S+O$, $S(x^\alpha)=0$ for $|\alpha|\leq M$, and $O^*(x^\alpha)=0$ for $\alpha\in\N_0^n$.  Furthermore, if $T$ satisfies $WBP_\nu$, then both $S$ and $O$ satisfy $WBP_\nu$, in which case Theorem \ref{t:calculus}, Theorem \ref{t:BesovBounds}, Theorem \ref{t:TriebelLizorkinBounds}, Corollary \ref{c:dual1}, and Corollary \ref{c:dual2} can be applied to $S^*$ and $O$.
\end{theorem}

Note that the results from Sections \ref{Sect4} and \ref{Sect5} can be applied to $O$ regardless of whether $T$ satisfies $WBP_\nu$ since $O$ is a sum of paraproducts of the form \eqref{Pib} and by Corollary \ref{c:paraproducts} the results can be applied to each of these paraproducts.

\begin{proof}[of Theorem \ref{t:Smooth+Oscillating}]
Let $T$ be as above.  Define $b_0=[[T]]_0=T(1)\in\dot B_\infty^{-\nu,\infty}$.  For $1\leq |\alpha|\leq M$, define
\begin{align*}
b_\alpha=[[T]]_\alpha-\sum_{|\beta|<|\alpha|}[[\Pi_{b_\beta}^\beta]]_\alpha\in \dot B_\infty^{|\alpha|-\nu,\infty}.
\end{align*}
Also define
\begin{align*}
S=T-\sum_{|\alpha|\leq M}\Pi_{b_\alpha}^\alpha\;\;\;\text{and}\;\;\;O=\sum_{|\alpha|\leq M}\Pi_{b_\alpha}^\alpha.
\end{align*}
It immediately follows that $S\in SIO_\nu(M+\gamma)$ and $O\in SIO_\nu(\infty)$.  Using Lemmas \ref{l:TxEquiv} and \ref{l:paraproducts}, we also have
\begin{align*}
[[S]]_0&=[[T]]_0-\sum_{|\alpha|\leq M}[[\Pi_{b_\alpha}^\alpha]]_0=[[T]]_0-[[\Pi_{b_0}^0]]_0=0,
\end{align*}
and for $0<|\alpha|\leq M$ we have
\begin{align*}
[[S]]_\alpha&=[[T]]_\alpha-\sum_{|\beta|\leq M}[[\Pi_{b_\beta}^\beta]]_\alpha
=[[T]]_\alpha-b_\alpha-\sum_{|\beta|< |\alpha|}[[\Pi_{b_\beta}^\beta]]_\alpha=0.
\end{align*}
By Lemma \ref{l:TxEquiv}, it follows that $S(x^\alpha)=0$ for all $|\alpha|\leq M$.  It also follows from Corollary \ref{c:paraproducts} that $O^*(x^\alpha)=0$ for all $\alpha\in\N_0^n$.   Note also that $O$ always satisfies the $WBP_\nu$, and if $T$ satisfies $WBP_\nu$, then so does $S$.
\end{proof}

\begin{remark}
It may seem a little strange that we use the non-convolution moments $[[T]]_\alpha$, rather than $T(x^\alpha)$, in Lemma \ref{l:paraproducts} and Theorem \ref{t:Smooth+Oscillating}.  By Lemma \ref{l:paraproducts}, when using vanishing moment conditions the two are equivalent.  However, there is a crucial difference when the moments are not required to vanish, but instead some other conditions as we do in Theorem \ref{t:Smooth+Oscillating}.  This difference manifests in our setting when computing $[[\Pi_b^\alpha]]_\beta$ versus $\Pi_b^\alpha(x^\beta)$ for $|\beta|>|\alpha|$.  In Lemma \ref{l:paraproducts}, we showed that $[[\Pi_b^\alpha]]_\beta\in\dot B_\infty^{|\beta|-\nu,\infty}$, which is the natural condition to expect in this setting.  However, it may not be the case that $\Pi_b^\alpha(x^\beta)\in\dot B_\infty^{|\beta|-\nu,\infty}$.  To demonstrate this, let the dimension $n=1$, $\nu=0$, $\alpha=0$, $b\in \dot B_\infty^{0,\infty}$, and $\beta=1$.  Then for $x\in\R$
\begin{align*}
|\psi_j*\Pi_b^0(x^\beta)(x)|&=\lim_{R\rightarrow\infty}\left|\sum_{k\in\Z}\int_{\R^{2}}\varphi_k(u-y)y\eta_R(y)Q_kb(u)\widetilde\psi_k *\psi_j^x(u)du\,dy\right|\\
&\geq|x|\left|\sum_{k\in\Z}\int_{\R}Q_kb(u)\widetilde\psi_k *\psi_j^x(u)du\right|-|\psi_j*[[\Pi_b^0]]_1(x)|\\
&\geq|x|\,|\psi_j*b(x)|-|\psi_j*[[\Pi_b^0]]_1(x)|.
\end{align*}
Taking $j=0$, we have
\begin{align*}
\|\Pi_b^0(x)\|_{\dot B_\infty^{1,\infty}}&\geq|x|\,|\psi*b(x)|-\|[[\Pi_b^0]]_1\|_{\dot B_\infty^{1,\infty}}.
\end{align*}
By Lemma \ref{l:paraproducts}, $\|[[\Pi_b^0]]_1\|_{\dot B_\infty^{1,\infty}}<\infty$, and it is not hard to construct $b\in \dot B_\infty^{0,\infty}$ such that $|x|\,|\psi*b(x)|$ is unbounded (for example, $b(x)=\sin(x)$ would do).  Hence $\Pi_b^0(x)\notin\dot B_\infty^{1,\infty}$ for such $b$.  Similar constructions can be done in any dimension and for $b\in \dot B_\infty^{|\alpha|-\nu,\infty}$ to produce the property $\Pi_b^\alpha(x^\beta)\notin\dot B_\infty^{|\beta|-\nu}$
\end{remark}

\subsection{Sparse Domination}\label{Sect7.5}

There has been a lot of interest lately in sparse domination results for Calder\'on-Zygmund operators.  The first such result is due to Lerner \cite{L}, but there have been many extensions and improvements; see for example \cite{Moen,CR,BFP,CDO,Lacey,LA,LS,Z-K}.  However, there do not appear to be any results that apply to regularity estimates for operators or to hyper-singular operators.  There are some sparse estimates for fractional integral operators, for example in \cite{Moen}.

We will apply the sparse domination result from \cite{CR}, which can be summarized as follows.  For a collection of dyadic cubes $S$, define the dyadic operator
\begin{align*}
\mathcal A_Sf(x)=\sum_{Q\in S}\<f\>_Q\chi_Q(x),
\end{align*}
where $\<f\>_Q=\frac{1}{|Q|}\int_Qf(x)dx$.  The collection of dyadic cubes belonging to the same dyadic grid $S$ is sparse if for every $Q\in S$ there exists a measurable subset $E(Q)\subset Q$ with $|E(Q)|>|Q|/2$ and $E(Q)\cap E(Q')=\emptyset$ for ever $Q'\in S$ with $Q'\subsetneq Q$.  They prove that if $T\in CZO_0$ and $f$ is an integrable function supported in a cube $Q_0$, then there exist sparse dyadic collections $S_1,...,S_{3^n}$ (possibly associated to different dyadic grids) such that
\begin{align}
|Tf(x)|\less\sum_{i=1}^{3^n}\mathcal A_{S_i}(|f|)(x)
\end{align}
almost everywhere on $Q_0$.  We use this result to prove our next result.

\begin{corollary}\label{c:sparse}
Let $\nu\in\R$, $L$ be an integer with $L\geq|\nu|$, $M\geq\max(L,L-\nu)$, $(L-\nu)_*<\gamma\leq1$, and $T\in SIO_\nu(M+\gamma)$ satisfies $WBP_\nu$.
\begin{itemize}
\item Assume that $T^*(x^\alpha)=0$ for all $|\alpha|\leq L$.  Then for any $1<p<\infty$, $0<t<\widetilde L+\gamma$, cube $Q_0\subset\R^n$, and $f\in \dot W^{-t,p}$ with $\supp(|\nabla|^{-t}f)\subset Q_0$, there is a polynomial $P_f$ and sparse collections of dyadic cubes $S_1,...,S_{3^n}$ such that
\begin{align}\label{sparse1}
|\,|\nabla|^{-(\nu+t)}Tf(x)-P_f(x)|\less\sum_{i=1}^{3^n}\mathcal A_{S_i}(|\,|\nabla|^{-t}f|)(x)\;\;\;\text{a.e. }x\in Q_0.
\end{align}

\item Assume that $T(x^\alpha)=0$ for all $|\alpha|\leq L$.  Then for any $1<p<\infty$, $0<t<\widetilde L+\gamma$, cube $Q_0\subset\R^n$, and $f\in \dot W^{\nu +t,p}$ with $\supp(|\nabla|^{\nu+t}f)\subset Q_0$, there is a polynomial $P_f$ and sparse collections of dyadic cubes $S_1,...,S_{3^n}$ such that
\begin{align}\label{sparse2}
|\,|\nabla|^tTf(x)-P_f(x)|\less\sum_{i=1}^{3^n}\mathcal A_{S_i}(|\,|\nabla|^{\nu+t}f|)(x)\;\;\;\text{a.e. }x\in Q_0.
\end{align}
\end{itemize}
\end{corollary}

\begin{proof}
Let $\nu\in\R$, $L$ be an integer with $L\geq|\nu|$, $M\geq\max(L,L-\nu)$, $(L-\nu)_*<\gamma\leq1$, and $T\in SIO_\nu(M+\gamma)$ satisfy $WBP_\nu$.  Assume that $T^*(x^\alpha)=0$ for all $|\alpha|\leq L$.  By Theorem \ref{t:calculus}, for any $0<t<L+\gamma$ there exists $T_t\in CZO_0$ such that $P_f=|\nabla|^{-(\nu+t)}T|\nabla|^tTf-T_tf$ is a polynomial for all $f\in \S_\infty$.  Note that $T_t\in CZO_0$ implies that $T_t$ is already well-defined on $L^p$ for $1<p<\infty$.  Hence $T_t(|\nabla|^{-t}f)$ is well-defined and belongs to $L^p$ for any $f\in \dot W^{-t,p}$.  It also follows that $|\nabla|^{-(\nu+t)}Tf=T_t(|\nabla|^{-t}f)$ for all $f\in\S_\infty$, and $\|\,|\nabla|^{-(\nu+t)}Tf-P_f\|_{L^p}=\|T_t(|\nabla|^{-t}f)\|_{L^p}\less\|f\|_{\dot W^{-t,p}}$.  Then $|\nabla|^{-(\nu+t)}Tf$ can be defined pointwise for every $f\in \dot W^{-t,p}$ with $1<p<\infty$ via the equation $|\nabla|^{-(\nu+t)}Tf=T_t(|\nabla|^{-t}f)+P_f$.  Now it is just a matter of applying the pointwise sparse operator estimate from \cite{CR}.  Let $f\in \dot W^{-t,p}$ with $\supp(|\nabla|^{-t}f)\subset Q_0$, and since $T_t\in CZO_0$ it follows that
\begin{align*}
|\,|\nabla|^{-(\nu+t)}Tf(x)-P_f(x)|&=|\,T_t(|\nabla|^{-t}f)(x)|\less\sum_{i=1}^{3^n}\mathcal A_{S_i}(|\,|\nabla|^{-t}f|)(x).
\end{align*}
The estimate in \eqref{sparse2} can be proved in the same way.
\end{proof}

\subsection{Operator Calculus}\label{Sect7.6}

Throughout this article, we have been working with a restricted operator calculus, where we only consider compositions of the form $|\nabla|^{-s}T|\nabla|^t$.  In this application, we construct a true operator calculus (or operator algebra) made up of singular integrals with different singularities.  In order to make our notation and computations a little simpler here, we will only work with operators in $SIO_\nu(\infty)$ that satisfy $T(x^\alpha)=T^*(x^\alpha)=0$ for all $\alpha\in\N_0^n$.  These assumption are necessary for some of the algebras we construct, but not for all.  Before we continue, we will need some additional information about the operators we defined in Corollary \ref{c:ao}, which we provide in the next lemma.

\begin{lemma}\label{l:lambdamoment}
Let $\nu\in\R$ and $T\in SIO_\nu(\infty)$.  Assume that $T(x^\alpha)=T^*(x^\alpha)=0$ for all $\alpha\in\N_0^n$ and $T\in WBP_\nu$.  Fix $L\in\N_0$, $\psi,\tilde\psi\in\mathcal D_P$ for $P$ sufficiently large, and define $\lambda_{j,k}(x,y)=\<T^*\psi_j^x,\tilde\psi_k^y\>$ for $j,k\in\Z$ and $x,y\in\R^n$.  Then
\begin{align}
\int_{\R^n}\lambda_{j,k}(x,y)x^\alpha dx=\int_{\R^n}\lambda_{j,k}(x,y)y^\alpha dx=0\label{lambdamoment}
\end{align}
for all $|\alpha|\leq L$.
\end{lemma}

\begin{proof}
Let $L\in\N_0$, $|\alpha|\leq L$, and $\eta\in \mathcal D_P$ with $\eta=1$ on $B(0,1)$ and $\eta_R(x)=\eta(x/R)$, where $P$ is chosen sufficiently large.  By Theorem \ref{t:truncation}, we have
\begin{align*}
\int_{\R^n}\lambda_{j,k}(x,y)y^\alpha dy&=\int_{\R^n}\<T\psi_j^x,\psi_k^y\>y^\alpha dy=\lim_{R\rightarrow\infty}\<T\psi_j^x,F_{\alpha,R,k}\>,
\end{align*}
where
\begin{align*}
F_{\alpha,R,k}(u)=\int_{|y|<R}\psi_k(y-u)y^\alpha dy.
\end{align*}
It follows that $\supp(F_{\alpha,R,k})\subset B(0,R+2^{k+1})\backslash B(0,R-2^{k+1})$ and $F_{\alpha,R,k}\in\mathcal D_P$.  Then as long as $R>4|x|+2^{k+4}+2^{j+4}$, it follows that
\begin{align*}
|\<T\psi_j^x,F_{\alpha,R,k}\>|
&=\left|\int_{\R^{2n}}(K(u,v)-J_x^M[K(\cdot,v)](u))\psi_j^x(u)F_{\alpha,R,k}(v)du\,dv\right|\\
&\less\int_{\R^{2n}}\frac{|x-u|^{M+\gamma}}{|u-v|^{n+\nu+M+\gamma}}|\psi_j^x(u)F_{\alpha,R,k}(v)|du\,dv\\
&\less\frac{2^{(M+\gamma)j}}{R^{n+\nu+M+\gamma}}\int_{|y|<R}\(\int_{\R^{2n}}|\psi_j^x(u)\psi_k(v) |du\,dv\)|y^\alpha|dy\\
&\less\frac{2^{(M+\gamma)j}}{R^{n+\nu+M+\gamma}}R^{n+|\alpha|}=2^{(M+\gamma)j}R^{-(\nu+M+\gamma-|\alpha|)},
\end{align*}
which tends to zero as $R\rightarrow\infty$.  Here we take $M>L+|\nu|+\gamma$ and $P>M$; note that $P$ then depends on $L$, and so we cannot completely remove the restriction $|\alpha|\leq L$ in the statement of Lemma \ref{l:lambdamoment}.  Therefore the first integral condition in \eqref{lambdamoment} holds for $|\alpha|\leq L$, and by symmetry the same holds for the second one.
\end{proof}

\begin{theorem}\label{t:operatoralgebra}
Let $V\subset\R$ be a set that is closed under addition.  Then the collection of operators
\begin{align*}
\mathfrak A_V=\{T\in SIO_\nu(\infty):\nu\in V,\;T(x^\alpha)=T^*(x^\alpha)=0,\,T\in WBP_\nu\}
\end{align*}
is an operator algebra in the sense that it is closed under composition and transpose.
\end{theorem}

\begin{proof}
Fix two real numbers $\nu_1,\nu_2\in V$.  In order to show $S,T \in\mathfrak A_V$ implies $S\circ T\in SIO_{\nu_1+\nu_2}(\infty)$ where $T\in SIO_{\nu_1}$ and $S\in SIO_{\nu_2}$, we must first show that $S\circ T$ and $(S\circ T)^*$ are defined as (or can be extended to) operators from $\S_P$ into $\S'$ for some $P\in\N_0$ sufficiently large.  We first note that the vanishing moment and weak boundedness properties of operators in $\mathfrak A_V$ imply that all members of $\mathfrak A_V\cap SIO_\nu$ are bounded from $\dot W^{\nu+s,p}$ into $\dot W^{s,p}$ for all $s\in\R$ such that $s<-\nu$ and for $s>0$ by Theorem \ref{t:TriebelLizorkinBounds} and Corollary \ref{c:dual1}, respectively.  For $f\in\S_P$ and $g\in \S$ (with $P\geq|\nu_1|+|\nu_2|$), we note that for any $1<p<\infty$
\begin{align*}
|\<S\circ Tf,g\>|&=|\<Tf,S^*g\>|\leq\|Tf\|_{\dot W^{-\mu,p}}\|S^*g\|_{\dot W^{\mu,p'}}\\
&\leq\|T\|_{\dot W^{\nu_1-\mu,p}\to\dot W^{-\mu,p}}\|S^*\|_{\dot W^{\mu,p'}\to\dot W^{\nu_2+\mu,p'}}\|f\|_{\dot W^{\nu_1-\mu,p}}\|g\|_{\dot W^{\nu_2+\mu,p'}},
\end{align*}
where $\mu>\max(\nu_1,-\nu_2)$.  Note that $T$ is bounded from $\dot W^{\nu_1-\mu,p}$ into $\dot W^{-\mu,p}$ since $-\mu<-\nu_1$, and $S^*$ is bounded from $\dot W^{\mu,p'}$ into $\dot W^{\nu_2+\mu,p'}$ since $\nu_2+\mu>0$.  Since $\S_P$ and $\S$ embed continuously into $\dot W^{-\mu,p}$ and $\dot W^{\nu_2+\mu,p'}$, respectively, it follows that $S\circ T$ is continuous from $\S_P$ into $\S'$.  By symmetry, it follows that $(S\circ T)^*$ is also continuous from $\S_P$ into $\S'$.  Furthermore, this inequality, and a similar one for $(S\circ T)^*$, imply that $S\circ T$ satisfies the $WBP_{\nu_1+\nu_2}$

Next we show that $S\circ T$ has a standard kernel.  Let $\psi$ and $\widetilde\psi$ be as in Lemma \ref{l:calderon}, $Q_kf=\psi_k*f$, and $\widetilde Q_kf=\widetilde\psi_k*f$.  For $f,g\in\mathcal D_P$, we have
\begin{align*}
\<S\circ Tf,g\>&=\sum_{j,k,\ell,m\in\Z}\<\widetilde Q_\ell Q_\ell TQ_m\widetilde Q_mf,\widetilde Q_kQ_kS^*Q_j\widetilde Q_jg\>\\
&=\sum_{j,k,\ell,m\in\Z}\int_{\R^{2n}}\omega_{j,k,\ell,m}(x,y)f(y)g(x)dy\,dx,
\end{align*}
where
\begin{align*}
\omega_{j,k,\ell,m}(x,y)=\int_{\R^{2n}}\widetilde \psi_\ell(u-w)\widetilde\omega_{j,k,\ell,m}(w,\xi)\widetilde \psi_j(\xi-x)dz\,d\xi\,dv\,dw\,du
\end{align*}
and
\begin{align*}
\widetilde\omega_{j,k,\ell,m}(w,\xi)=\int_{\R^{3n}}\lambda_{\ell,m}^T(w,z)\widetilde \psi_m(z-y)\widetilde \psi_k(u-v)\lambda_{k,j}^{S^*}(v,\xi)dz\,dv\,du.
\end{align*}
For any fixed $K,N\geq0$, it follows from Corollary \ref{c:ao} and Lemma \ref{l:lambdamoment}, as well as similar arguments to those in the proof of Corollary \ref{c:ao}, that
\begin{align*}
&\left|\widetilde\omega_{j,k,\ell,m}(w,\xi)\right|\less2^{\nu_1\min(\ell,m)+\nu_2\min(j,k)}2^{-K\max(|j-k|,|k-\ell|,|\ell-m|)}\Phi_{\min(j,k,\ell,m)}^N(w-\xi).
\end{align*}
From this bound, it easily follows that
\begin{align*}
|\omega_{j,k,\ell,m}(x,y)|\less2^{\nu_1\min(\ell,m)+\nu_2\min(j,k)}2^{-K\max(|j-k|,|k-\ell|,|\ell-m|)}\Phi_{\min(j,k,\ell,m)}^N(x-y).
\end{align*}
Then for any $\alpha,\beta\in\N_0^n$ and $x\neq y$, we have
\begin{align*}
&\sum_{j,k,\ell,m\in\Z}|D^\alpha_yD^\beta_x\omega_{j,k,\ell,m}(x,y)|\\
&\hspace{2cm}\less\sum_{j,k,\ell,m\in\Z}2^{(\nu_1+|\alpha|)m+(\nu_2+|\beta|)j}2^{-K\max(|j-k|,|k-\ell|,|\ell-m|)}\Phi_{\min(j,k,\ell,m)}^N(x-y)\\
&\hspace{2cm}\less\sum_{j,k,\ell,m\in\Z}2^{(\nu_1+\nu_2+|\alpha|+|\beta|)\min(j,k,\ell,m)}2^{-\tilde K\max(|j-k|,|k-\ell|,|\ell-m|)}\Phi_{\min(j,k,\ell,m)}^N(x-y)\\
&\hspace{2cm}\less\sum_{j\in\Z}2^{(\nu_1+\nu_2+|\alpha|+|\beta|)j}\Phi_j^N(x-y)\less\frac{1}{|x-y|^{n+\nu_1+\nu_2+|\alpha|+|\beta|}}.
\end{align*}
These sums converges as long as $K>4(|\nu_1|+|\nu_2|+|\alpha|+|\beta|)$ and $N>n+\nu_1+\nu_2+|\alpha|+|\beta|$.  It follows that
\begin{align*}
K_{S\circ T}(x,y)=\sum_{j,k,\ell,m\in\Z}\omega_{j,k,\ell,m}(x,y)
\end{align*}
is the kernel of $S\circ T$ and is a $(\nu_1+\nu_2)$-order standard kernel that is $C^\infty$ off of the diagonal.  Therefore $S\circ T\in SIO_{\nu_1+\nu_2}(\infty)$.  Using the estimate for $\omega_{j,k,\ell,m}$ above, for $f,g\in\S_\infty$ and $t\in\R$ we have
\begin{align*}
&|\<S\circ Tf,g\>|\leq\sum_{j,k,\ell,m\in\Z}2^{t(m-j)}\left|\<\widetilde Q_\ell Q_\ell TQ_m\widetilde Q_m^t(|\nabla|^{-t}f),\widetilde Q_kQ_kS^*Q_j\widetilde Q_j^{-t}(|\nabla|^tg)\>\right|\\
&\hspace{1cm}\less\sum_{j,k,\ell,m\in\Z}2^{t(m-j)}\int_{\R^{2n}}\left|\widetilde\omega_{j,k,\ell,m}(x,y)\widetilde Q_m^t(|\nabla|^{-t}f)(y)\widetilde Q_j^{-t}(|\nabla|^tg)(x)\right|dy\,dx\\
&\hspace{1cm}\less\int_{\R^{n}}\sum_{j,k,\ell,m\in\Z}2^{(\nu_1+\nu_2)m}2^{-\tilde K\max(|j-k|,|k-\ell|,|\ell-m|)}\mathcal M(\widetilde Q_m^t(|\nabla|^{-t}f))(x)|\widetilde Q_j^{-t}(|\nabla|^tg)(x)|dx\\
&\hspace{1cm}\less\|f\|_{\dot W^{\nu_1+\nu_2+t,p}}\|g\|_{\dot W^{-t,p'}}.
\end{align*}
In this estimate, we fix a $P\in\N_0$ large enough depending on $t$.  Therefore $S\circ T$ can be extended to a bounded linear operator from $\dot W^{\nu_1+\nu_2+t,p}$ into $\dot W^{t,p}$ for all $t\in\R$ and $1<p<\infty$.  Taking $t=0$, we see that $S\circ T\in CZO_{\nu_1+\nu_2}(\infty)$, and we can apply Theorem \ref{t:T1sufficient}, which implies that $(S\circ T)^*(x^\alpha)=0$ for all $\alpha\in\N_0^n$.  By symmetry $S\circ T(x^\alpha)=0$ for all $\alpha\in\N_0^n$ as well.  Also $\nu_1+\nu_2\in V$, and so $S\circ T\in \mathfrak A_V$ which verifies that $\mathfrak A_V$ closed under composition as long as $V$ is closed under addition.  Since $SIO_\nu(\infty)$, $WBP_\nu$, and the condition $T(x^\alpha)=T^*(x^\alpha)=0$ are all symmetric under transposition, it is obvious that $\mathfrak A_V$ is closed under transposes too.  Therefore $\mathfrak A_V$ is an algebra that is closed under composition and transposes.
\end{proof}

\begin{remark}
Theorem \ref{t:operatoralgebra} defines many operator algebras for different classes of singular integral operators.  If one takes $V=\{0\}$, then $\mathfrak A_{\{0\}}$ is a set of Calder\'on-Zygmund operators and one of the operator algebras discussed by Coifman and Meyer in \cite{CMbook}.  Some other interesting examples of algebras $\mathfrak A_V$ can be constructed by taking $V$ to be, for example, $V=\R$ or $V=\{\mu\,\nu:\nu\in\Z\}$ for some fixed $\mu\in\R$.  One can also modify any of these examples by replacing $V$ with $V\cap (0,\infty)$, $V\cap[0,\infty)$, $V\cap (-\infty,0)$, or $V\cap(-\infty,0]$, which amounts to restricting an algebra to differential or fractional integral operators (strictly or including order-zero operators).  Furthermore, one can combine different algebras by defining $\mathfrak A_V=\mathfrak A_{V_1}+\mathfrak A_{V_2}$, where $V\subset\R$ is the smallest set that is closed under addition and contains $V_1\cup V_2$.  Using this convention, we can consider the set $V=\{\mu_1-\mu_2\pi:\mu_1,\mu_2\in\mathbb Q,\,\mu_1\geq0,\,\mu_2\geq0,\,\mu_1\cdot\mu_2\neq0\}\subset\R$, which is closed under addition.  Since $\pi$ is transcendental, it follows that $0\notin V$.  This generates a somewhat peculiar example of $\mathfrak A_V$ since with this selection of $V$, it follows that $\mathfrak A_V$ is an operator that has both derivative and fractional integral operators, but does not contain any Calder\'on-Zygmund operators (order zero operators).  Of course many other examples can be generated by selecting $V$ in different ways, and even further understand $\mathfrak A_V$ through the algebraic properties of $V$.

Note that the properties imposed on the operators in $\mathfrak A_V$ imply that if $\mathfrak A_V\cap SIO_\nu\subset CZO_\nu$.  Also note that any algebra $\mathfrak A_V$ contains all convolutions operators in $CZO_\nu(\infty)$ as long as $\nu\in V$.  That is
\begin{align*}
\bigcup_{\nu\in V}\{T\in CZO_\nu(\infty):T\text{ is a convolution}\}\subset\mathfrak A_V\subset\bigcup_{\nu\in V}CZO_\nu(\infty).
\end{align*}
Furthermore, it is not hard to see that both inclusions here are strict for non-empty sets $V\subset\R$.
\end{remark}

\section{acknowledgements}\label{ackref}
The authors would like to thank Diego Maldonado for posing some questions related to the second author about weighted operator estimates on Hardy spaces that led to some of the results in this work.  In particular, Proposition \ref{p:extrapolation} and item (5) of Theorem \ref{t:T1sufficient} were motivated by discussion with him.  We would also like to thank Rodolfo H. Torres for his valuable feedback on this article, in particular his advice pertaining to pseudodifferential operators of type $1,1$ that were discussed in Section \ref{Sect7.2}.


\begin{thebibliography}{99}

\bibitem{AM} {  J. Alvarez and M. Milman}, $H^p$ continuity properties of Calder\'on-Zygmund-type operators, {\em J. Math. Anal. Appl. }118 1 (1986) 63-79.

\bibitem{AJ} {  K. F. Andersen and R. T. John}, Weighted ineuqalities for vector-valued maximal functions and singular integrals, {\em Studia Math. }69 (1980/81) 19-31.

\bibitem{BMNT} {  \'A. B\'enyi, D. Maldonado, V. Naibo, and R. H. Torres}, On the H\"ormander classes of bilinear pseudodifferential operators, {\em Integr. Equ. Oper. Theory }67 (2010) 341-364.

\bibitem{BT} {  \'A. B\'enyi and R. H. Torres}, Symbolic calculus and the transposes of bilinear pseudodifferential operators, {\em Comm. Partial Differential Equations }28 5-6 (2003) 1161-1181.

\bibitem{BFP} {  F. Bernicot, D. Frey, and S. Petermichl}, Sharp weighted norm estimates beyond Calder\'on-Zygmund theory, {\em Anal. PDE }9 5 (2016) 1079-1113.

\bibitem{B} {  J. M. Bony}, Calcul symbolique et propagation des singulararit\'es pour les \'equrations aux d\'eriv\'ees partielles non lin\'eaires, {\em Ann. Sci. \'Ecole Norm. Sup. }14 2 (1981) 109-246.

\bibitem{Bourd} {  G. Bourdaud}, Une alg\`ebre maximale d'op\'erateurs pseudo-diff\'erentiels de type $1,1$, {\em S\'eminaire sur les \'Equations aux D\'eriv\'ees Partielles 1987-1988}, Exp. No. VII, 17 pp., \'Ecole Polytech., Palaiseau, 1988.

\bibitem{Bui} {  H.-Q. Bui}, Weighted Besov and Triebel spaces: Interpolation by the real method, {\em Hiroshima Math. J. }12 (1982) 581-605.

\bibitem{CR} {  J. Conde-Alonso and G. Rey}, A pointwise estimate for positive dyadic shifts and some applications, {\em Math. Ann. }365 3-4 (2016) 1111--1135.

\bibitem{CHO} {  L. Chaffee, J. Hart, and L. Oliviera}, Sobolev-$BMO$ and fractional integrals on super-critical ranges of Lebesgue spaces, {\em J. Funct. Anal. }272 (2017) 631-660.

\bibitem{CM1} {  R. R. Coifman and Y. Meyer}, Au del\`a des op\'erateurs pseudo-diff\'erentiels, {\em Ast\'erisque }57 (1978).

\bibitem{CMbook} {  R. R. Coifman and Y. Meyer}, Wavelets. Calder\'on-Zygmund and multilinear operators, {\em Cambridge Studies in Advanced Mathematics }48. Cambridge University Press, Cambridge, 1997.

\bibitem{CDO} {  A. Culiuc, F. Di Plinio, and Y. Ou}, Domination of multilinear singular integrals by positive sparse forms, {\em arXiv:1603.05317}.

\bibitem{DJ} {  G. David and  J. L. Journ\'e}, A boundedness criterion for generalized Calder\'on-Zygmund operators, {\em A. of Math. (2) }120 2 (1984) 371-397.

\bibitem{Duo} {  J. Duoandikoetxea}, Extrapolation of weight revisited: New proofs and sharp bounds, {\em J. Funt. Anal. }260 (2011) 1886--1901.

\bibitem{FHJW} {  M. Frazier, Y.S. Han, B. Jawerth, and G. Weiss}, The $T1$ theorem for Triebel-Lizorkin spaces, {\em Lecture Notes in Math. }1384, Springer, Berlin, 1989.

\bibitem{FJ2} {  M. Frazier and B. Jawerth}, A discrete transform and decompositions of distribution spaces, {\em J. Funct. Anal. }93 1 (1990) 34--170.

\bibitem{FTW} {  M. Frazier, R. H. Torres, and G. Weiss}, The boundedness of Calder\'on-Zygmund operators on the spaces $\dot F_p^{\alpha,q}$, {\em Rev. Mat. Iveramericana }4 1 (1988) 41--72.

\bibitem{GH} {  L. Grafakos and D. He}, Weak hardy spaces. Some topics in harmonic analysis and applications, {\em Adv. Lect. Math. (ALM) }34 177--202 Int Press, Somerville, MA, 2016.

\bibitem{HH}  {  Y.-S. Han and S. Hofmann}, $T1$ theorems for Besov and Triebel-Lizorkin spaces, {\em Trans. Amer. Math. Soc. }337 2 (1993) 839--853.

\bibitem{Hart2} {  J. Hart}, A new proof of the bilinear T(1) theorem, {\em Proc. Amer. Math. Soc. }142 9 (2014) 3169--3181.

\bibitem{HartLu1} {  J. Hart and G. Lu}, Hardy space estimates for Littlewood-Paley-Stein square functions and Calder\'on-Zygmund operators, {\em J. Fourier Anal. Appl. }22 1 (2016) 159--186.

\bibitem{HO} {  J. Hart and L. Oliveira}, Hardy space estimates for limited ranges of Muckenhoupt $A_p$ weights, {\em Adv. Math. }313 (2017) 803--838.

\bibitem{Horm1}  {  L. H\"ormander}, Pseudodifferential operators of type $1,1$, {\em Comm. Partial Diff. Eq. }13 (1988) 1085--1111.

\bibitem{Horm2}  {  L. H\"ormander}, Continuity of pseudodifferential operators of type $1,1$, {\em Comm. Partial Diff. Eq. }14 (1989) 231--243.

\bibitem{Jaw}  {  B. Jawerth}, Some observations on Besov and Lizorkin-Triebel spaces, {\em Math. Scand. }40 1 (1977) 94--104.

\bibitem{Lacey} {  M. T. Lacey}, Sparse Bounds for Spherical Maximal Functions, to appear in J. dÕAnalyse Math.

\bibitem{LA} {  M. T. Lacey and D. M. Arias}, The sparse T1 Theorem, {\em Houston J. Math. }43 1 (2017) 111--127.

\bibitem{LS} {  M. T. Lacey and S. Spencer}, Sparse bounds for oscillatory and random singular integrals, {\em New York J. Math. }23
(2017) 119--131.

\bibitem{L}  {  A. Lerner}, On an estimate of Calder\'on-Zygmund operators by dyadic positive operators, {\em J. Anal. Math. }121 (2013) 141--161.

\bibitem{LW}  {  C.-C. Lin and K. Wang}, Calder\'on-Zygmund operators acting on generalized Carleson measure spaces, {\em Studia Math. }211 3 (2012) 231--240.

\bibitem{LZ} {  G. Lu and Y. Zhu}, Bounds of singular integrals on weighted Hardy spaces and discrete Littlewood-Paley analysis, {\em J. Geom. Anal. }22 3 (2012) 666-684.

\bibitem{M1} {  Y. Meyer}, R\'egularit\'e des solutions des \'equations aux d\'eriv\'ees partielles non lin\'eares, in Seminaire Bourbaki, 1979-1980, 293-302. {\em Lecture notes in Math }842, Spring-Verlag, 1980.

\bibitem{M2} {  Y. Meyer}, Continuit\'e sur les espaces de H\"older et de Sobolev des op\'eratuers d\'finis par des int\'egrales singuli\`eres, {\em S\'eminaire \'Equations aux d\'eriv\'ees partielles (Polytechnique) }1 (1983-1984) 1--11.

\bibitem{Moen} {  K. Moen}, Sharp weighted bounds without testing or extrapolation, {\em Arch. Math. }99 (2012) 457--466.

\bibitem{NS} {  A. Nagel and E. Stein}, Lectures on pseudodifferential operators: Regularity theorems and applications to nonelliptic problems, {\em Mathematical Notes }24. Princeton University Press, Princeton, N.J.; University of Tokyo Press, Tokyo, 1979.

\bibitem{N}  {  U. Neri}, Fractional integration on the space $H^1$ and its dual, {\em Studia Math. }53 (1975) 175--189

\bibitem{R} {  T. Runst}, Pseudodifferential operators of the ``exotic'' class $L_{1,1}^0$ in spaces of Besov and Triebel-Lizorkin type, {\em Ann. Global Anal. Geom. }3 1 (1985) 13--28.

\bibitem{Seel} {  R. T. Seeley}, Elliptic singular integral equations, {\em 1967 Singular Integrals (Proc. Sympos. Pure Math., Chicago, Ill., 1966) }308--315 Amer. Math. Soc., Providence, R.I.

\bibitem{St3} {  E. Stein}, Harmonic analysis: real-variable methods, orthogonality, and oscillatory integrals, {\em Monographs in Harmonic Analysis }43 III. Princeton University Press, Princeton, NJ, 1993.

\bibitem{Str1} {  R. Strichartz}, Bounded mean oscillation and Sobolev spaces, {\em Indiana Univ. Math. J. }29 4 (1980) 539--558.

\bibitem{Str2} {  R. Strichartz}, Traces of BMO-Sobolev spaces, {\em Proc. Amer. Math. Soc. }83 3 (1981) 509--513.

\bibitem{T2} {  R. H. Torres}, Continuity properties of pseudodifferential operators of type $1,1$, {\em Comm. Partial Differential Equations }15  9 (1990) 1313--1328.

\bibitem{T} {  R. H. Torres}, Boundedness results for operators with singular kernels on distribution spaces, {\em Mem. Amer. Math. Soc. }90 (1991) 442.

\bibitem{Z-K} {  P.  Zorin-Kranich}, $A_p$-$A_\infty$ estimates for multilinear maximal and sparse operators, to appear in J. d'Analyse Math.

\bibitem{Z} {  A. Zygmund}, Smooth functions, {\em Duke Math. J. }12 (1945) 47--76.

\end{thebibliography}
\end{document}